\documentclass[preprint,nopreprintline,11pt]{elsarticle}
\usepackage[a4paper, left=2cm,right=2cm,top=2cm,bottom=2cm]{geometry}
\usepackage{mathptmx}
\usepackage{amssymb}
\usepackage{mathrsfs}
\usepackage{amsmath,empheq,fixmath}
\usepackage{nicefrac,xfrac}
\usepackage{multirow}
\usepackage{array}
\usepackage{amsthm}
\usepackage{booktabs}
\usepackage{csvsimple}

\usepackage{graphicx}
\usepackage{subfiles}

\usepackage[justification=centering]{caption}
\usepackage{subcaption}
\usepackage{accents}
\usepackage{enumerate}
\usepackage{hyperref}       % hyperlinks
\usepackage{cleveref}
\hypersetup{
	colorlinks=true,
	linkcolor={blue},
	citecolor={blue},
	urlcolor={blue},
	breaklinks=true,
	plainpages=true
}

\usepackage{calc}
\usepackage{siunitx}
\usepackage{tikz} % To generate the plot from csv
\usetikzlibrary{backgrounds}
\usetikzlibrary{shapes,arrows,positioning,intersections,quotes}
\usepackage{tcolorbox}
\usepackage{pgfplots}
\pgfplotsset{compat=newest,scaled y ticks=true} % Allows to place the legend below plot
\usepgfplotslibrary{units} % Allows to enter the units nicely
\usetikzlibrary{external}
\usepackage{listings}
\definecolor{codegreen}{rgb}{0,0.6,0}
\definecolor{codegray}{rgb}{0.5,0.5,0.5}
\definecolor{codepurple}{rgb}{0.58,0,0.82}
\definecolor{backcolour}{rgb}{0.95,0.95,0.92}
\lstdefinestyle{mystyle}{
	backgroundcolor=\color{backcolour},   
	commentstyle=\color{codegreen},
	keywordstyle=\color{magenta},
	stringstyle=\color{codepurple},
	basicstyle=\ttfamily\footnotesize,
	breakatwhitespace=false,         
	breaklines=true,                 
	captionpos=b,                    
	keepspaces=true,                  
	showspaces=false,                
	showstringspaces=false,
	showtabs=false,                  
	tabsize=1
}

\biboptions{sort&compress}
\allowdisplaybreaks

\theoremstyle{definition}

\newtheorem{theorem}{\textbf{Theorem}}[section]
\newtheorem{corollary}{Corollary}[theorem]
\newtheorem{lemma}[theorem]{\textbf{Lemma}}

\newtheorem{remark}[theorem]{\textbf{Remark}}
\newtheorem{assumption}{\textbf{Assumption}}[section]

\newcommand{\temporalInterval}{I_T}
\newcommand{\spDom}{\Omega} % spatial domain
\newcommand{\stDom}{U} % space-time domain

\newcommand{\xcomp}[1]{x_{#1}}
\newcommand{\xii}{\xcomp{i}}
\newcommand{\ucomp}[1]{u_{#1}}

\newcommand{\xvec}{\mvec{x}}
\newcommand{\xvecst}{\hat{\mvec{x}}}
\newcommand{\uc}{\bar{\uvec}}
\newcommand{\uf}{\uvec'}
\newcommand{\pc}{\bar{p}}
\newcommand{\pf}{p'}
\newcommand{\avec}{\mvec{a}}
\newcommand{\bvec}{\mvec{b}}
\newcommand{\fvec}{\mvec{f}}
\newcommand{\fvecTilde}{\tilde{\mvec{f}}}
\newcommand{\gvec}{\mvec{g}}
\newcommand{\uvec}{\mvec{u}}
\newcommand{\uvecTilde}{\tilde{\mvec{u}}}
\newcommand{\vvec}{\mvec{v}}
\newcommand{\wvec}{\mvec{w}}

\newcommand{\uvech}{\mvec{u}_h}
\newcommand{\vvech}{\mvec{v}_h}
\newcommand{\wvech}{\mvec{w}_h}
\newcommand{\zerovec}{\mathbf{0}}
\newcommand{\ph}{p_h}
\newcommand{\qh}{q_h}
\newcommand{\rh}{r_h}

\newcommand{\partialder}[2]{\frac{\partial #1}{\partial #2}}
\newcommand{\timeder}[1]{\partialder{#1}{t}}
\newcommand{\taum}{\tau_m}
\newcommand{\tauc}{\tau_c}
\newcommand{\taumMax}{\taum^{M}}
\newcommand{\taucMax}{\tauc^{M}}
\newcommand{\visco}{\nu}
\newcommand{\pder}[1]{\partialder{p}{\xii}}
\newcommand{\normalvec}{\hat{\mvec{n}}}

\newcommand{\spaceV}{\mvec{V}}
\newcommand{\spaceVd}{\mvec{V}^D}
\newcommand{\spaceVh}{{\spaceV}_h}
\newcommand{\spaceVhd}{\spaceVh^D}
\newcommand{\spaceQ}{Q}
\newcommand{\spaceQh}{{\spaceQ}_h}
\newcommand{\spaceHdiv}{\mvec{H}_{\text{div}}}
\newcommand{\interpolant}{I_h}
\newcommand{\interp}[1]{\interpolant #1}
\newcommand{\nsd}{d}
\newcommand{\nsdt}{\hat{d}}
\newcommand{\adv}{\mvec{a}}

\newcommand{\bit}{\Gamma_{0}}
\newcommand{\bft}{\Gamma_{T}}
\newcommand{\bsp}{\Gamma_S}
\newcommand{\mesh}{K^h}

\newcommand{\bftmesh}{\Gamma_{T}^h}
\newcommand{\grad}{\mvec{{\nabla}}}
\newcommand{\divergence}{\grad\cdot}
\newcommand{\ddiv}[1]{#1\cdot\grad}

\newcommand{\adiv}{\avec\cdot\grad}

\newcommand{\gradst}{\tilde{{\mvec{{\nabla}}}}}
\newcommand{\laplacian}{\Delta}

\newcommand{\bform}{B}
\newcommand{\bformUnstab}{A}
\newcommand{\bformh}{\bform_h}
\newcommand{\lform}{L}
\newcommand{\lformh}{\lform_h}

\newcommand{\mvec}[1]{{\mathbold{#1}}}
\newcommand{\mmat}[1]{{\mathbf{#1}}}

\newcommand{\intTimeSpace}{\int_0^T\int_{\spDom}}
\newcommand{\intSpaceTime}{\int_{\spDom}\int_0^T}
\newcommand{\intSpace}{\int_{\spDom}}
\newcommand{\intTime}{\int_0^T}

\newcommand{\intSTDom}{\int_{\stDom}}

\newcommand{\inner}[2]{\left( #1, #2 \right)}
\newcommand{\innerh}[2]{\left( #1, #2 \right)_h}
\newcommand{\innerFT}[2]{\left( #1, #2 \right)_{\Gamma_T}}
\newcommand{\innerFTh}[2]{\left( #1, #2 \right)_{\Gamma_T,h}}

\newcommand{\norm}[1]{\left\| #1 \right\|}
\newcommand{\normL}[3][2]{\left\| #2 \right\|_{L^{#1}(#3)}}
\newcommand{\normH}[3][1]{\left\| #2 \right\|_{H^{#1}(#3)}}
\newcommand{\seminorm}[1]{\left| #1 \right|}

\newcommand{\normh}[1]{\left\| #1 \right\|_h}
\newcommand{\normFT}[1]{\left\| #1 \right\|_{\bft}}
\newcommand{\normFTh}[1]{\left\| #1 \right\|_{\bft,h}}
\newcommand{\normElm}[1]{\left\| #1 \right\|_e}
\newcommand{\vertiii}[1]{{\left\vert\kern-0.25ex\left\vert\kern-0.25ex\left\vert #1 \right\vert\kern-0.25ex\right\vert\kern-0.25ex\right\vert}}

\newcommand{\normVW}[2][\avec]{\vertiii{#2}_{#1}}
\newcommand{\uproxy}{\gamma}

\newcommand{\oneOver}[1]{\frac{1}{#1}}
\newcommand{\half}{\frac{1}{2}}
\newcommand{\halfnice}{\nicefrac{1}{2}}

\newcommand{\approxConstant}{C_{a}}
\newcommand{\traceConstant}{C_{\Gamma}}
\newcommand{\shapeRegularConst}{C_h}

\newcommand{\petsc}{\textsc{PETSc}}
\newcommand{\parmetis}{\textsc{ParMETIS}}

\newcommand{\ut}{\partial_t \uvec}
\newcommand{\vt}{\partial_t \vvec}

\newcommand{\opLAdv}[1][a]{L_{{#1}}}
\newcommand{\opLTAdv}[1][a]{M_{{#1}}}
\newcommand{\opLAdvStar}[1][a]{L_{{#1}}^*}
\newcommand{\opLTAdvStar}[1][a]{M_{{#1}}^*}

\newcommand{\davec}{\delta\avec^{21}}
\newcommand{\duh}{\delta\tilde{\uvec}_h}
\newcommand{\dph}{\delta\tilde{p}_h}

\newcommand{\dt}{\Delta t}

\newcommand{\upair}{\phi}

\newcommand{\upairf}{\phi'}
\newcommand{\upairh}{\phi_h}

\newcommand{\vpair}{\psi}
\newcommand{\vpairh}{\psi_h}

\newcommand{\cinv}{C_{I}}
\newcommand{\elemVol}{\stDom^e}

\newcommand{\elemFT}{\bft^e}
\newcommand{\elmsum}{\Xsum_{e \in \mesh}}
\newcommand{\elmsumFT}{\Xsum_{e \in \bftmesh}}

\newcommand{\hzero}{h_0}
\newcommand{\he}{h_e}
\newcommand{\eep}{E_e^p}
\newcommand{\eeu}{E_e^u}
\newcommand{\eepFT}{F_e^p}
\newcommand{\eeuFT}{F_e^u}
\newcommand{\coer}{C_s}

\newcommand{\eie}[1]{#1 - \interp{#1}}
\newcommand{\setbuilder}[1]{\left\{ #1 \right\}}

\newcommand{\colref}[2]{\hyperref[#2]{#1~\ref*{#2}}}
\newcommand{\secref}[1]{\colref{Section}{#1}}
\newcommand{\figref}[1]{\colref{Figure}{#1}}
\newcommand{\tabref}[1]{\colref{Table}{#1}}
\newcommand{\lemref}[1]{\colref{Lemma}{#1}}
\newcommand{\cororef}[1]{\colref{Corollary}{#1}}
\newcommand{\thmref}[1]{\colref{Theorem}{#1}}
\newcommand{\assumref}[1]{\colref{Assumption}{#1}}
\newcommand{\logLogSlopeTriangle}[5]
{
    \pgfplotsextra
    {
        \pgfkeysgetvalue{/pgfplots/xmin}{\xmin}
        \pgfkeysgetvalue{/pgfplots/xmax}{\xmax}
        \pgfkeysgetvalue{/pgfplots/ymin}{\ymin}
        \pgfkeysgetvalue{/pgfplots/ymax}{\ymax}

        \pgfmathsetmacro{\xArel}{#1}
        \pgfmathsetmacro{\yArel}{#3}
        \pgfmathsetmacro{\xBrel}{#1-#2}
        \pgfmathsetmacro{\yBrel}{\yArel}
        \pgfmathsetmacro{\xCrel}{\xArel}
        \pgfmathsetmacro{\lnxB}{\xmin*(1-(#1-#2))+\xmax*(#1-#2)} % in 
        \pgfmathsetmacro{\lnxA}{\xmin*(1-#1)+\xmax*#1} % in [xmin,xmax].
        \pgfmathsetmacro{\lnyA}{\ymin*(1-#3)+\ymax*#3} % in [ymin,ymax].
        \pgfmathsetmacro{\lnyC}{\lnyA+#4*(\lnxA-\lnxB)}
        \pgfmathsetmacro{\yCrel}{\lnyC-\ymin)/(\ymax-\ymin)} % THE IMPROVED 
        \coordinate (A) at (rel axis cs:\xArel,\yArel);
        \coordinate (B) at (rel axis cs:\xBrel,\yBrel);
        \coordinate (C) at (rel axis cs:\xCrel,\yCrel);

        \draw[#5]   (A)-- node[pos=0.5,anchor=north] {1}
                    (B)-- 
                    (C)-- node[pos=0.5,anchor=west] {#4}
                    cycle;
    }
}
\DeclareFontFamily{U} {cmex}{}

\DeclareFontShape{U}{cmex}{m}{n}{
 <-6> cmex5
 <6-7> cmex6
 <7-8> cmex7
 <8-9> cmex8
 <9-10> cmex9
 <10-12> cmex10
 <12-> cmex12}{}

\DeclareSymbolFont{Xcmex} {U} {cmex}{m}{n}

\DeclareMathSymbol{\Xsum}{\mathop}{Xcmex}{80}
\begin{document}
	
\begin{frontmatter}
	\title{\textbf{Solving fluid flow problems in space-time with multiscale stabilization: formulation and examples}}
	
	%\author{Biswajit Khara, Robert Dyja, Kumar Saurabh, Anupam Sharma, Baskar Ganapathysubramanian}
	
	% Robert Dyja affiliation: Faculty of Mechanical Engineering and Computer Science, Czestochowa University of Technology, Dabrowskiego 69, 42-201 Czestochowa, Poland; e-mail: robert.dyja@icis.pcz.pl
	\author[ISUME]{Biswajit Khara
		\texorpdfstring{\corref{CORRAUTH}}{}}
	\emailauthor{bkhara@iastate.edu}{Biswajit Khara}
	\author[CUTME]{Robert Dyja}
	\author[ISUME]{Kumar Saurabh}
	\author[ISUAER]{Anupam Sharma}
	\author[ISUME]{Baskar Ganapathysubramanian \texorpdfstring{\corref{CORRAUTH}}{}}
	\emailauthor{baskarg@iastate.edu}{Baskar Ganapathysubramanian}
	\address[ISUME]{Department of Mechanical Engineering, Iowa State University, Ames, IA, USA}
	\address[ISUAER]{Department of Aerospace Engineering, Iowa State University, Ames, IA, USA}
	\address[CUTME]{Faculty of Mechanical Engineering and Computer Science, Czestochowa University of Technology, Czestochowa, Poland}
	\cortext[CORRAUTH]{Corresponding authors}
	
	% \maketitle
	
	\begin{abstract}
		We present a space-time continuous-Galerkin finite element method for solving incompressible Navier-Stokes equations. To ensure stability of the discrete variational problem, we apply ideas from the variational multi-scale method. The finite element problem is posed on the ``full" space-time domain, considering time as another dimension. We provide a rigorous analysis of the stability and convergence of the stabilized formulation. And finally, we apply this method on two benchmark problems in computational fluid dynamics, namely, lid-driven cavity flow and flow past a circular cylinder. We validate the current method with existing results from literature and show that very large space-time blocks can be solved using our approach.
	\end{abstract}
\end{frontmatter}
	
	% Include all the section-subfiles one after the other
	\section{Introduction}
\label{sec:intro}
%Intro-intro
%Spacetime parallelism
%Early spacetime multigrid methods
%Other attempts at space-time methods different usage of this terminology
%Continuous Galerkin space-time attempts
%Theoretical works
%Our work

Transient physical processes are typically described by time-dependent partial differential equations (PDEs). Mathematical analysis shows that, given smooth and compatible initial condition and boundary conditions, many of these equations have solutions of certain regularity in both space and time dimensions \cite{evans2022partial}. When solving these equations numerically, the most common practice is to use the method of lines which has two variants: (i) where the PDE is discretized in space to obtain a large system of ordinary differential equations (ODEs), which are then integrated in time; and (ii) where the equations are discretized is time first, to obtain a continuous PDE in spatial variables, which are then solved using techniques for solving stationary PDEs \cite{leveque2007finite}.

The nature of the method of lines makes it a sequential process. To illustrate this, suppose $ \temporalInterval = [0,T] $ denotes a time interval of interest. In the method of lines, we discretize this interval into $ n_t $ finite time steps, each of length $ \Delta t = \frac{T}{(n_t - 1)} $. So the PDE of interest, now semidiscrete, has to be solved at time points: $ \Delta t,\ 2\Delta t,\ 3\Delta t,...,\ (n_t-1)\Delta t = T $. In such algorithms, the solution at a particular step depends on the solutions at previous time steps. Thus, the solution process essentially becomes evolutionary: marching from one time step to the next, thereby mimicking the physical process itself. 

But, it is not necessary for a numerical method to really mimic the physical process in its evolutionary characteristics. Prior work \citep{nievergelt1964parallel,Hackbusch:1985:PMM:4673.4714,chartier1993parallel,horton1995space,gander1996overlapping,gander1999optimal,LionsEtAl2001} has shown that solving for a solution at a later time does not have to wait for the solutions of intermediate points to finish. Rather, all these computations can progress in parallel. A brief review of some of these methods can be found in \cite{gander201550}. A common theme that runs through these methods is the idea of parallelization of computation, possibly coupled with decomposing the spatiotemporal domain (i.e., the tensor product of the spatial domain and the time window) into multiple smaller subdomains.

Multiple types of decompositions have been proposed in the context of space-time parallelism. Early works such as \cite{nievergelt1964parallel,saha1996parallel} or more recent works such as \cite{LionsEtAl2001} decompose the PDE primarily in time. Such methods fall under the methods relying on temporal domain decomposition. Similarly, works based on waveform relaxation methods \cite{gander1996overlapping,gander1999optimal} achieve decomposition primarily in spatial domain. Domain decomposition in both space and time was considered by \cite{Hackbusch:1985:PMM:4673.4714, horton1995space, lubich1987multi}, where mulitigrid approach was adopted to solve the spatiotemporal system of equations. But it was noticed that the multigrid coarsening does not work the same way in time as it would in the spatial dimension; in particular, coarsening in time may not always lead to right convergence. 

In the context of finite element method (FEM), a method of lines discretization usually employs a finite difference scheme (e.g., Euler scheme), and many space-time parallel implementations rely on them \cite{LionsEtAl2001, dyja2018parallel}. But works such as \cite{hughes1988space, hulbert1990space} have applied discontinuous Galerkin (dG) discretization in the time dimension in tandem with continuous Galerkin (cG) approximation in space. These methods are also commonly known as ``space-time finite element methods'', but these are still sequential in nature, i.e., one usually solves one ``slab'' of time-window before solving another in order. Some more examples of this kind of discretization can be found in \cite{pontaza2004space, tezduyar2006space, tezduyar2008arterial, scovazzi2007stabilized, song2015nitsche, gander2016analysis}.

Recently, efforts have been made to apply continuous Galerkin method for time discretization \cite{munts2006space, behr2008simplex,steinbach2015space, langer2016space, fuhrer2021space}. And works such as \cite{ishii2019solving} explore the scalability and advantages of solving these formulations on large computing clusters. Most of these works are limited to linear parabolic equations such as the heat equation, but they do highlight some of the issues of approximating the dependence of the solution on time through continuous basis functions, in particular, the need for a stabilized method. Theoretical analysis \cite{andreev2012stability} of this problem also refers to a need for such a method. 

In this work, our goal is to develop a fast and scalable \textit{continuous Galerkin} method for a space-time coupled discretization of the Navier-Stokes equations. As will be discussed in \secref{sec:formulations}, the stabilization requirements for the space-time discretization of the Navier-Stokes equations reduce to the following: (i) stabilization of the dominant convection terms and (ii) stabilization of the saddle point nature of the incompressible Navier-Stokes equations. The first kind of stability can be provided by multiple ways, including the streamline upwind Petrov Galerkin (SUPG) method \cite{brooks1982streamline}, the Galerkin Least Square (GLS) method \cite{hughes1989new,douglas1989absolutely}, or the variational multiscale (VMS) method \cite{hughes1995multiscale}. The second type of stability can be provided by the pressure stabilized Petrov Galerkin (PSPG) method \cite{hughes1986new}. But it has been shown that the application of the multiscale ideas (i.e., VMS) can naturally lead to a stabilization that include both SUPG and PSPG type of stabilizations \cite{bazilevs2007variational}, along with a grad-divergence type stabilization.

Thus, in this paper, we argue that VMS also helps stabilize the space-time variational equations in a continuous Galerkin setting. The variational problem then can be solved in a domain-decomposed manner to obtain a fully coupled space-time problem that also translates into a stable linear algebra problem. Therefore, our contributions in this paper are as follows:
\begin{enumerate}
	\item A continuous Galerkin method for solving the Navier-Stokes equations in coupled space-time.
	\item An application of VMS to stabilize the linear algebra problem against both spatial and temporal convective effects and with equal-order velocity-pressure pair spaces.
	\item A rigorous analysis of the space-time variational problem.
	\item Validation of the method against benchmark problems.
\end{enumerate}

The rest of the paper is organized as follows:  \secref{sec:formulations} presents the mathematical background and derives the variational formulation, \secref{sec:analysis} provides an analysis of the variational problem, \secref{sec:implementation} provides the implementations details, \secref{sec:num-examples-results} presents the numerical experiments and validation results; and  \secref{sec:conclusions} provides some discussions and conclusions.

\section{Space-time variational formulation}
\label{sec:formulations}
\subsection{Preliminaries} \label{sec:prelim}
\begin{figure}[!t]
	\centering
	\def \mpagefrac{0.7}
	\begin{minipage}{\mpagefrac\textwidth}
		\centering
		\begin{tikzpicture}
		\def \myscale{\mpagefrac / 0.5} % needs \usepackage{calc}
		\draw (0, 0) node[inner sep=0] {\includegraphics[trim=3.3in 0.5in 3.3in 0.8in, clip,width=0.7\linewidth]{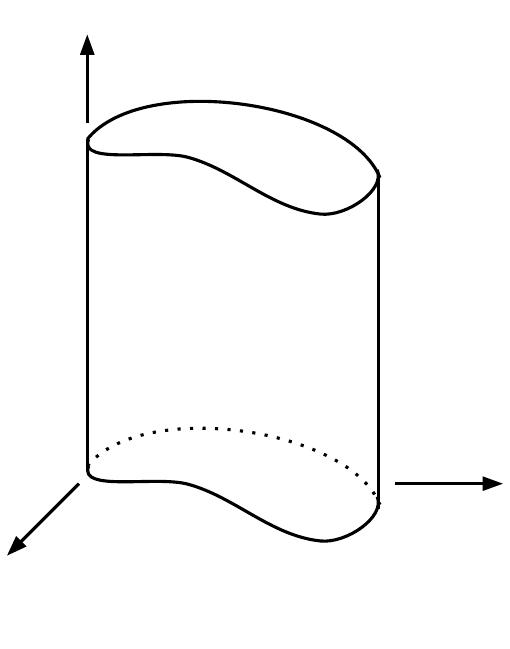}};
		\draw (0, 0) node {$ \stDom = \spDom \times \temporalInterval $};
		\draw (-2.5*\myscale, 3.0*\myscale) node {$t$};
		\draw (-2.7*\myscale,  2.1*\myscale) node {$t=T$};
		\draw (-2.7*\myscale, -1.5*\myscale) node {$t=0$};
		\draw (-2.2*\myscale, -2.7*\myscale) node {$x_1$};
		\draw ( 2.6*\myscale, -1.5*\myscale) node {$x_2$};
		\draw[stealth-, thick] (-0.2*\myscale, -1.6*\myscale) -- ++(0.5*\myscale,-1.3*\myscale) node[right, black, fill=white, anchor=north]{$\bit = \spDom\times \{0\}$};
		\draw[stealth-, thick] (-0.2*\myscale, 2.1*\myscale) -- ++(0.2*\myscale,1.0*\myscale) node[right, black, fill=white, anchor=south]{$\bft = \spDom\times \{T\}$};
		\draw[stealth-, thick] (1.45*\myscale, 0.1*\myscale) -- ++(0.7*\myscale,1.0*\myscale) node[right, black, fill=white, anchor=south west]{$\bsp = \partial \spDom \times (0,T)$};
%		\draw (3*\myscale, 3*\myscale) node {$ (3, 3) $};
%		\draw (-3*\myscale, -3*\myscale) node {$ (-3, -3) $};
		\end{tikzpicture}
	\end{minipage}
	\caption{Schematic depiction of the space-time domain $ \stDom = \spDom \times \temporalInterval $, where $ \spDom \subset \mathbb{R}^{\nsd} $, and $ \temporalInterval = [0,T] \subset \mathbb{R}^{+}. $} 
	\label{fig:schematic-space-time-domain}
\end{figure}
Consider an open, bounded spatial domain $\spDom \in \mathbb{R}^\nsd, \nsd\in\{2,3\}$ with Lipschitz continuous boundary $\partial \spDom$ and a bounded time interval $\temporalInterval = (0,T] \in \mathbb{R^+}$. We define the space-time domain as the Cartesian product of the two: $\stDom = \spDom \times \temporalInterval = \spDom\times(0,T] \subset \mathbb{R}^{\nsdt}$, $\nsdt = \nsd + 1$ (see \figref{fig:schematic-space-time-domain}). Suppose $ \xvec\in\spDom $ and $ t\in\temporalInterval $ and define $ \xvecst = (\xvec, t)^T \in \stDom$. The overall boundary of this space-time domain is defined as $\Gamma = \partial \stDom$. This ``overall boundary" is the union of the ``spatial boundaries" and the ``time boundaries". The spatial domain boundary is denoted by $\bsp = \partial \spDom \times (0,T)$; whereas the time boundaries are denoted by $\bit = \bar{\spDom}\times\{0\}$ and $\bft = \bar{\spDom}\times\{T\}$ which are the initial and final time boundaries respectively. Therefore $\Gamma = \bsp \cup \bit \cup \bft$.

On this space-time domain $ \stDom $, the incompressible Navier-Stokes equations can then be written for the vector function $\uvec = \uvec(\xvecst): \stDom \rightarrow \mathbb{R}^\nsd$ and the scalar function $ p = p(\xvecst):\stDom\rightarrow \mathbb{R} $ as:
\begin{subequations}\label{eq:ns-original}
	\begin{align}
	\timeder{\uvec}+(\uvec\cdot\grad)\uvec + \half(\divergence\uvec)\uvec -\visco\laplacian \uvec+\grad p - 
	\fvec &= \zerovec, \text{ }\xvecst\in\stDom, \label{eq:ns-original-momentum}
	\\
	\grad\cdot\uvec &= 0, \text{ }\xvecst\in\stDom, \label{eq:ns-original-continuity}
	\\
	\uvec &= \zerovec, \text{ }\xvecst\in\bsp, \label{eq:ns-original-dirichlet}
%	\\
%	\mathscr{B}_{NR} \uvec &= h, \text{ }\xvecst\in\Gamma_{S2}\subseteq\bsp, \label{eq:ns-original-neumann}
	\\
	\uvec &= \uvec_0, \text{ }\xvecst\in\bit, \label{eq:ns-original-ic}
	\end{align}
\end{subequations}
where, $ \uvec = (\ucomp{1}, \ldots, \ucomp{\nsd})^T $ is the velocity vector, $ p $ is the pressure, $ \nu $ is the kinematic viscosity (assumed constant), and $\fvec : \stDom \rightarrow \mathbb{R}^\nsd$ is the forcing function (assumed to be smooth). Note that $\grad$ is the usual spatial gradient operator in the space $\mathbb{R}^{\nsd}$, i.e., $ \grad \equiv (\nicefrac{\partial}{\partial x}, \nicefrac{\partial}{\partial y}) $ for $ \nsd = 2 $ and $ \grad \equiv (\nicefrac{\partial}{\partial x}, \nicefrac{\partial}{\partial y}, \nicefrac{\partial}{\partial z}) $ for $ \nsd = 3 $ respectively. We further define the space-time gradient operator $\gradst$ as: $ \gradst \equiv (\grad, \nicefrac{\partial}{\partial t}) $.

%The velocity $ \uvec $ on the domain boundaries is determined via some essential (or Dirichlet) boundary condition on a subset of the spatial boundary $ \Gamma_{S1} $. In addition, a Neumann/Robin condition is applied on another subset of the spatial boundary $ \Gamma_{S2} $ (equation \ref{eq:ns-original-neumann}). Here, $ \mathscr{B}_{NR} $ is a  linear combination of differential operators of order 1 or 0. Note that, $ \Gamma_{S1}\cup\Gamma_{S2} \subseteq \bsp $ and $ \Gamma_{S1}\cap\Gamma_{S2} = \varnothing $.

Let us define two operators $ \opLAdv[\avec] $ and $ \opLTAdv[\avec] $ as follows:
\begin{subequations} \label{eq:def-la-ma}
	\begin{align}
	\opLAdv[\avec] \uvec &= \adiv \uvec + \halfnice (\divergence \avec) \uvec, \label{eq:opL} \\
	\text{and}\ \ \opLTAdv[\avec] \uvec &= \ut + \opLAdv[\avec] \uvec = \ut + \adiv \uvec + \halfnice (\divergence \avec) \uvec. \label{eq:opLT}
	\end{align}
\end{subequations} 
For a given $ \avec $, the operators $ \opLAdv[\avec]\uvec $ and $ \opLTAdv[\avec]\uvec $ are linear in $ \uvec $, whereas $ \opLAdv[\uvec]\uvec,\ \opLTAdv[\uvec]\uvec $ are nonlinear. With this notation, we can rewrite \eqref{eq:ns-original-momentum} as
\begin{align} \label{eq:ns-linearized-0}
\opLTAdv[\uvec]\uvec - \visco\laplacian\uvec + \grad p &= \fvec \ \text{in} \ \stDom.
\end{align}

\begin{remark}
	Note that in \eqref{eq:ns-original-momentum} and \eqref{eq:def-la-ma}, we have included the additional terms $ \halfnice(\divergence\uvec)\uvec $ and $ \halfnice(\divergence\avec)\uvec $ respectively. This term was introduced in \cite{temam1968methode}, and has been included in the literature to ensure dissipativity of  various numerical schemes, especially when the function spaces are not divergence-free \cite{shen1996error, shen1997pseudo, codina2001stabilized}. In addition, the presence of this term makes the operators $ \opLAdv[\avec] $ and  $ \opLTAdv[\avec] $ skew-symmetric (see \autoref{item:adjoint-op} in \secref{sec:anlys-prelim}).
\end{remark}

\begin{remark}
	The case of non-zero boundary  conditions will be discussed in \secref{sec:nonzero-icbc}.
\end{remark}

\subsubsection{Space-time inner products and norms}
\label{sec:inner-prod-norms}
The $ L^2 $-inner product on the space-time domain $ \stDom $ is expressed as: 
\begin{itemize}
	\item Scalar-valued functions:
	\begin{align} \label{eq:inner-prod-scalar}
	\inner{a}{b} &:= \inner{a}{b}_{\stDom} = \intSTDom  a(\xvecst)b(\xvecst) \text{ }d\xvecst = \intTimeSpace a(\xvec, t)b(\xvec,t) \text{ } d\mvec{x} dt.
	\end{align}
	\item Vector-valued functions:
	\begin{align} \label{eq:inner-prod-vector}
	\inner{\mvec{a}}{\mvec{b}} &:= \inner{\mvec{a}}{\mvec{b}}_{\stDom} = \intSTDom  \mvec{a}(\xvecst)^T \mvec{b}(\xvecst) \text{ }d\xvecst = \intTimeSpace \mvec{a}(\xvec, t)\cdot\mvec{b}(\xvec,t) \text{ } d\mvec{x} dt,
	\end{align}
	\item Gradients of vector-valued functions:
	\begin{align} \label{eq:inner-prod-gradients}
	\inner{\grad \mvec{a}}{\grad \mvec{b}} &:=   \Xsum_{i=1}^{\nsd} \inner{\grad a_i}{\grad b_i},
	\end{align}
	where $ a_i $ and $ b_i $ are the scalar components of $ \avec $ and $ \bvec $ respectively. Each inner product $ \inner{\grad a_i}{\grad b_i} $ is calculated according to \eqref{eq:inner-prod-vector}.
\end{itemize}

We define a ``temporal slice'' as a slice of $ \stDom $ obtained by fixing the time at a particular $ t=t' $, and denote it by $ \spDom(t') $. An inner product on such a slice at $ t=t' $ will be denoted with a subscript, i.e.,
\begin{align}
\inner{\mvec{a}}{\mvec{b}}_{\spDom(t')} = \int_{\spDom(t')} \mvec{a}(\xvec,t=t')^T \mvec{b}(\xvec,t=t')\text{ } d \mvec{x}.
\end{align}
A special case of temporal slices is the ``final time'' slice, i.e., $ \spDom(T) $. Since $ \spDom(T) $ is the same as $ \bft $, we will denote this inner product by
\begin{align}
\innerFT{\avec}{\bvec} = \inner{\mvec{a}}{\mvec{b}}_{\spDom(T)} = \int_{\spDom(T)} \avec(\xvec,T)^T\bvec(\xvec,T) \ d\xvec.
\end{align}
Similarly, the space-time and the spatial $ L^2 $-norms are defined as:
\begin{align}
\norm{\avec} := \norm{\avec}_{\stDom} = \normL{\avec}{\stDom} &= \sqrt{\inner{\avec}{\avec}}, \\
\norm{\avec}_{\spDom(t')} := \normL{\avec}{\spDom(t')} &= \sqrt{\inner{\avec}{\avec}_{\spDom (t')}}, \ t' \in\temporalInterval, \\
\normFT{\avec} := \norm{\avec}_{L^2(\bft)} = \normL{\avec}{\spDom (T)} &= \sqrt{\inner{\avec}{\avec}_{\spDom (T)}}.
\end{align}
respectively, where $ \avec(\xvec,t) $ is any function that belongs to $ L^2(\stDom) $. In the sequel, unless otherwise specified, $ \inner{\cdot}{\cdot} $ and $ \norm{\cdot} $  will denote $ \inner{\cdot}{\cdot}_{\stDom} $ and $ \norm{\cdot}_{\stDom} $ respectively.

\begin{remark}[\textit{Integration by parts}]
	A typical integration-by-parts \textit{over spatial derivatives} can be written as follows.
	\begin{itemize}
		\item Scalar functions:
		\begin{align}
		\inner{\laplacian a}{b} = \intTimeSpace \laplacian a \ b \ d\spDom \ dt &= \intTime \left[ -\intSpace \grad a \cdot \grad b\ d\spDom + \int_{S=\partial \spDom}  (\normalvec\cdot\grad a)b \ dS \right]dt \\
		&= -\inner{\grad a}{\grad b} + \inner{\normalvec\cdot \grad a}{b}_{\bsp}.
		\end{align}
		\item Vector functions:
		\begin{align*}
		\inner{\laplacian \avec}{\bvec} = \Xsum_{i=1}^{\nsd} \inner{\laplacian a_i}{b_i} &= \Xsum_{i=1}^{\nsd} \left[-\inner{\grad a_i}{\grad b_i} + \inner{\normalvec\cdot \grad a_i}{b_i}_{\bsp} \right] \\
		&= -\inner{\grad \avec}{\grad \bvec} + \inner{(\normalvec\cdot \grad)\avec}{\bvec}_{\bsp}.
		\end{align*}
	\end{itemize}	
	Similarly, integration by parts \textit{over temporal derivatives} can be written as:
	\begin{itemize}
		\item Scalar functions:
		\begin{align*}
		\inner{\partial_t a}{b} = \intSpaceTime \partial_t a\ b\ dt d\spDom &= \intSpace \left[-\intTime a\ \partial_t b\ dt + ab \big|_{0}^{T} \right]d\spDom \\
		&= -\inner{a}{\partial_t b} + \intSpace a(\xvec,T)b(\xvec,T)d\spDom - \intSpace a(\xvec,0)b(\xvec,0) d\spDom \\
		&= -\inner{a}{\partial_t b} + \innerFT{a}{b} - \inner{a}{b}_{\bit}
		\end{align*}
		\item Vector functions:
		\begin{align*}
		\inner{\partial_t \avec}{\bvec} = \Xsum_{i=1}^{\nsd} \inner{\partial_t a_i}{b_i} &= -\inner{a_i}{\partial_t b_i} + \innerFT{a_i}{b_i} - \inner{a_i}{b_i}_{\bit} \\
		&= -\inner{\avec}{\partial_t \bvec} + \innerFT{\avec}{\bvec} - \inner{\avec}{\bvec}_{\bit}.
		\end{align*}
	\end{itemize}
\end{remark}

\subsubsection{Discretization, discrete inner products, and discrete norms}
\label{sec:discretization}
We define a tessellation $\mesh$ as the partition of $\stDom$ into a finite number of non-overlapping $(\nsd+1)$-dimensional elements $e$ such that the closure of $\stDom$ is given by $\bar{\stDom} \approx \bigcup_{e \in \mesh} e$. Let us denote the projection of $\mesh$ on the $\nsd$-dimensional slice of $\stDom$ at $t=T$ by $\bftmesh$, such that the closure of the ``final-time boundary'' $\bft$ is given by $\bar{\bft} \approx \bigcup_{e \in \bftmesh} e $.  Given a mesh $ \mesh $ (and consequently $\bftmesh$), we define the following notations for discrete inner products and norms on the mesh $ \mesh $ and $\bftmesh$:
\begin{subequations} \label{eq:discrete-norms}
	\begin{align}
		\innerh{\avec}{\bvec} &:= \Xsum_{e \in \mesh} \inner{\avec}{\bvec}_{e}, \quad
		\normh{\avec} := \left[\Xsum_{e \in \mesh} \inner{\avec}{\avec}_{e}\right]^{\halfnice}, \\
		\innerFTh{\avec}{\bvec} &:= \Xsum_{e \in \bftmesh} \inner{\avec}{\bvec}_{e} \quad \normFTh{\avec} := \left[\Xsum_{e \in \bftmesh} \inner{\avec}{\bvec}_{e} \right]^{\halfnice}.
	\end{align}
\end{subequations}

\subsection{Stabilized finite element formulations} \label{sec:stabilized-formulations}

\subsubsection{Galerkin formulation and its lack of stability} \label{sec:lack-of-stability}
A classical Galerkin FEM formulation of \eqref{eq:ns-original} suffers from two major issues:
%There are two major issues associated with the FEM formulation in Eq. \eqref{eq:var-problem-unstabilized}:
\begin{itemize}
	\item The first issue is associated with the choice of the continuous function spaces themselves, i.e., each of the velocities belong to the same function space as the pressure. This is an obvious case of a non \textit{inf-sup} stable pair of function spaces. So, a stable numerical solution cannot be guaranteed \cite{babuvska1973finite, brezzi1974existence, john2016finite}.
	\item At the discrete level, a continuous Galerkin method is most commonly implemented using low order local Lagrangian basis functions (globally $ C^0 $) which generally have no directional properties. Differential operators such as the Laplacian can be approximated very well with such basis functions, so that the final linear algebra system is stable. But it is well known that the use of such functions to approximate differential operators possessing directional properties (e.g., first derivatives like $ \partialder{}{x} $) can lead to numerical instabilities if these derivatives become dominant in the equation \cite{brooks1982streamline, hughes1989new, codina1998comparison}. In the case of the Navier-Stokes equations, this might happen when the convection term $ (\uvec\cdot\grad)\uvec $ becomes much larger than the diffusion term $ -\nu\laplacian \uvec $ (or in other words, when $ \nu \ll 1 $). In addition to the spatial advection term $ (\uvec\cdot\grad)\uvec $, the time derivative term $ \timeder{\uvec} $ also possesses a directional property.
\end{itemize}

There are multiple ways to resolve these two issues. The issue of instabilities due to non inf-sup conformity can be resolved by using methods such as PSPG \cite{hughes1986new, douglas1989absolutely, tezduyar1992incompressible} that remove the saddle point nature of the resulting FEM problem. The second issue can be resolved by using some kind of ``upwinding''-type modification to the test function $ \mvec{w} $ (e.g., SUPG \cite{brooks1982streamline}) or by using a least squares  modifications to the FEM problem (e.g., GLS \cite{hughes1989new}) etc. In this work, we apply the variational multiscale method (VMS \cite{hughes1995multiscale}) which provides both kinds of stability.

\subsubsection{Multiscale decomposition, and function spaces}

The variational multiscale (VMS) method is a technique that naturally resolves both issues discussed above in a consistent manner \cite{bazilevs2007variational}. Following the VMS methodology of \cite{hughes1995multiscale, bazilevs2007variational}, we decompose the infinite dimensional velocity $ \uvec $ and pressure $ p $ as:
\begin{align}\label{eq:vms-coarse-fine-decomposition}
\uvec = \uc + \uf \quad \text{and} \quad p = \pc + \pf,
\end{align}
where the bar ( $ \bar{} $ ) denotes a \emph{coarse scale} that can be resolved by the numerical method (i.e., by the grid/mesh) and the prime $(') $ denotes the \emph{fine scales} or the \emph{sub-grid scales} that are not resolved by the numerical method.

Let us fix the mesh $ \mesh $, and suppose that $ \uvech $ and $ \ph $ are the discrete versions of $ \uc $ and $ \pc $ on $ \mesh $ respectively. The variational problem is then formulated for $ \uvech $ and $ \ph $ and the \textit{effect} of $ \uvec' $ and $ p' $ (i.e., the sub-grid scales) is expressed using the residuals of $ \uvech $ and $ \ph $ (i.e., the coarse scales). In the residul based VMS approach, the fine scales are modeled using the coarse-scale residuals (for a given $ \avec $) as
 \begin{subequations}
 	\label{eq:fine-scale-approximations}
 	\begin{align}
 	\uf &= -\taum^K (\opLTAdv[\avec]\uvech -\visco\laplacian\uvech + \grad\ph - \fvec) \ \text{in any element}\ K\in \mesh, \\
 	\pf &= -\tauc^K (\divergence\uvech) \ \text{in any element}\  K \in \mesh,
 	\end{align}
 \end{subequations}
 where  $ \taum $ and $ \tauc $ are some element-wise parameters.

\subsubsection{Stabilized variational form}
Now, consider the function spaces
\begin{subequations}\label{eq:fem-subspaces-modified}
	\begin{align}
		\spaceV &:= \left\{ \mvec{v} \in \mvec{H}^1(\stDom):\  \mvec{v} = \zerovec \text{ on }\bsp\cup\bit\right\}, \\
		\spaceVd &:= \left\{ \mvec{v} \in \mvec{H}^1(\stDom):\  \mvec{v} = \zerovec \text{ on }\bsp,\ \vvec = \uvec_0 \ \text{on} \ \bit \right\}, \\
		\spaceQ &:= {H}^1(\stDom), \\
		\spaceVh &:= \left\{ \vvech \in {\mvec{V}}:\  \vvech|_{e} \in \mvec{P_k}(e), \text{ } e \in \mesh \right\}, \\
		\spaceVhd &:= \left\{ \vvech \in {\spaceVd}:\  \vvech|_{e} \in \mvec{P_k}(e), \text{ } e \in \mesh \right\}, \\
		\spaceQh &:= \left\{ \qh \in {Q} :\ \qh|_{e} \in P_k(e), \text{ } e \in \mesh \right\}.
	\end{align}
\end{subequations}
Here $P_k(e),\ k \geq 1$, is the space of tensor-product polynomial functions in element $e$ where the tensor-product basis functions are constructed by polynomials of degree $k$ in each dimension. And $\mvec{P_k}(e)$ is the space of vector-valued functions in element $e$.
%\begin{align}\label{eq:fem-subspaces-modified}
%\spaceVh &= \left\{ \vvech \in {\mvec{V}}:\  \vvech|_{K} \in \mvec{P_k}(K), \text{ } K \in \mesh \right\}, \\
%\spaceQh &= \left\{ \qh \in {Q} \enspace|\enspace \qh|_{K} \in P_k(K), \text{ } K \in \mesh \right\}.
%\end{align}

For describing the variational problem, we  will use the following symbols for test and trial function pairs:
\begin{subequations}
	\begin{align}
	\phi &:= (\uvec,p) \quad \text{(trial functions)}, \\
	\psi &:= (\vvec,q) \quad \text{(test functions)}.
	%		\omega &:= (\wvec,r) \quad \text{(special trial functions)}
	\end{align}
\end{subequations}

For a given $ \avec \in \spaceVd$, multiplying \eqref{eq:ns-original} by a vector test function $ \vvec \in \spaceV $ and performing integration-by-parts wherever applicable, we have the following bilinear and linear forms $ \bformUnstab(\avec;\upair,\vpair) $ and $ \lform(\avec; \vpair) $: 
%Let us define the bilinear $ \bformUnstab(\avec;\upair,\vpair) $ and the linear form $ \lform(\avec; \vpair) $ for a given convection field $ \avec \in \spaceVd $ as follows:
\begin{subequations}\label{eq:var-problem-b-and-f}
	\begin{align}\label{eq:var-problem-b}
	\bformUnstab(\avec;\upair,\vpair) 
	&= - \inner{\uvec}{\opLTAdv[\avec]\vvec}+\visco\inner{\grad \uvec}{\grad \vvec}-\inner{p}{\grad\cdot \vvec} + \inner{\grad\cdot\uvec}{q} + \inner{\uvec}{\vvec}_{\bft}, \\
	\label{eq:var-problem-f}
	\lform(\vpair) &= \inner{\fvec}{\vvec}.
	\end{align}
\end{subequations}

Substituting $ (\uvec,p) = (\uvech,\ph) + (\uf,\pf) $ (or, $ \phi = \upairh + \upairf $ ) in \eqref{eq:var-problem-b}, we have
\begin{subequations}
	\label{eq:bform-vms-decomp}
	\begin{align}
		\bformUnstab(\avec;\upair,\vpair) &= \bformUnstab(\avec;\upairh,\vpairh) - \innerh{\uf}{\opLTAdv[\avec]\vvech} - \innerh{\pf}{\divergence \vvech} - \innerh{\uf}{\grad \qh} + \innerFTh{\uf}{\vvech} \\
		&= \bformUnstab(\avec;\upairh,\vpairh) - \innerh{\uf}{[\opLTAdv[\avec]\vvech+\grad \qh]} - \innerh{\pf}{\divergence \vvech} + \innerFTh{\uf}{\vvech} .
	\end{align}
\end{subequations}
Now, substituting the expressions for $ (\uf,\pf) $  \eqref{eq:fine-scale-approximations}, and also the expression for $ \bformUnstab(\avec;\upair,\vpair) $  \eqref{eq:var-problem-b-and-f} in \eqref{eq:bform-vms-decomp}, we can rearrange the equation as
\begin{align}
\bformh(\avec; \upairh, \vpairh) = \lformh(\avec; \vpairh),
\end{align}
where
\begin{subequations} \label{eq:fem-problem-vms-bh-lh}
\begin{align}
\label{eq:fem-problem-vms-bh}
\begin{split}
\bformh(\avec; \upairh,\vpairh) &= - \inner{\uvech}{\opLTAdv[\avec]\vvech}+\visco\inner{\grad \uvech}{\grad \vvech}-\inner{\ph}{\grad\cdot \vvech} + \inner{\grad\cdot\uvech}{\qh} + \inner{\uvech}{\vvech}_{\bft} \\ &\quad + \innerh{\taum [\opLTAdv[\avec]\uvech-\visco\laplacian\uvech + \grad \ph ]}{[\opLTAdv[\avec]\vvech+\visco\laplacian\vvech +\grad \qh]} \\
&\quad + \innerh{\tauc \divergence \uvech}{\divergence \vvech} -\innerFTh{\taum [\opLTAdv[\avec]\uvech-\visco\laplacian\uvech + \grad \ph ]}{\vvech},
\end{split}
\\
\label{eq:fem-problem-vms-lh}
\lformh(\avec; \vpairh) &= \inner{\fvec}{\vvech} + \innerh{\taum \fvec}{[\opLTAdv[\avec]\vvech-\visco\laplacian\vvech +\grad \qh]} - \innerFTh{\taum \fvec}{\vvech}.
\end{align}
\end{subequations}

Then the stabilized finite element method is to find $ \upairh = (\uvech,\ph) \in \spaceVhd\times\spaceQh $, such that
\begin{align}\label{eq:fem-problem-vms}
\begin{split}
\bformh(\uvech; \upairh, \vpairh) &= \lformh(\upairh;\vpairh) \quad \forall \vpairh \in \spaceVh \times \spaceQh.
\end{split}
\end{align}

\begin{remark}
	Note that in \eqref{eq:fem-problem-vms}, a choice of basis functions $ \vpairh = (\vvech, 0) $ will recover the (weak) momentum equations, whereas a choice $ \vpairh = (\zerovec, \qh) $ will recover the (weak) continuity equation.
\end{remark}
\begin{remark}
	Note that for a given $ \avec \in \spaceVhd $, the inner product operators defined in \eqref{eq:fem-problem-vms-bh-lh} are linear, i.e., $ \bformh(\avec;\upairh,\vpairh) $ is linear in both $ \upairh $ and $ \vpairh $; and similarly $ \lform(\avec; \vpairh) $ is linear in $ \vpairh $. However, \eqref{eq:fem-problem-vms} is \textit{nonlinear} since the unknown $ \uvech $ appears in place of $ \avec $.
\end{remark}

\begin{remark}[\textbf{Differences with respect to method of lines discretizations}]
	In a \textit{method of lines} discretization, the time derivative $ \ut $ is discretized using finite differences. Suppose the time interval $ \temporalInterval $ is discretized into $ N $ points $ \{t_i\},\ i=1,\ldots,N $. Using the following notation:
	\begin{align*}
	p_i &= p(\cdot,t_i), \\
	\uvec_i &= \uvec(\cdot,t_i), \\
	\partial_t \uvec_i &= \partial_t \uvec(\cdot,t_i), \\
	\inner{\avec}{\bvec}_i &= \inner{\avec}{\bvec}_{\spDom(t_i)},
	\end{align*}	
	and assuming $ \uvec_D= \zerovec $ in \eqref{eq:ns-original-dirichlet}, the Galerkin formulation of \eqref{eq:ns-original} at $ t= t_{i} $ would be to find $ \uvec_i \in H^1_0(\spDom(t_i)) $ such that
	\begin{align} \label{eq:var-prob-sequential}
	\inner{\ut_{i}}{\vvec}_{i} + \inner{\opLAdv[\uvec_{i}] \uvec_i}{\vvec}_{i} + \visco \inner{\grad \uvec_i}{\grad \vvec}_{i} -\inner{p_i}{\grad\cdot \vvec}_{i} + \inner{\grad\cdot\uvec_i}{q}_{i} = \inner{\fvec_i}{\vvec}_{i}
	\end{align}
	for all $ \vvec \in H^1_0(\spDom(t_i)) $. Comparing \eqref{eq:fem-problem-vms} with \eqref{eq:var-prob-sequential}, we note the following:
	\begin{enumerate}[(i)]
		\item \textit{Multiscale decomposition:} When the VMS method is applied to \eqref{eq:var-prob-sequential}, the following decomposition is assumed:
		\begin{align} \label{eq:vms-decomp-sequential}
		\uvec(\xvec,t_i) &= \uc(\xvec,t_i) + \uf(\xvec,t_i) \\
		\implies \uvec_i &= \uc_i + \uf_i.
		\end{align}
		Therefore, the VMS approximation in a sequential discretization is essentially a \textit{spatial approximation}. However, in the present formulation, the multiscale decomposition becomes a \textit{spatiotemporal} ansatz since we assume
		\begin{align} \label{eq:vms-decomp-st}
		\uvec(\xvec,t) &= \uc(\xvec,t) + \uf(\xvec,t),
		\end{align}
		%\todo{Should we add a comment comparing with LES methodology?}
		
		\item All the inner products appearing in \eqref{eq:var-problem-b-and-f} are over the space-time $ \stDom $, whereas those appearing in \eqref{eq:var-prob-sequential} are over the spatial domain $ \spDom(t_i) $.
		
		\item \textit{The time derivative term:} When the multiscale decomposition \eqref{eq:vms-decomp-sequential} is applied to \eqref{eq:var-prob-sequential}, the time derivative term $ \inner{\ut_i}{\vvec} $ yields
		\begin{align*}
		\inner{\ut_i}{\vvec}_i = \inner{\partial_t \uc_i}{\vvec}_i + \inner{\partial_t \uf_i}{\vvec}_i.
		\end{align*}		
		In \cite{bazilevs2007variational}, it is assumed that
		\begin{subequations}
			\begin{align}
			\inner{\partial_t \uf_i}{\vvec}_i &= 0 \\
			\text{or}\quad \inner{\ut_i}{\vvec}_i &= \inner{\partial_t \uc_i}{\vvec}_i
			\end{align}
		\end{subequations}
		for all $ i $. However, in the current formulation, we \textit{do not} make such an assumption. Noting $ \uvec' (\cdot,t=0) = \zerovec $, we have
		\begin{align*}
		\inner{\ut'}{\vvec} = -\inner{\uvec'}{\vt} + \innerFT{\uvec'}{\vvec}.
		\end{align*}				
		
		\item \textit{The viscous term:} Similarly, the viscous term from \eqref{eq:var-prob-sequential} gives
		\begin{align}
		\inner{\grad \uvec_i}{\grad \vvec}_{i} = \inner{\grad \uc_i}{\grad \vvec}_{i} + \inner{\grad \uf_i}{\grad \vvec}_{i}.
		\end{align}
		In \cite{bazilevs2007variational}, the following projection is assumed on each time slab $ i $:
		\begin{subequations}
			\begin{align} \label{eq:vms-projection-1}
			\inner{\grad\uvec'_i}{\grad \vvec}_{i} &= 0 \\
			\label{eq:vms-projection-2}
			\text{or}\quad\inner{\grad \uvec_i}{\grad \vvec}_{i} &= \inner{\grad \uc_i}{\grad \vvec}_{i}.
			\end{align}
		\end{subequations} 
		In this paper, we \textit{extend this assumption} over the whole spatiotemporal domain $ \stDom $ as
		\begin{align} \label{eq:vms-projection-spacetime}
		\inner{\grad \uvec}{\grad \vvec} = \inner{\grad \uc}{\grad \vvec}.
		\end{align}
		This implies
		\begin{align*}
		0 = \inner{\grad\uvec'}{\grad \vvec} = \intTime \inner{\grad \uf}{\grad \vvec}_{\spDom(t)}\ dt,
		\end{align*}
		or $ \inner{\grad \uf}{\grad \vvec}_{\spDom(t)} = 0 $ almost everywhere in $\temporalInterval$. We have made use of assumption \eqref{eq:vms-projection-spacetime} in deriving the expression for $ B^h $ in \eqref{eq:fem-problem-vms-bh}.
	\end{enumerate}
\end{remark}
\begin{remark}
	The issue of numerical instability due to the time derivative does not become prominent in a conventional time-marching formulation, because the time derivative $ \ut $ is  approximated using either finite differences or time-discontinuous basis functions, but \textit{not} using continuous basis functions.	For example, using a BDF-$ r $ scheme, the time derivative at time $ t_{i} $ is approximated as:
	\begin{align}\label{eq:time-disc-bdf-r}
	\ut_i = \oneOver{\dt} \Xsum_{j=0}^{r} \beta_{j}\uvec_{i-j},
	\end{align}
	where $ \beta_j $ are constants. In the weak formulation, we obtain the following inner product:
	\begin{align}
	\inner{\ut_i}{\vvec}_{\spDom} = \oneOver{\dt} \inner{\left[\Xsum_{j=0}^{r} \beta_{j}\uvec_{i-j}\right]}{\vvec}_{i}.
	\end{align}
	The term $ \inner{\nicefrac{\uvec_{i}}{\dt}} {\vvec}_{\spDom} $ effectively acts a linear ``mass'' term, and adds to the stability of the linear algebra problem.
\end{remark}
\begin{remark}
	The stabilization parameters $ \taum $ and $ \tauc $ are defined element-wise and are intimately related to the transformation Jacobian between element and the master element \cite{hughes1986newBeyondSUPG, shakib1991new, bazilevs2007variational}. Suppose the size of element $ K \in \mesh $ is $ h $. Given a convection field $ \avec $ in element $ K $, we can define the \textit{spatiotemporal convection} as
	\begin{align} \label{eq:augmented-advection}
		\widetilde{\avec} := (\avec, 1),
	\end{align}
	where the ``1'' denotes the coefficient of $ \ut $ in \eqref{eq:ns-original-momentum} (convection in time). Then $ \taum $ and $ \tauc $ are calculated as \cite{codina2001stabilized}
%	\begin{subequations} \label{eq:taum-tauc}
%		\begin{align}
%		\taum &= \left[(\widetilde{\avec} \cdot\mvec{G}) \widetilde{\avec}+ C_I \visco^2 (\mvec{G}:\mvec{G}) \right]^{-\halfnice} \\
%		\tauc &= \frac{1}{\taum (\gvec \cdot \gvec)},
%	\end{align}
%	\end{subequations}
	\begin{subequations} \label{eq:taum-tauc-codina}
		\begin{align}
			\taum &= \left[ \frac{c_1\cinv^2\visco}{h^2} + \frac{c_2 |\widetilde{\avec}|_{\infty}}{h} \right]^{-1}, \\
			\tauc &= c_3 \cinv^2\visco + c_4 |\widetilde{\avec}|_{\infty} h,
		\end{align}
	\end{subequations}
%	where the inner products $ \mvec{G}:\mvec{G} $ and $ \gvec\cdot \gvec $ are the defined  as
%	\begin{align*}
%	\mvec{G}:\mvec{G} &= G_{ij}G_{ij}, \\
%	\gvec\cdot \gvec &= g_i g_i,
%	\end{align*}
	where $ c_i > 0,\ i=1,\ 2,\ 3,\ 4 $, and $ C_I $ is a constant value for which the inverse Poincare inequality holds for each element in the mesh \cite{ciarlet2002finite, brenner2007mathematical}.
\end{remark}

\section{Analysis of the variational problem} \label{sec:analysis}
\subsection{Overview} \label{sec:anlys-overview}
In this section, we present an analysis of the boundedness, stability and convergence of the FEM problem \eqref{eq:fem-problem-vms}. For this, we will use the following linearized form of \eqref{eq:fem-problem-vms} where the convection field $ \avec \in \spaceVhd$ is known \textit{a priori} (recall $ \upairh=(\uvech,\ph)\in \spaceVhd\times \spaceQh $):
\begin{align}\label{eq:fem-problem-lin-vms}
\begin{split}
\bformh(\avec; \upairh, \vpairh) &= \lformh(\upairh;\vpairh) \quad \forall \vpairh=(\vvech,\qh) \in \spaceVh \times \spaceQh,
\end{split}
\end{align}
where $ \bformh $ and $ \lformh $ are defined in \eqref{eq:fem-problem-vms-bh-lh}. The analysis in this section closely follows that presented in \cite{zhou1993least}.
%\begin{align}
%\label{eq:fem-problem-vms-lin-bh}
%\begin{split}
%\bformh(\avec; \upairh,\vpairh) &= \bformUnstab(\avec; \upairh,\vpairh) + \innerh{\taum [\opLTAdv[\avec]\uvech+\visco\laplacian\uvech + \grad \ph ]}{[\opLTAdv[\avec]\vvech-\visco\laplacian\vvech +\grad \qh]} + \innerh{\tauc \divergence \uvech}{\divergence \vvech} \\
%&\quad -\innerFTh{\taum [\opLTAdv[\avec]\uvech-\visco\laplacian\uvech + \grad \ph ]}{\vvech},
%\end{split}
%\\
%\label{eq:fem-problem-vms-lin-lh}
%\lformh(\avec; \vpairh) &= \inner{\fvec}{\vvech} + \innerh{\taum \fvec}{[\opLTAdv[\avec]\vvech-\visco\laplacian\vvech +\grad \qh]} - \innerFTh{\taum \fvec}{\vvech}.
%\end{align}

Some preliminaries are summarized in \secref{sec:anlys-prelim}, and the main results are presented in \secref{sec:anlys-results}. A brief overview of the results is as follows:
\begin{enumerate}[\itshape(a)]
	\item We first prove that for a given convection field $ \avec $, the term $ \bformh $ appearing in \eqref{eq:fem-problem-lin-vms} is coercive with respect to the solution $ \uvech $ (\lemref{thm:coercivity}). This also implies that for each given $ \avec $, \eqref{eq:fem-problem-lin-vms} yields a unique solution (\cororef{thm:uniqueness-linearized}).
	\item Next we prove that \eqref{eq:fem-problem-lin-vms} defines a continuous map $ F:\avec\rightarrow \upairh $ (\lemref{thm:map-continuity}). By using Brouwer's fixed point theorem, we then prove that there exists at least one fixed point of \eqref{eq:fem-problem-vms}, such that $ \upairh = F(\upairh) $ (\lemref{thm:existence-nonlinear}).
	\item We then prove that the nonlinear solution, i.e., the fixed point solution of \eqref{eq:fem-problem-vms} is unique (\thmref{thm:uniqueness}).
	\item \label{blah} Finally, we determine the order of convergence of the FEM problem \eqref{eq:fem-problem-vms} in terms of the mesh size $ h $ (\thmref{thm:convergence}).
\end{enumerate}

\subsection{Preliminaries} \label{sec:anlys-prelim}
Before stating the lemmas and theorems, we note a few preliminaries:
\begin{enumerate}[\itshape(I)]
	\item \label{item:discrete-norm} \textit{Discrete norm:} We choose the following discrete norm to prove stability and convergence of the FEM problem \eqref{eq:fem-problem-vms}:
	\begin{align}
	\normVW{\upairh} := \left[ \normFT{\uvech}^2 + \visco\norm{\grad \uvech}^2 + \normh{\taum^{\halfnice}[\opLTAdv[\avec] \uvech + \grad \ph]}^2\right]^{\halfnice}.
	\end{align}
	Note that this norm is dependent on the mesh, the convection field $ \avec $, as well as the choice of the stabilization parameter $ \taum $. The definition of $ \normh{\cdot} $ is given in \eqref{eq:discrete-norms}.

	\item \label{item:convection-op} \textit{The convection operator:} The operators $ \opLAdv[\avec] $ and $ \opLTAdv[\avec] $ defined in \eqref{eq:opL} and \eqref{eq:opLT} are linear in $ \uvec $. For two advection fields $ \avec^1 $ and $ \avec^2 $, we denote the difference by $ \davec = \avec^1-\avec^2 $. Then we have,
	\begin{align*}
	\opLTAdv[\avec^1]\uvec - \opLTAdv[\avec^2]\uvec &= (\ut + \opLAdv[\avec^1]\uvec) - (\ut + \opLAdv[\avec^2]\uvec) \\
	&= \opLAdv[\avec^1]\uvec - \opLAdv[\avec^2]\uvec \\
	&= (\ddiv{\avec^1}) \uvec + \halfnice (\divergence \avec^1) \uvec - (\ddiv{\avec^2}) \uvec - \halfnice (\divergence \avec^2) \uvec \\
	&= (\ddiv{\delta \avec^{12}}) \uvec + \halfnice (\divergence \delta \avec^{12}) \uvec \\
	&= \opLAdv[\delta \avec^{12}],
	\end{align*}
	i.e.,
	\begin{align}
	\opLTAdv[\delta \avec^{12}] =\opLTAdv[\avec^1]\uvec - \opLTAdv[\avec^2]\uvec = \opLAdv[\avec^1]\uvec - \opLAdv[\avec^2]\uvec = \opLAdv[\delta \avec^{12}].
	\end{align}

	\item \label{item:adjoint-op} \textit{Adjoint operators:} The adjoint operators of $ \opLAdv[\avec] $ and $ \opLTAdv[\avec] $ are as follows:
	\begin{subequations}
		\begin{align}
		\opLAdvStar[\avec] &= -\opLAdv[\avec], \\
		\opLTAdvStar[\avec] &= -\opLTAdv[\avec].
		\end{align}
	\end{subequations}
    Indeed, for any $ \eta \in C_c^{\infty}(\stDom) $ (i.e., the space of smooth functions with compact support), we have
	\begin{align*}
	\inner{\opLAdv[\adv]\uvec}{\eta} &= \inner{\left[(\ddiv{\avec}) \uvec + \halfnice (\divergence \avec) \uvec\right]}{\eta} = \inner{\left[(\ddiv{\avec}) \uvec + (\divergence \avec) \uvec\right]}{\eta} - \inner{\halfnice (\divergence \avec)\uvec}{\eta} \\
	&= \inner{\divergence (\uvec\avec^T)}{\eta} - \inner{\halfnice (\divergence \avec)\uvec}{\eta} \\
	&= -\inner{\uvec\avec^T}{\grad\eta} - \inner{\uvec}{\halfnice (\divergence \avec)\eta} \quad \text{[integration by parts on the first term]} \\
	&= -\inner{\uvec}{\left[(\ddiv{\avec}) \vvec + \halfnice (\divergence \avec) \vvec\right]} \\
	&=-\inner{\uvec}{\opLAdv[\avec]\eta}.
	\end{align*}
	And similarly,
	\begin{align}
		\inner{\opLTAdv[\adv]\uvec}{\eta} = \inner{\ut}{\eta} + \inner{\opLAdv[\avec]}{\eta} = -\inner{\uvec}{\partial_t \eta} - \inner{\uvec}{\opLAdv[\avec]\eta} = -\inner{\uvec}{\opLTAdv[\avec]\eta}.
	\end{align}
		
	\item \label{item:tau-max} \textit{Maximum value of} $ \taum $: For a given discretization $ \mesh $, we denote the maximum value of $ \taum $ as
	\begin{align}
	\taumMax = \max_{K \in \mesh} \taum^K.
	\end{align}
	
	\item \label{item:inverse-est} \textit{Inverse estimates:} For any function $ \uvech \in \spaceVh  $, the following inverse estimates hold:
	\begin{subequations}
		\begin{align}
		\normElm{\grad\uvech} &\leq \frac{\cinv}{\he}\normElm{\uvech}, \\
		\normElm{\laplacian\uvech} &\leq \frac{\cinv}{\he}\normElm{\grad\uvech}, \\
		{\norm{\uvech}}_{L^{\infty}(e)} &\leq \frac{\cinv}{(\he)^{\nsd / 2}}\normElm{\uvech},
		\end{align}
	\end{subequations}
	for every $ e \in \mesh $, where $ \he $ is the size of element $ e $, and $ \cinv $ is a constant. Details can be found in Section 3.2 of \cite{ciarlet2002finite}.
	
	\item \label{item:trace-ineq} \textit{Trace inequality:} The norm of a function $ \uvec $ on the boundary can be bounded by the norm inside the domain, i.e.,
	\begin{align}
	\norm{\uvec}_{\partial \stDom} \leq \traceConstant \norm{\uvec}_{\stDom} = \traceConstant \norm{\uvec}.
	\end{align}
	A proof can be found in Section 5.5 of \cite{evans2022partial}.
	
	\item \label{item:shape-regularity} \textit{Shape regularity:} We assume that the mesh $\mesh$ is \emph{shape regular}, i.e., there exists a constant $\shapeRegularConst$ such that
	\begin{align} \label{eq:shape-regularity-condition}
		\frac{\he}{\rho_e} \leq \shapeRegularConst \ \forall e \in \mesh,
	\end{align}
	where $\he$ is the diameter of the element $e$, and $ \rho_e $ is the supremum of the diameters of all the spheres contained in the element $e$.
	
	\item \label{item:interp-sobolev} \textit{Interpolation in a Sobolev space:} Assume $\mesh$ is a given mesh where \eqref{eq:shape-regularity-condition} holds. Let $ k,\ l,\ s \in \mathbb{N} $.And let $\interpolant$ be the interpolation operator from $ \spaceV \cap [H^s(\stDom)]^{\nsd} $ to $\spaceVh$, with order of interpolation $k$. If $k, l, s $ satisfy $0\leq l \leq (k+1) \leq s $, then
	\begin{align} \label{eq:interpolation-sobolev}
	\normH[l]{\vvec-\interp{\vvec}}{\stDom}\leq \approxConstant h^{(k+1-l)}\normH[k+1]{\vvec}{\stDom},
	\end{align}
	where the constant $\approxConstant$ only depends on $k,l,m$ and is independent of $h$ and $v$ (see
	\cite{brenner2007mathematical} and Section 6.7 of \cite{jt1976introduction} for details). The following estimates are direct consequences of \eqref{eq:interpolation-sobolev}.
	\begin{subequations}\label{estimate:sobolev-interp}
		\begin{align}
		\normElm{\vvec-\interp{\vvec}} &\leq {\approxConstant}_0 h^{k+1}\normH[k+1]{\vvec}{e}\ \forall e \in \mesh, \label{estimate:sobolev-interp-L2}\\
		\normElm{\grad (\vvec-\interp{\vvec})} &\leq {\approxConstant}_1 h^{k}\normH[k+1]{\vvec}{e}\ \forall e \in \mesh, \label{estimate:sobolev-interp-grad-L2}\\
		\normElm{\partial_t(\vvec-\interp{\vvec})} &\leq {\approxConstant}_1 h^{k}\normH[k+1]{\vvec}{e}\ \forall e \in \mesh, \label{estimate:sobolev-interp-time-der-L2} \\
		\normElm{\laplacian (\vvec-\interp{\vvec})} &\leq {\approxConstant}_2 h^{k-1}\normH[k+1]{\vvec}{e}\ \forall e \in \mesh. \label{estimate:sobolev-interp-laplacian-L2}
		\end{align}
	\end{subequations}
%	For the final time boundary $ \bft $, we similarly have
%	\begin{align}
%	\normL{\vvec-\interp{\vvec}}{\bft} 
%	&\leq C_{a\Gamma}h^{k+\half}\normH[k+1]{\vvec}{\stDom}.
%	\label{estimate:sobolev-interp-boundary-L2}
%	\end{align}
\end{enumerate}

\subsection{Analysis results} \label{sec:anlys-results}
\begin{assumption} \label{thm:assumptions-c1-c2}
	In \eqref{eq:taum-tauc-codina}, $ c_1 > 2,\ c_2 > 0 $.
\end{assumption}
\begin{assumption} \label{thm:assumptions-hmax}
	Given a mesh $ \mesh $, we assume that the maximum mesh size $ h_{max} < \epsilon c_2 |\avec|_{\infty} $ for some $ \epsilon > 0  $, where $ \epsilon < \frac{2}{\traceConstant} $, and $ \traceConstant $ is defined in \autoref{item:trace-ineq}.
\end{assumption}
\begin{lemma}[\textit{Coercivity}] \label{thm:coercivity}
	If \assumref{thm:assumptions-c1-c2} and \assumref{thm:assumptions-hmax} are satisfied, then the bilinear form $ \bformh $  in \eqref{eq:fem-problem-lin-vms} is coercive, i.e., there exists a positive constant $ \coer $ such that
	\begin{align}
	\bformh(\avec; \upairh, \upairh) \geq \coer \normVW[\avec]{\upairh}^2
	\end{align}
	for all $ \upairh \in \spaceVh\times\spaceQh $.
\end{lemma}
\begin{proof}
	Taking $ \vpairh = \upairh $ in \eqref{eq:fem-problem-lin-vms}, we have
	\begin{multline*}
	\bformh(\avec; \upairh, \upairh) = \halfnice \normFT{\uvech}^2 + \visco\norm{\grad \uvech}^2 + \normh{\taum^{\halfnice}[\opLTAdv[\avec]\uvech + \grad \ph ]}^2 -\normh{\taum^{\halfnice}\visco \laplacian \uvech}^2 \\
	+ \normh{\tauc^{\halfnice}\divergence\uvech}^2 -\innerFTh{\taum [\opLTAdv[\avec]\uvech + \grad \ph ]}{\uvech} + \innerFTh{\taum\visco\laplacian\uvech}{\uvech}.
	\end{multline*}
	The first three terms are positive. The fourth term can be estimated using an inverse inequality as
	\begin{align*}
	-\normh{\taum^{\halfnice}\visco \laplacian \uvech}^2 \geq -\frac{\visco^2 \cinv^2}{h^2} \normh{\taum^{\halfnice} \visco \grad \uvech}^2.
	\end{align*}
	Using Cauchy's inequality $ -2ab \geq -\oneOver{\epsilon}a^2 - \epsilon b^2 $ ($ \epsilon>0 $)  for any $ a $ and $ b $, the sixth term can be estimated as
	\begin{flalign*}
	-\innerFTh{\taum [\opLTAdv[\avec]\uvech + \grad \ph ]}{\uvech} &\geq - \frac{\epsilon}{2} \normFTh{\taum^{\halfnice}[\opLTAdv[\avec]\uvech + \grad \ph ]}^2 - \oneOver{2\epsilon}\normFTh{\taum^{\halfnice} \uvech}^2 \\
	&\geq - \frac{\epsilon \traceConstant}{2} \normh{\taum^{\halfnice}[\opLTAdv[\avec]\uvech + \grad \ph ]}^2 - \oneOver{2\epsilon}\normFTh{\taum^{\halfnice} \uvech}^2.
	\end{flalign*}
	And the seventh term as
	\begin{align*}
	\innerFTh{\taum\visco\laplacian\uvech}{\uvech} &\geq \taumMax\visco \innerFTh{\laplacian\uvech}{\uvech} \geq -\taumMax\visco \normFTh{\grad \uvech}^2 + 0 \geq -\frac{\taumMax\visco \cinv^2}{h^2} \normFTh{\uvech}^2.
	\end{align*}
	Combining,
	\begin{align*}
	\bformh(\avec; \upairh, \upairh)
%	&\geq \half \normFT{\uvech}^2 + \visco\norm{\grad \uvech}^2 + \normh{\taum^{\halfnice}[\opLTAdv[\avec]\uvech + \grad \ph ]}^2 -\normh{\taum^{\halfnice}\visco \laplacian \uvech}^2 + \normh{\tauc^{\halfnice}\divergence\uvech}^2 \\
%	&\hspace{8em} -\frac{\epsilon}{2}\normFTh{\taum^{\halfnice}[\opLTAdv[\avec]\uvech + \grad \ph ]}^2 - \oneOver{2\epsilon}\normFTh{\uvech}^2 -\taumMax\visco \normFTh{\grad \uvech}^2 \\
	&\geq \half \normFT{\uvech}^2 + \visco\norm{\grad \uvech}^2 + \normh{\taum^{\halfnice}[\opLTAdv[\avec]\uvech + \grad \ph ]}^2 -\frac{\visco^2 \cinv^2}{h^2}\normh{\taum^{\halfnice}\visco \grad \uvech}^2 + \normh{\tauc^{\halfnice}\divergence\uvech}^2 \\
	&\hspace{8em} -\frac{\epsilon\traceConstant}{2}\normh{\taum^{\halfnice}[\opLTAdv[\avec]\uvech + \grad \ph ]}^2 - \oneOver{2\epsilon}\normFTh{\taum^{\halfnice}\uvech}^2 -\frac{\taumMax\visco \cinv^2}{h^2} \normFTh{\uvech}^2 \\
	&= \Xsum_{\elemVol\in\mesh}\left[ \left( 1 - \frac{\taum^e \visco \cinv^2}{h^2}\right) \visco \norm{\grad \uvech}_{\elemVol}^2 + \left(1-\frac{\epsilon\traceConstant}{2}\right) \norm{\taum^{\halfnice}[\opLTAdv[\avec]\uvech + \grad \ph ]}_{\elemVol}^2 + \norm{\tauc^{\halfnice}\divergence\uvech}_{\elemVol}^2 \right] \\
	&\quad +\Xsum_{\elemFT\in\bftmesh} \left[ \left\{ \half -\left(\oneOver{2\epsilon} +\frac{\visco\cinv^2}{h^2}\right) \taumMax \right\} \norm{\uvech}_{\elemFT}^2 \right].
	\end{align*}
	The first term in first sum is positive since (using \eqref{eq:taum-tauc-codina} and \assumref{thm:assumptions-c1-c2})
	\begin{align}
	\oneOver{\taum^e} = \frac{c_1\cinv^2\visco}{h^2} + \frac{c_2 |\widetilde{\avec}|_{\infty}}{h} > \frac{\visco \cinv^2}{h^2} \implies \left( 1 - \frac{\taum^e \visco \cinv^2}{h^2} \right) > 0.
	\end{align}
	The second term is positive by \assumref{thm:assumptions-hmax}. And the coefficient of the last term is
	\begin{align}
	\half \left\{ 1 - \left( \oneOver{\epsilon} + \frac{2\visco \cinv^2}{h^2} \right) \taumMax \right\} = \half \left\{ 1 - \frac{\oneOver{\epsilon} + \frac{2\visco \cinv^2}{h^2}}{\frac{c_2 |\widetilde{\avec}|_{\infty}}{h} + \frac{c_1\cinv^2\visco}{h^2}} \right\} = \half \frac{\left( \frac{c_2 |\widetilde{\avec}|_{\infty}}{h} - \oneOver{\epsilon} \right) + \frac{(c_1-2)\cinv^2\visco}{h^2}  }{\frac{c_2 |\widetilde{\avec}|_{\infty}}{h} + \frac{c_1\cinv^2\visco}{h^2}} >0
	\end{align}
	since $ c_1>2 $ and $ 1/\epsilon < c_2 |\avec|_{\infty} / h $ by \assumref{thm:assumptions-c1-c2} and \assumref{thm:assumptions-hmax} respectively. By choosing $ \coer $ as
	\begin{align*}
	\coer = \min \left\{\left( 1 - \frac{\taumMax \visco \cinv^2}{h^2}\right), \left(1-\frac{\epsilon\traceConstant}{2}\right), \left( \half -\left[\oneOver{2\epsilon} +\frac{\visco\cinv^2}{h^2}\right] \taumMax \right)\right\},
	\end{align*}
	we prove the estimate.
%	Each of the quantities in the parentheses must be positive, i.e.,
%	\begin{align}
%	1 - \frac{\taum^e \visco \cinv^2}{h^2} > 0 &\implies \taum^e < \frac{h^2}{\visco \cinv^2}, \label{eq:cond-1} \\
%	1-\frac{\epsilon\traceConstant}{2} >0 &\implies \epsilon < \frac{2}{\traceConstant}, \label{eq:cond-2} \\
%	\text{and}\ \half -\oneOver{2\epsilon}\taumMax -\frac{\taumMax\visco\cinv^2}{h^2} >0 &\implies \taumMax < \oneOver{{\left(\oneOver{\epsilon} + \frac{\cinv^2\visco}{2h^2}\right)}}. \label{eq:cond-3}
%	\end{align}
%	Using \eqref{eq:taum-tauc-codina} in condition \eqref{eq:cond-1}, we have
%	\begin{align}
%		\left[ \frac{c_1\cinv^2\visco}{h^2} + \frac{c_2 |\widetilde{\avec}|_{\infty}}{h} \right]^{-1} < \frac{h^2}{\visco \cinv^2} \implies \frac{c_1\cinv^2\visco}{h^2} + \frac{c_2 |\widetilde{\avec}|_{\infty}}{h} > \frac{\visco \cinv^2}{h^2},
%	\end{align}
%	which is satisfied since $ c_1 > 1 $.
%	Also, using \eqref{eq:taum-tauc-codina} in \eqref{eq:cond-3}, we have
%	\begin{align*}
%		\oneOver{\epsilon} + \frac{\cinv^2\visco}{2h^2} < \frac{c_1\cinv^2\visco}{h^2} + \frac{c_2 |\widetilde{\avec}|_{\infty}}{h} \implies
%		(2c_1 - 1) \cinv^2 \visco + 2 c_2 |\widetilde{\avec}|_{\infty} h - 2h^2\epsilon > 0,
%	\end{align*}
%	which is satisfied since $ c_1 > 1 $ and since $ h < \frac{c_2 |\widetilde{\avec}|_{\infty}}{\epsilon} $.
%	Now, 
\end{proof}

\begin{corollary} \label{thm:uniqueness-linearized}
	Given $ \avec \in \spaceVh $, \eqref{eq:fem-problem-lin-vms} has a unique solution.
\end{corollary}
\begin{proof}
	Suppose \eqref{eq:fem-problem-lin-vms} does not have a unique solution, and assume that given a particular $ \avec \in \spaceVh $, there are two solutions $ \upairh^1 $ and $ \upairh^2 $. Then we have
	\begin{subequations}
		\begin{align}
		\bformh(\avec; \upairh^1, \vpairh) &= \lformh(\avec;\vpairh) \quad \forall \vpairh \in \spaceVh \times \spaceQh \\
		\bformh(\avec; \upairh^2, \vpairh) &= \lformh(\avec;\vpairh) \quad \forall \vpairh \in \spaceVh \times \spaceQh.
		\end{align}
	\end{subequations}
	Subtracting, we have
	\begin{align*}
	\bformh(\avec; \upairh^1-\upairh^2, \vpairh) &= 0 \quad \forall \vpairh \in \spaceVh \times \spaceQh.
	\end{align*}
	Choosing $ \vpairh = \upairh^1-\upairh^2 $, we have
	\begin{align*}
	0 = \bformh(\avec; \upairh^1-\upairh^2, \upairh^1-\upairh^2) \geq \coer \normVW[\avec]{\upairh^1-\upairh^2}^2,
	\end{align*}
	which implies that we must have $ \upairh^1=\upairh^2 $. 
\end{proof}
\begin{lemma} \label{thm:map-continuity}
	The linearized stabilized equation \eqref{eq:fem-problem-lin-vms}
	determines a continuous map $ F: \avec \rightarrow \upairh = F(\avec)$.
\end{lemma}
\begin{proof}
	Taking $ \vpairh = \upairh $ in \eqref{eq:fem-problem-lin-vms}, we have
	\begin{align} \label{eq:lemma-map-equality}
	\bformh(\avec; \upairh, \upairh) &= \lformh(\avec;\upairh).
	\end{align}
	From \lemref{thm:coercivity}, we can write $ \coer \normVW[\avec]{\upairh}^2 \leq \bformh(\avec; \upairh, \upairh)  $. And $ \lformh(\avec; \upairh) $ can be estimated as
	\begin{align*}
	\lformh(\avec; \upairh) &= \inner{\fvec}{\uvech} + \innerh{\taum \fvec}{[\opLTAdv[\avec]\uvech-\visco\laplacian\uvech +\grad \ph]} - \innerFTh{\taum \fvec}{\uvech} \\
	&\leq \norm{\fvec}\norm{\uvech} + \normh{\taum^{\halfnice}\fvec} \normh{\taum^{\halfnice}[\opLTAdv[\avec]\uvech -\visco\laplacian\uvech +\grad \ph]} + \normFTh{\taum^{\halfnice} \fvec} \normFTh{\taum^{\halfnice}\uvech} \\
	&\leq \norm{\oneOver{\visco} \fvec}\norm{\visco \grad \uvech} + \normh{\taum^{\halfnice}\fvec} \left(\normh{\taum^{\halfnice}[\opLTAdv[\avec]\uvech +\grad \ph]} + \visco \normh{\taum^{\halfnice} \laplacian\uvech} \right) + \normFTh{\taum^{\halfnice} \fvec} \normFTh{\taum^{\halfnice}\uvech} \\
	&\leq \left[\norm{\visco^{-1}\fvec}^2 + 2\normh{\taum^{\halfnice}\fvec}^2 + \normFTh{\taum^{\halfnice}\fvec}^2\right]^{\halfnice} \\
	&\qquad \times \left[\visco\norm{\grad \uvech}^2 + \normh{\taum^{\halfnice}[\opLTAdv[\avec]\uvech +\grad \ph]}^2 + \visco \normh{\taum^{\halfnice} \laplacian\uvech}^2 + \normFTh{\taum^{\halfnice}\uvech}^2 \right]^{\halfnice}
	\end{align*}
	or,
	\begin{align*}
	\lformh(\avec; \upairh) &\leq C_b \lambda(f) \normVW[\avec]{\uvech}
	\end{align*}	
	where
	\begin{subequations} \label{eq:cb-and-lambdaF}
		\begin{align}
		C_b &= \max \left\{\visco(1+\taumMax), \taumMax|_{\bft} \right\} \\
		\lambda(f) &= \left[\norm{\visco^{-1}\fvec}^2 + 2\normh{\taum^{\halfnice}\fvec}^2 + \normFTh{\taum^{\halfnice}\fvec}^2\right]^{\halfnice}.
		\end{align}
	\end{subequations}
	Then \eqref{eq:lemma-map-equality} gives
	\begin{align*}
	\coer \normVW[\avec]{\upairh}^2 \leq \bformh(\avec; \upairh, \upairh) &= \lformh(\avec;\upairh) \leq C_b \lambda(f) \normVW[\avec]{\upairh} \\
	\implies \normVW[\avec]{\upairh} &\leq R
	\end{align*}
	where
	\begin{align} \label{eq:def:R}
	R = \left(\frac{C_b}{\coer}\right) \lambda(f) = \left(\frac{C_b}{\coer}\right) \left[\norm{\visco^{-1}\fvec}^2 + 2\normh{\taum^{\halfnice}\fvec}^2 + \normFTh{\taum^{\halfnice}\fvec}^2\right]^{\halfnice}.
	\end{align}
	Now if we define the set 
	\begin{align} \label{eq:def:BR}
	B_R = \setbuilder{\upairh \in \spaceVh\times\spaceQh : \normVW[\avec]{\upairh} \leq R},
	\end{align}
	then it is easy to see that given any $ (\avec, \cdot)\in B_R $, $ F $ maps to $ \upairh $ in $ B_R $, i.e., $ F:B_R\rightarrow B_R $. 
 
 Next we prove that $ F $ is continuous.
	Let $ (\avec^i,\cdot) \in B_R \ (i=1,2) $ and that
	\begin{subequations}
		\label{eq:lemma-map-two-funcs}
		\begin{align}
		\bformh(\avec^1; \upairh^1, \vpairh) &= \lformh(\avec^1;\vpairh) \quad \forall \vpairh \in \spaceVh \times \spaceQh \\
		\bformh(\avec^2; \upairh^2, \vpairh) &= \lformh(\avec^2;\vpairh) \quad \forall \vpairh \in \spaceVh \times \spaceQh
		\end{align}
	\end{subequations}
	Then we also have $ \normVW[\avec^i]{\upairh^i} \leq R $ for $ i = 1,2 $. From \eqref{eq:lemma-map-two-funcs}
	\begin{align*}
	\bformh(\avec^1; \upairh^1, \vpairh) - \bformh(\avec^2; \upairh^2, \vpairh) &= \lformh(\avec^1, \vpairh) - \lformh(\avec^2, \vpairh) \\
	&= \innerh{\taum \fvec}{(\opLTAdv[\avec^1]\vvech - \opLTAdv[\avec^2]\vvech)} \\
	\implies \bformh(\avec^1; \upairh^1, \vpairh) &= \bformh(\avec^2; \upairh^2, \vpairh) + \innerh{\taum \fvec}{\opLTAdv[\delta \avec^{12}]\vvech}
	\end{align*}
	
	Now,
	\begin{align*}
	\bformh(\avec^1; \upairh^1-\upairh^2, \vpairh) &= \bformh(\avec^1; \uvech^1, \vvech) - \bformh(\avec^1; \uvech^2, \vvech) \\
	&= \bformh(\avec^2; \uvech^2, \vvech) - \bformh(\avec^1; \uvech^2, \vvech) + \innerh{\taum \fvec}{\opLTAdv[\delta \avec^{12}]\vvech} \\
	&= -\inner{\uvech^2}{(\opLTAdv[\avec^2]-\opLTAdv[\avec^1])\vvech} \\
	&\quad + \innerh{\taum [\opLTAdv[\avec^2]\uvech^2-\visco\laplacian\uvech^2 + \grad \ph^2 ]}{[\opLTAdv[\avec^2]\vvech-\visco\laplacian\vvech +\grad \qh]} \\
	&\quad- \innerh{\taum [\opLTAdv[\avec^1]\uvech^2-\visco\laplacian\uvech^2 + \grad \ph^2 ]}{[\opLTAdv[\avec^1]\vvech-\visco\laplacian\vvech +\grad \qh]} \\
	&\quad -\innerFTh{\taum (\opLTAdv[\avec^2]\uvech - \opLTAdv[\avec^1] \uvech)}{\vvech} \\
	&\quad + \innerh{\taum \fvec}{\opLTAdv[\delta \avec^{12}]\vvech} \\
	&= -\inner{\uvech^2}{\opLTAdv[\delta \avec^{21}]\vvech} \\
	&\quad + \innerh{\taum \opLTAdv[\delta \avec^{21}]\uvech^2}{[\opLTAdv[\avec^1]\vvech-\visco\laplacian\vvech +\grad \qh]} \\
	&\quad + \innerh{\taum ([\opLTAdv[\avec^2]\uvech^2-\visco\laplacian\uvech^2 +\grad \ph^2]}{\opLTAdv[\delta \avec^{21}] \vvech} \\
	&\quad -\innerFTh{\taum \opLTAdv[\delta \avec^{21}]\uvech^2}{\vvech} \\
	&\quad - \innerh{\taum \fvec}{\opLTAdv[\delta \avec^{21}]\vvech}
	\end{align*}
	
	\begin{align*}
	\bformh(\avec^1; \upairh^1-\upairh^2, \vpairh) &= \bformh(\avec^1; \uvech^1, \vvech) - \bformh(\avec^1; \uvech^2, \vvech) \\
	&= -\inner{\uvech^2}{\opLAdv[\delta \avec^{21}]\vvech} \\
	&\quad + \innerh{\taum \opLAdv[\delta \avec^{21}]\uvech^2}{[\opLTAdv[\avec^1]\vvech-\visco\laplacian\vvech +\grad \qh]} \\
	&\quad + \innerh{\taum ([\opLTAdv[\avec^2]\uvech^2-\visco\laplacian\uvech^2 +\grad \ph^2]}{\opLAdv[\delta \avec^{21}] \vvech} \\
	&\quad -\innerFTh{\taum \opLAdv[\delta \avec^{21}]\uvech^2}{\vvech} \\
	&\quad - \innerh{\taum \fvec}{\opLAdv[\delta \avec^{21}]\vvech} \\
	&=: S_1 + S_2 + S_3 + S_4 + S_5
	\end{align*}
	$ S_5 $	 can be bounded as
	\begin{align*}
	|S_5| \leq \norm{\taum f}_{L^2}\norm{\delta \avec^{21}}_{L^{\infty}}\norm{\grad \vvech}_{L^2}.
	\end{align*}
	Using the discrete inverse inequality and Poincare's inequality, we have
	\begin{align*}
	\norm{\davec}_{L^{\infty}} &\leq Ch^{-\xi} \norm{\davec}_{L^2} \leq Ch^{-\xi} \norm{\grad \davec}_{L^2}
	\end{align*}
	where $ \xi = -1,-\halfnice $ for $ n=2,3 $ respectively. Therefore
	\begin{align*}
	\seminorm{S_5} \leq \cinv h^{-\xi} \norm{\taum f}_{L^2}\norm{\grad \delta \avec^{21}}_{L^2}\norm{\grad \vvech}_{L^2}.
	\end{align*}
	For $ S_1 $ through $ S_4 $, we can similarly obtain
	\begin{align*}
	\seminorm{S_1} &\leq \beta R\norm{\grad \davec}_{L^2}\norm{\grad \vvech}_{L^2} \\
	\seminorm{S_2} &\leq \cinv h^{-\xi}\taumMax \norm{\grad \uvech^2}_{L^2} \norm{\grad \davec}_{L^2} \left( \normh{\taum^{\halfnice} [\opLTAdv[\avec^1]\vvech + \grad \qh]} + \visco \normh{\grad \vvech}  \right) \\
	&\leq \cinv h^{-\xi}\taumMax R \norm{\grad \davec}_{L^2} \left( \normh{\taum^{\halfnice} [\opLTAdv[\avec^1]\vvech + \grad \qh]} + \visco \normh{\grad \vvech}  \right) \\
	\seminorm{S_3} &\leq \cinv h^{-\xi}\taumMax \left( \normh{\taum^{\halfnice} [\opLTAdv[\avec^2]\uvech^2 + \grad \ph^2]} + \visco \normh{\grad \uvech^2} \right) \norm{\grad \davec}_{L^2} \norm{\grad \vvech}_{L^2} \\
	&\leq \cinv h^{-\xi}\taumMax R \norm{\grad \davec}_{L^2} \norm{\grad \vvech}_{L^2} \\
	\seminorm{S_4} &\leq \cinv h^{-\xi_b} \taumMax R \normFT{\grad \davec} \normFT{\grad \vvech} \leq \cinv^2 h^{-\xi_b-1} \taumMax R \norm{\grad \davec} \normFT{\vvech}
	\end{align*}
	where in case of $ S_4 $, we have used $ \normFT{\grad\davec} \leq C\norm{\grad\davec} $ and $ \normFT{\grad \vvech} \leq \frac{\cinv}{h}\normFT{\uvech}$. Thus, combining them, we have
	\begin{align*}
	\bformh(\avec^1; \upairh^1-\upairh^2, \vpairh) &\leq \Xsum_{i=1}^4 \seminorm{S_i} \\
	&\leq \norm{\grad \davec} \Bigg[ C_1(R,h,\taumMax) \normFT{\vvech} + C_2(R,h,\taumMax) \norm{\grad\vvech} + C_3(R,h,\taumMax) \normh{\taum^{\halfnice} [\opLTAdv[\avec^1]\vvech + \grad \qh]} \Bigg] \\
	&\leq \norm{\grad \davec} \left[C_1^2 + C_2^2 + C_3^2 \right]^{\halfnice} \left[\normFT{\vvech}^2 + \norm{\grad\vvech}^2 + \normh{\taum^{\halfnice} [\opLTAdv[\avec^1]\vvech + \grad \qh]}^2 \right]^{\halfnice} \\
	&=: \tilde{C} \normVW[\avec^1]{\vpairh} \norm{\grad \davec}
	\end{align*}
	where $ \tilde{C} = \left[C_1^2 + C_2^2 + C_3^2 \right]^{\halfnice} $. Now, if we choose $ \vpairh = \upairh^1 - \upairh^2 $, then
	\begin{align*}
	\bformh(\avec^1; \upairh^1-\upairh^2, \upairh^1-\upairh^2) &\leq \tilde{C}\normVW[\avec^1]{\upairh^1-\upairh^2} \norm{\grad \davec}
	\end{align*}
	By means of coercivity, we also have
	\begin{align*}
	\bformh(\avec^1; \upairh^1-\upairh^2, \upairh^1-\upairh^2) \geq \coer \normVW[\avec^1]{\upairh^1-\upairh^2}^2.
	\end{align*}
	So, we have
	\begin{align*}
	\coer \normVW[\avec^1]{\upairh^1-\upairh^2}^2 &\leq \tilde{C}\normVW[\avec^1]{\upairh^1-\upairh^2} \norm{\grad \davec} \\
	\implies \normVW[\avec^1]{\upairh^1-\upairh^2} &\leq \left(\frac{\tilde{C}}{\coer}\right) \norm{\grad(\avec^1-\avec^2)},
	\end{align*}
	which proves that $ \upairh $ continuously depends on $ \avec $, thus $ F $ is a continuous map.
\end{proof}

\begin{lemma}[\textit{Existence}] \label{thm:existence-nonlinear}
	If $ \fvec \in [L^2(\stDom )]^{\nsd} $, then the nonlinear problem \eqref{eq:fem-problem-vms} has at least one solution $ \upairh = (\uvech, \ph) $.
\end{lemma}
\begin{proof}
	By \lemref{thm:map-continuity}, we know that the linear map $ \upairh = F(\avec) $ corresponding to \eqref{eq:fem-problem-lin-vms} is a continuous map from $ B_R $ to $ B_R $. Therefore by Brouwer's fixed point theorem, $ F $ must have at least one fixed point $ \upairh = F(\uvech) $. Therefore \eqref{eq:fem-problem-vms} has at least one solution.
\end{proof}

%\subsection{Uniqueness of Solution} \label{sec:uniqueness}
\begin{theorem}[\textit{Uniqueness}] \label{thm:uniqueness}
	The nonlinear problem \eqref{eq:fem-problem-vms} has a unique solution $ \upairh = (\uvech, \ph) $ that satisfies \begin{align}
	\normVW[\uvech]{\upairh} \leq R,
	\end{align}
	where $ R = \left(\nicefrac{C_b}{\coer}\right) \lambda(f) $; $ C_b,\ \lambda(f) $ are defined in \eqref{eq:cb-and-lambdaF}, and $ \coer $ is the constant of stability (defined in \lemref{thm:coercivity}).
\end{theorem}
\begin{proof}
	By \lemref{thm:existence-nonlinear},  \eqref{eq:fem-problem-vms} has at least one solution $ \upairh = (\uvech, \ph)$. So, we have
	\begin{align}
		\bformh(\uvech; \upairh, \vpairh) &= \lformh(\uvech;\vpairh) \quad \forall \vpairh \in \spaceVh \times \spaceQh.
	\end{align}
	Choosing $ \vpairh = \upairh $ in the above equation, and following a similar process as in the proof of \lemref{thm:map-continuity}, we find that
	\begin{align*}
		\coer \normVW[\uvech]{\upairh}^2 \leq \bformh(\uvech; \upairh, \upairh) &= \lformh(\uvech;\upairh) \leq C_b \lambda(f) \normVW[\uvech]{\upairh} \\
		\implies \normVW[\uvech]{\upairh} &\leq \left(\nicefrac{C_b}{\coer}\right) \lambda(f) = R
	\end{align*}
	To prove that $ \upairh $ is unique, let us assume that $ \upairh^1 = (\uvech^1,\ph^1) $ and $ \upairh^2  = (\uvech^2,\ph^2) $ are two different solutions of \eqref{eq:fem-problem-vms}, i.e.,
		\begin{subequations}
		\label{eq:lemma-uniqueness-two-funcs}
		\begin{align}
			\bformh(\uvech^1; \upairh^1, \vpairh) &= \lformh(\uvech^1;\vpairh) \quad \forall \vpairh \in \spaceVh \times \spaceQh \\
			\bformh(\uvech^2; \upairh^2, \vpairh) &= \lformh(\uvech^2;\vpairh) \quad \forall \vpairh \in \spaceVh \times \spaceQh.
		\end{align}
	\end{subequations}
	Suppose $ \eta_h = \upairh^1-\upairh^2 = (\uvech^1-\uvech^2, \ph^1-\ph^2) = (\wvech,\rh)$. Once again, following a similar process as in the proof of \lemref{thm:map-continuity}, we have
	\begin{align*}
		\bformh(\uvec^1; \eta_h, \eta_h) &= \bformh(\uvec^1; \upairh^1-\upairh^2, \eta_h) \\
		 &= \bformh(\uvec^1; \upairh^1, \wvech) - \bformh(\uvec^1; \upairh^2, \wvech) \\
		&= -\inner{\uvech^2}{\opLAdv[\wvech]\wvech} \\
		&\quad + \innerh{\taum \opLAdv[\wvech]\uvech^2}{[\opLTAdv[\uvec^1]\wvech-\visco\laplacian\wvech +\grad \rh]} \\
		&\quad + \innerh{\taum ([\opLTAdv[\uvec^2]\uvech^2-\visco\laplacian\uvech^2 +\grad \rh^2]}{\opLAdv[\wvech] \wvech} \\
		&\quad -\innerFTh{\taum \opLAdv[\wvech]\uvech^2}{\wvech} \\
		&\quad - \innerh{\taum \fvec}{\opLAdv[\wvech]\wvech} \\
		&=: T_1 + T_2 + T_3 + T_4 + T_5
	\end{align*}
	Using \lemref{thm:coercivity}, we have
	\begin{align}
	\coer \normVW[\uvec^1]{\eta_h}^2 \leq \bformh(\uvec^1; \eta_h, \eta_h) = T_1 + T_2 + T_3 + T_4 + T_5 \\
	\label{eq:lemma-uniqueness-e2}
	\implies \coer \left( \normFT{\wvech}^2 + \norm{\grad \wvech}^2 + \normh{\taum^{\halfnice}[\opLTAdv[\uvec^1] \wvech + \grad \rh]}^2\right) - \Xsum_{j=1}^5 T_j \leq 0.
	\end{align}
	Now, we estimate each $ T_j,\ j=1,\ldots,5 $ as follows.
	\begin{align*}
		\seminorm{T_1} &\leq \beta R \norm{\grad \wvech}^2, \\
		\seminorm{T_2} &\leq \cinv h^{-\xi} \taumMax R \left(\norm{\grad \wvech} \normh{\taum^{\halfnice}[\opLTAdv[\uvech^1]\wvech + \grad \rh]} + \visco\norm{\grad \wvech}^2\right) \\
		&\leq \cinv h^{-\xi} \taumMax R \left[\left(\visco + \half\right) \norm{\grad \wvech}^2 + \half \normh{\taum^{\halfnice}[\opLTAdv[\uvech^1]\wvech + \grad \rh]}^2 \right], \\
		\seminorm{T_3} &\leq \cinv h^{-\xi} \taumMax R \norm{\grad \wvech}^2, \\
		\seminorm{T_4} &\leq \cinv h^{-\xi_b} \taumMax R \normFT{\grad \wvech}^2 \\ 
		&\leq \cinv^2 h^{-\xi_b-1} \taumMax R \normFT{\wvech}^2, \\
		\seminorm{T_5} &\leq \cinv h^{-\xi} \taumMax \norm{\fvec} \norm{\grad\wvech}^2.
	\end{align*}
	Substituting each of these five estimates into \eqref{eq:lemma-uniqueness-e2}, and collating the similar terms, we have
	\begin{align} \label{eq:lemma-uniqueness-e3}
		\lambda_1 \normFT{\wvech}^2
		+\lambda_2 \normh{\taum^{\halfnice}\left[\opLTAdv[\uvech^1] \wvech + \grad \rh \right]}^2 + \lambda_3 \norm{\grad \wvech}^2 \leq 0,
	\end{align}
	where
	\begin{align*}
		\lambda_1(h) &= \coer - \cinv^2 h^{-\xi_b-1} \taumMax R, \\
		\lambda_2(h) &= \coer - \half \cinv^2 h^{-2\xi}  {\taumMax}^2 R^2, \\
		\lambda_3(h) &= \coer - \left[\beta R + \cinv h^{-\xi}\taumMax \left\{R \left( \nicefrac{3}{2} + \visco \right) + \norm{\fvec} \right\} + \half \cinv^2 h^{-2\xi} {\taumMax}^2 R^2 \right].
	\end{align*}
    There exists a $ \hzero $ such that for each $ h \leq \hzero $, we have $\lambda_i(h) \geq 0,\ i=1,2,3$. Then each of the norms appearing on the left hand side of \eqref{eq:lemma-uniqueness-e3} must amount to 0. Thus we must have $ \wvech = \mathbf{0} $, and also $ \rh = 0 $. Therefore, we have that  $ \upairh^1 = \upairh^2 $, proving the uniqueness of the solution.
\end{proof}

\begin{theorem}[\textit{Convergence}] \label{thm:convergence}
	Assume that the exact solution of \eqref{eq:fem-problem-vms} is $ \phi = (\uvec, p) \in (H^{k+1}(\stDom )\cap \spaceV) \times (H^{k+1}(\stDom )\cap\spaceQ),\ k \in \mathbb{N} $. Then for the discrete solution $ \upairh $, we have the following estimate:
	\begin{align}
	\normVW[\uvec]{\upairh - \upair} \leq C h^{k},
	\end{align}
	where $ C $ depends on $ \norm{\uvec}_{H^{k+1}} $ and $ \norm{p}_{H^{k+1}} $, and $ k $ is the order of the basis functions.
\end{theorem}
\begin{proof}
	Suppose $ \interpolant^u: \spaceV \cap H^{k+1}(\stDom) \rightarrow \spaceVh  $ is the projection of $ \uvec $ from $\spaceV$ to $\spaceVh$; and similarly $ \interpolant^p: \spaceQ \cap H^{k+1}(\stDom ) \rightarrow \spaceQh $. We assume that both $ \interpolant^u $ and $ \interpolant^p $ have the degree of interpolation $ k \in \mathbb{N} $.
	
	Let $ \wvech = \uvech - \interpolant\uvech $ and $ \rh = \ph - \interpolant\ph $, and $ \hat{\eta} = (\wvech, \rh)$. Consequently $ \hat{\eta} = \upairh - \interpolant\upairh $. Using \eqref{eq:fem-problem-lin-vms},
	\begin{align*}
		\bformh(\uvec; \hat{\eta}, \hat{\eta}) &= \bformh(\uvec; \upairh - \interp{\upair}, \hat{\eta}) \\
		&= \bformh(\uvec; \upairh, \hat{\eta}) - \bformh(\uvec; \interp{\upair}, \hat{\eta}) \\
		&= \bformh(\uvec; \upair, \hat{\eta}) - \bformh(\uvec; \interp{\upair}, \hat{\eta}) \\
		&= \bformh(\uvec; \upair-\interp{\upair}, \hat{\eta}) \\
		&= -\inner{\duh}{\opLTAdv[\uvec] \wvech} + \visco \inner{\grad (\duh)}{\grad \wvech} - \inner{\dph}{\divergence\wvec} + \inner{\divergence (\duh)}{\rh} \\
		&\quad + \innerFT{\duh}{\wvech} + \innerh{\taum [\opLTAdv[\uvec] (\duh) -\visco \laplacian (\duh) + \grad (\duh)] }{[\opLTAdv[\uvec]-\visco\laplacian\wvech + \grad\rh]} \\
		&\quad \innerh{\tauc \divergence(\duh)}{\divergence \wvech} - \innerFTh{\taum [\opLTAdv[\uvec](\duh)-\visco\laplacian(\duh)+\grad (\dph)]}{\wvech} \\
		&= \Xsum_{i=1}^{8} S_i
	\end{align*}
	Denote $ \alpha := k+1 $, $ \eeu := \normH[k+1]{\uvec}{e} $ and $ \eep := \normH[k+1]{p}{e} $. Using the estimates in \eqref{estimate:sobolev-interp} we have
	\begin{subequations}
		\begin{align}
		\normElm{\duh} &\leq C\he^{\alpha}\eeu, \\
		\normElm{\grad \duh} &\leq C \he^{\alpha-1}\eeu, \\
		\normElm{\partial_t \duh} &\leq C \he^{\alpha-1}\eeu, \\
		\normElm{\laplacian \duh} &\leq C\he^{\alpha-2}\eeu, \\
		\normElm{\dph} &\leq C\he^{\alpha}\eep, \\
		\normElm{\grad \dph} &\leq C\he^{\alpha-1} \eep.
	\end{align}
	\end{subequations}
	Then,
	\begin{align*}
	S_1 = -\inner{\duh}{\opLTAdv[\uvec] \wvech} =\elmsum -\inner{\duh}{\opLTAdv[\uvec] \wvech}_e &= \elmsum \beta \normElm{\grad \duh}\normElm{\grad \uvec}\normElm{\grad \wvech} \\
	&\leq \elmsum \beta R C \he^{\alpha-1} \eeu \normElm{\grad\wvech}
	\end{align*}
	\begin{align*}
	S_2 = \visco \inner{\grad (\duh)}{\grad \wvech} = \elmsum \visco \inner{\grad (\duh)}{\grad \wvech}_e &\leq \elmsum \visco \normElm{\grad \duh}\normElm{\grad \wvech} \\
	&\leq \elmsum \visco C\he^{\alpha-1}\eeu \normElm{\grad \wvech}
	\end{align*}
	\begin{align*}
	S_3 =  - \inner{\dph}{\divergence\wvec} = \elmsum - \inner{\dph}{\divergence\wvec}_e \leq \elmsum \normElm{\dph}\normElm{\divergence\wvech} &\leq \elmsum \sqrt{\nsd}\normElm{\dph}\normElm{\grad \wvech} \\
	&\leq \elmsum \sqrt{\nsd}C\he^{\alpha}\eep\normElm{\grad \wvech}
	\end{align*}
	\begin{align*}
	S_4 &=  \inner{\divergence (\duh)}{\rh} = \inner{\divergence (\eie{\uvec})}{\rh} = \inner{\divergence \uvec - \divergence \interp{\uvec}}{\rh} = 0,
	\end{align*}
	since $ \divergence\uvec =0 $ and $ \divergence\interp{\uvec} = 0$.
	\begin{align*}
	S_5 &= \innerFT{\duh}{\wvech} = \elmsumFT {\inner{\duh}{\wvech}}_e \leq \elmsumFT \normElm{\duh} \normElm{\wvech} \leq \elmsumFT C\he^{\alpha}\eeu \norm{\wvech}_e 
	\end{align*}
	\begin{align*}
	S_6 &= \innerh{\taum [\opLTAdv[\uvec] (\duh) -\visco \laplacian (\duh) + \grad (\duh)] }{[\opLTAdv[\uvec]-\visco\laplacian\wvech + \grad\rh]} \\
	&= \elmsum \inner{\taum [\opLTAdv[\uvec] (\duh) -\visco \laplacian (\duh) + \grad (\duh)] }{[\opLTAdv[\uvec]-\visco\laplacian\wvech + \grad\rh]}_e \\
	&\leq \elmsum \taum^e \left(CR\he^{\alpha-1}\eeu + \visco C \he^{\alpha-2}\eeu + C\he^{\alpha}\eep \right) \left[\normElm{\opLTAdv[\uvec]\wvech + \grad\rh} + \frac{\visco \cinv}{\he} \normElm{\grad \wvech}\right]
	\end{align*}
	\begin{align*}
	S_7 = \innerh{\tauc \divergence(\duh)}{\divergence \wvech} &= \elmsum \inner{\tauc \divergence(\duh)}{\divergence \wvech}_e \leq \elmsum \tauc^e\normElm{\divergence\duh}\normElm{\divergence\wvech} \\
	&\leq \elmsum \tauc^e\sqrt{\nsd}\normElm{\grad\duh} \sqrt{\nsd}\normElm{\grad\wvech} \leq \elmsum \tauc^e \nsd C\he^{\alpha-1}\eeu \normElm{\grad \wvech} 
	\end{align*}
	\begin{align*}
	S_8 &= \innerFTh{\taum [\opLTAdv[\uvec](\duh)-\visco\laplacian(\duh)+\grad (\dph)]}{\wvech} \\
	&= \elmsumFT \inner{\taum [\opLTAdv[\uvec](\duh)-\visco\laplacian(\duh)+\grad (\dph)]}{\wvech}_e \\
	&\leq \elmsumFT \taum^e \left[ C\he^{\alpha-1}\eeuFT  + \visco C\he^{\alpha-2}\eeuFT + C\he^{\alpha}\eepFT \right]\normElm{\wvech}
	\end{align*}
	Combining all terms, we have
	\begin{align*}
	\Xsum_{i=1}^{8} S_i &= \elmsumFT \lambda_1^e \normElm{\wvech} + \elmsum \lambda_2^e \normElm{\grad \wvech} + \elmsum \lambda_3^e \normElm{\opLTAdv[\uvec]\wvech + \grad\rh} \\
	&\leq \left[ \elmsumFT (\lambda_1^e)^2 + \elmsum (\lambda_2^e)^2 + \elmsum (\lambda_3^e)^2 \right]^{\halfnice} \times \left[ \elmsumFT \normElm{\wvech}^2 + \elmsum \normElm{\grad \wvech}^2 + \elmsum \normElm{\opLTAdv[\uvec]\wvech + \grad\rh}^2 \right]^{\halfnice} \\
	&= \lambda \normVW[\uvec]{(\wvech,\rh)} = \lambda \normVW[\uvech]{\eta},
	\end{align*}
	where
	\begin{align*}
	\lambda_1^e &= \he^{\alpha-2}\left[C\he^2\eeuFT + \taum^e \left(C\he\eeuFT + \visco C \eeuFT + C\he^2\eepFT \right)\right], \\
	\lambda_2^e &= \he^{\alpha-2} \big[\beta R C\he^2 \eeu + \visco C\he\eeu + C\sqrt{\nsd}\he^2\eep + C\taucMax\nsd \he \eeu \\ & \qquad\qquad\qquad\qquad\qquad\qquad + \taum^e\left( CR\eeu + \visco C\visco\he^{-1}\eeu + C\he\visco\cinv \eep \right) \big], \\
	\lambda_3^e &= \he^{\alpha-2}\taum^e \left[CR\he\eeu + \visco C\eeu + C\he^2\eep \right], \\
	\text{and}\ \lambda &= \left[ \elmsumFT (\lambda_1^e)^2 + \elmsum (\lambda_2^e)^2 + \elmsum (\lambda_3^e)^2 \right]^{\halfnice} .
	\end{align*}
	We have $ \taum^e \sim \he^2 $, therefore we have
	\begin{align*}
		\lambda_1 &= C_1 h^{\alpha}, \\
		\lambda_2 &= C_2 h^{\alpha-1}, \\
		\lambda_3 &= C_3 h^{\alpha},
	\end{align*}
	where $ h=\max_{e \in \mesh} \he $.
	Using coercivity of the bilinear form, we have
	\begin{align*}
	\coer\normVW[\uvec]{\hat{\eta}}^2 \leq \bformh(\uvec; \hat{\eta}, \hat{\eta}) \leq \lambda \normVW[\uvec]{\hat{\eta}},
	\end{align*}
	therefore
	\begin{align}
	\normVW[\uvec]{\hat{\eta}} \leq \coer^{-1}\lambda = Ch^{\alpha-1}.
	\end{align}
	By triangle inequality,
	\begin{align*}
		\normVW[\uvec]{\upairh - \upair} &\leq \normVW[\uvec]{\upairh - \interpolant\upair} + \normVW[\uvec]{\interpolant \upair - \upair} \\
		&\leq C h^{\alpha-1} + C h^{\alpha-1}\\
		&= C h^{k}.
	\end{align*}
\end{proof}

\subsection{The case of nonzero initial and boundary conditions} \label{sec:nonzero-icbc}
In \eqref{eq:ns-original-dirichlet} and the preceding analysis, we considered zero boundary conditions. Suppose instead, we have the following set of equations:
\begin{subequations}\label{eq:ns-original-D}
	\begin{align}
		\opLTAdv[\uvecTilde]\uvecTilde -\visco\laplacian \uvecTilde+\grad p - 
		\fvecTilde &= \zerovec, \text{ }\xvecst\in\stDom, \label{eq:ns-original-momentum-D}
		\\
		\grad\cdot\uvecTilde &= 0, \text{ }\xvecst\in\stDom, \label{eq:ns-original-continuity-D}
		\\
		\uvecTilde &= \gvec, \text{ }\xvecst\in\bsp, \label{eq:ns-original-dirichlet-D}
		%	\\
		%	\mathscr{B}_{NR} \uvec &= h, \text{ }\xvecst\in\Gamma_{S2}\subseteq\bsp, \label{eq:ns-original-neumann}
		\\
		\uvecTilde &= \uvec_0, \text{ }\xvecst\in\bit, \label{eq:ns-original-ic-D},
	\end{align}
\end{subequations}
where \eqref{eq:ns-original-D} differs from \eqref{eq:ns-original} only on the boundary conditions (i.e. \eqref{eq:ns-original-dirichlet} and \eqref{eq:ns-original-dirichlet-D}). In this case, we proceed as follows.
Consider the following function spaces
\begin{subequations}
	\begin{align}
	\spaceVd_g &= \left\{ \mvec{v} \in \mvec{H}^1(\stDom):\ \mvec{v} = \mvec{g} \text{ on }\bsp, \text{ } \mvec{v} = \uvec_0 \text{ on } \bit \right\}, \\
	\spaceHdiv &= \setbuilder{\vvec \in \mvec{H}^1(\stDom):\ \divergence\vvec = 0}.
	\end{align}
\end{subequations}
Suppose $ (\tilde{\uvec}, p) $ is a solution of \eqref{eq:ns-original-D}. Now, let $ \uproxy \in \spaceVd_g \cap \spaceHdiv $, such that $ \text{Tr}(\uproxy) = \gvec $. Then the difference $ \uvec := \tilde{\uvec} - \uproxy \in \spaceVd $. Substituting $ \tilde{\uvec} = \uvec + \uproxy $ in \eqref{eq:ns-original}, and using the $ \opLTAdv $ notation (see \eqref{eq:def-la-ma}), we have
\begin{subequations}\label{eq:ns-transformed}
	\begin{align}
	\opLTAdv[\uvec+\uproxy]\uvec -\visco\laplacian\uvec + \grad p - \left[\fvecTilde - (\opLTAdv[\uvec+\uproxy]\uproxy - \visco\laplacian\uproxy)\right] &= \zerovec, \ \xvecst\in\stDom, \label{eq:ns-transformed-momentum}
	\\
	\grad\cdot\uvec &= 0, \ \xvecst\in\stDom, \label{eq:ns-transformed-continuity}
	\\
	\uvec &= 0, \ \xvecst\in\bsp \cup \bit. \label{eq:ns-transformed-dirichlet}
	\end{align}
\end{subequations}
where $ \uproxy $ is assumed to be sufficiently smooth. For a known advection field $ \avec \in \spaceVd $, the momentum equations can be written as
\begin{align}
	\opLTAdv[\avec+\uproxy]\uvec -\visco\laplacian\uvec + \grad p - \fvec &= \zerovec, \ \xvecst\in\stDom,
\end{align}
where $ \fvec := \left[\tilde{\fvec} - (\opLTAdv[\avec+\uproxy]\uproxy - \visco\laplacian\uproxy)\right] $. Since $ \uproxy $ is a chosen function, \eqref{eq:ns-transformed} can be solved for $ \uvec $ using a Galerkin formulation, and then added to $ \uproxy $ to obtain the solution $ \tilde{\uvec} $ of \eqref{eq:ns-original}. With this reformulation, we can proceed in the same way as \secref{sec:anlys-results} and obtain similar error estimates. 

%%For all the analytical results in the sequel, we will assume that an appropriate function $ \uproxy $ is available, and that \eqref{eq:ns-original} has been transformed into the form of \eqref{eq:ns-transformed}.\input{sections/s_analysis}

\section{Implementation Details}
\label{sec:implementation}
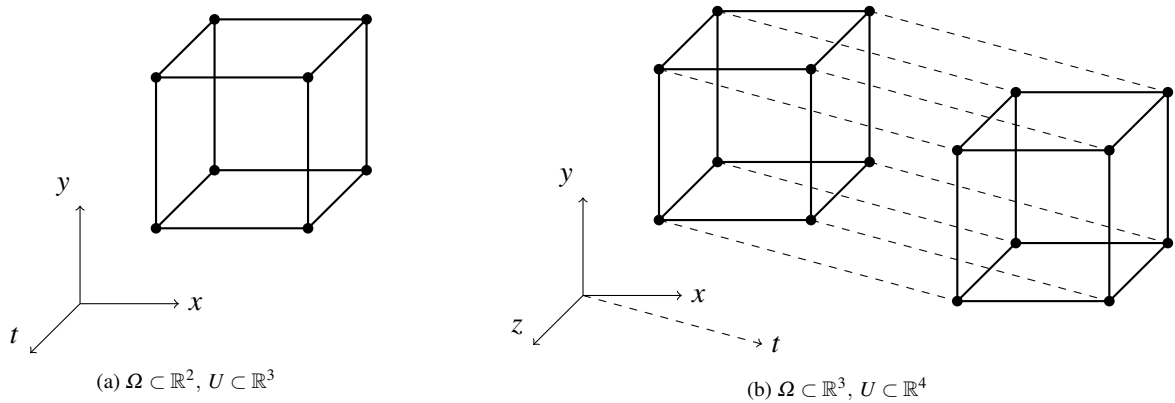
\begin{figure}[!t]
	\centering
	\begin{minipage}{0.3\textwidth}
		\centering
		\begin{tikzpicture}
		% Define coordinates for the corners of the cube
		\coordinate (A) at (0,0,0);
		\coordinate (B) at (2,0,0);
		\coordinate (C) at (2,2,0);
		\coordinate (D) at (0,2,0);
		\coordinate (E) at (0,0,2);
		\coordinate (F) at (2,0,2);
		\coordinate (G) at (2,2,2);
		\coordinate (H) at (0,2,2);
		
		% Draw the edges of the cube
		\draw[thick] (A) -- (B) -- (C) -- (D) -- cycle;
		\draw[thick] (A) -- (E) -- (F) -- (B);
		\draw[thick] (E) -- (H) -- (G) -- (F);
		\draw[thick] (H) -- (D);
		\draw[thick] (C) -- (G);
		
		% Draw markers at the corners
		\foreach \point in {A, B, C, D, E, F, G, H} {
			\fill (\point) circle[radius=2pt];
		}
		
		% Draw x and y axis
		\draw[->] (-1,-1,2) -- (0.3,-1,2) node[anchor=west] {$x$};
		\draw[->] (-1,-1,2) -- (-1,0.3,2) node[anchor=south east] {$y$};
		\draw[->] (-1,-1,2) -- (-1,-1,3.7) node[anchor=south east] {$t$};
		
		%	 % Draw ticks on x axis
		%	 \foreach \x in {-2,-1,1,2}
		%	 \draw (\x,2pt) -- (\x,-2pt) node[anchor=north] {\x};
		%	 
		%	 % Draw ticks on y axis
		%	 \foreach \y in {-2,-1,1,2}
		%	 \draw (2pt,\y) -- (-2pt,\y) node[anchor=east] {\y};
		%	 
		%	 % Label origin
		%	 \node[anchor=north east] at (0,0) {0};
		
		\end{tikzpicture}
		\subcaption{$\spDom \subset \mathbb{R}^2,\ {\stDom} \subset \mathbb{R}^3$}
		\label{fig:q1-3d}
	\end{minipage}
	\hspace{3mm}
	\begin{minipage}{0.65\textwidth}
		\centering
		\begin{tikzpicture}
		\def \hx{+2}
		\def \hy{-3}
		\def \hz{-5}
		\def \axscale{0.6}
		
		% Define coordinates for the corners of the cube
		\coordinate (A) at (0,0,0);
		\coordinate (B) at (2,0,0);
		\coordinate (C) at (2,2,0);
		\coordinate (D) at (0,2,0);
		\coordinate (E) at (0,0,2);
		\coordinate (F) at (2,0,2);
		\coordinate (G) at (2,2,2);
		\coordinate (H) at (0,2,2);
		
		% Draw the edges of the cube
		\draw[thick] (A) -- (B) -- (C) -- (D) -- cycle;
		\draw[thick] (A) -- (E) -- (F) -- (B);
		\draw[thick] (E) -- (H) -- (G) -- (F);
		\draw[thick] (H) -- (D);
		\draw[thick] (C) -- (G);
		
		% Draw markers at the corners
		\foreach \point in {A, B, C, D, E, F, G, H} {
			\fill (\point) circle[radius=2pt];
		}
		
		\coordinate (I) at (0+\hx,0+\hy,0+\hz);
		\coordinate (J) at (2+\hx,0+\hy,0+\hz);
		\coordinate (K) at (2+\hx,2+\hy,0+\hz);
		\coordinate (L) at (0+\hx,2+\hy,0+\hz);
		\coordinate (M) at (0+\hx,0+\hy,2+\hz);
		\coordinate (N) at (2+\hx,0+\hy,2+\hz);
		\coordinate (O) at (2+\hx,2+\hy,2+\hz);
		\coordinate (P) at (0+\hx,2+\hy,2+\hz);
		
		% Lraw the edges of the cube
		\draw[thick] (I) -- (J) -- (K) -- (L) -- cycle;
		\draw[thick] (I) -- (M) -- (N) -- (J);
		\draw[thick] (M) -- (P) -- (O) -- (N);
		\draw[thick] (P) -- (L);
		\draw[thick] (K) -- (O);
		
		% Lraw markers at the corners
		\foreach \point in {I, J, K, L, M, N, O, P} {
			\fill (\point) circle[radius=2pt];
		}
		
		\draw[dashed] (A) -- (I);
		\draw[dashed] (B) -- (J);
		\draw[dashed] (C) -- (K);
		\draw[dashed] (D) -- (L);
		\draw[dashed] (E) -- (M);
		\draw[dashed] (F) -- (N);
		\draw[dashed] (G) -- (O);
		\draw[dashed] (H) -- (P);
		
		% Draw x and y axis
		\draw[->] (-1,-1,2) -- (0.3,-1,2) node[anchor=west] {$x$};
		\draw[->] (-1,-1,2) -- (-1,0.3,2) node[anchor=south east] {$y$};
		\draw[->] (-1,-1,2) -- (-1,-1,3.7) node[anchor=south east] {$z$};
		\draw[->, dashed] (-1,-1,2) -- (-1+\axscale*\hx,-1+\axscale*\hy,2+\axscale*\hz) node[anchor=west] {$t$};
		
		\end{tikzpicture}
		\subcaption{$\spDom \subset \mathbb{R}^3,\ {\stDom} \subset \mathbb{R}^4$}
		\label{fig:q1-4d}
	\end{minipage}
	\caption{$Q_1$ reference elements in 3D (left) and 4D (right). $ Q_2 $ elements are constructed similarly by adding an extra DOF at the center of the lines in each direction, and then taking tensor product.}
	\label{fig:q1-spacetime-elements}
\end{figure}
The finite element formulation obtained in \eqref{eq:fem-problem-vms} can be readily coded in any existing FEM codebase that supports continuous Galerkin method. In the following section, we use $ H^1 $ finite element spaces based on globally $ C^0 $  Lagrangian basis functions. Schematic diagrams of $ Q_1 $ reference elements in 3D and 4D are shown in \figref{fig:q1-spacetime-elements}. For spatially 2D problems ($ \spDom \subset \mathbb{R}^2 $), the space-time domain is 3-dimensional (i.e., $ \stDom \subset \mathbb{R}^3 $) where the third dimension serves as time. For spatially 3D problems ($ \spDom \subset \mathbb{R}^3 $, $ \stDom \subset \mathbb{R}^4 $), the fourth dimension stands for time. The $ Q_2 $ element in 3D and 4D are obtained similarly, with additional degrees of freedoms.

The inner products in \eqref{eq:fem-problem-vms} are calculated by Gaussian quadrature rules. The inclusion of time in the FEM approximation results in a $ (d+1) $-dimensional mesh. The Gaussain quadrature points  in higher dimensions are obtained simply by taking Cartesian-product of one-dimensional quadrature points. Since this number grows exponentially with respect to the dimension, the computational effort to evaluate the elemental stiffness matrix consequently increases. Specifically, this makes the process of evaluation of the quantities \eqref{eq:fem-problem-vms-bh} and \eqref{eq:fem-problem-vms-lh} expensive. This is ideally resolved by domain decomposition based parallelization approaches, described next.  %, essentially by letting different processors perform the Gaussian integration on different parts of the space-time domain. An ideal parallelization scheme will be able to provide a perfect balance between the extra time spent per element and the number of elements that each processor gets work on.

\subsection{Space-time domain decomposition and parallelization}
\begin{figure}[!tb]
	\centering
	\begin{minipage}[t]{0.1\textwidth}
		{\includegraphics[trim=150 30 495 300, clip, width=\textwidth] {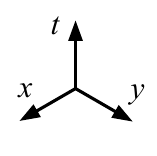}}
	\end{minipage}
	\hspace{2em}
	\begin{minipage}[t]{0.3\textwidth}
		{\includegraphics[trim=200 20 650 300, clip, width=\textwidth] {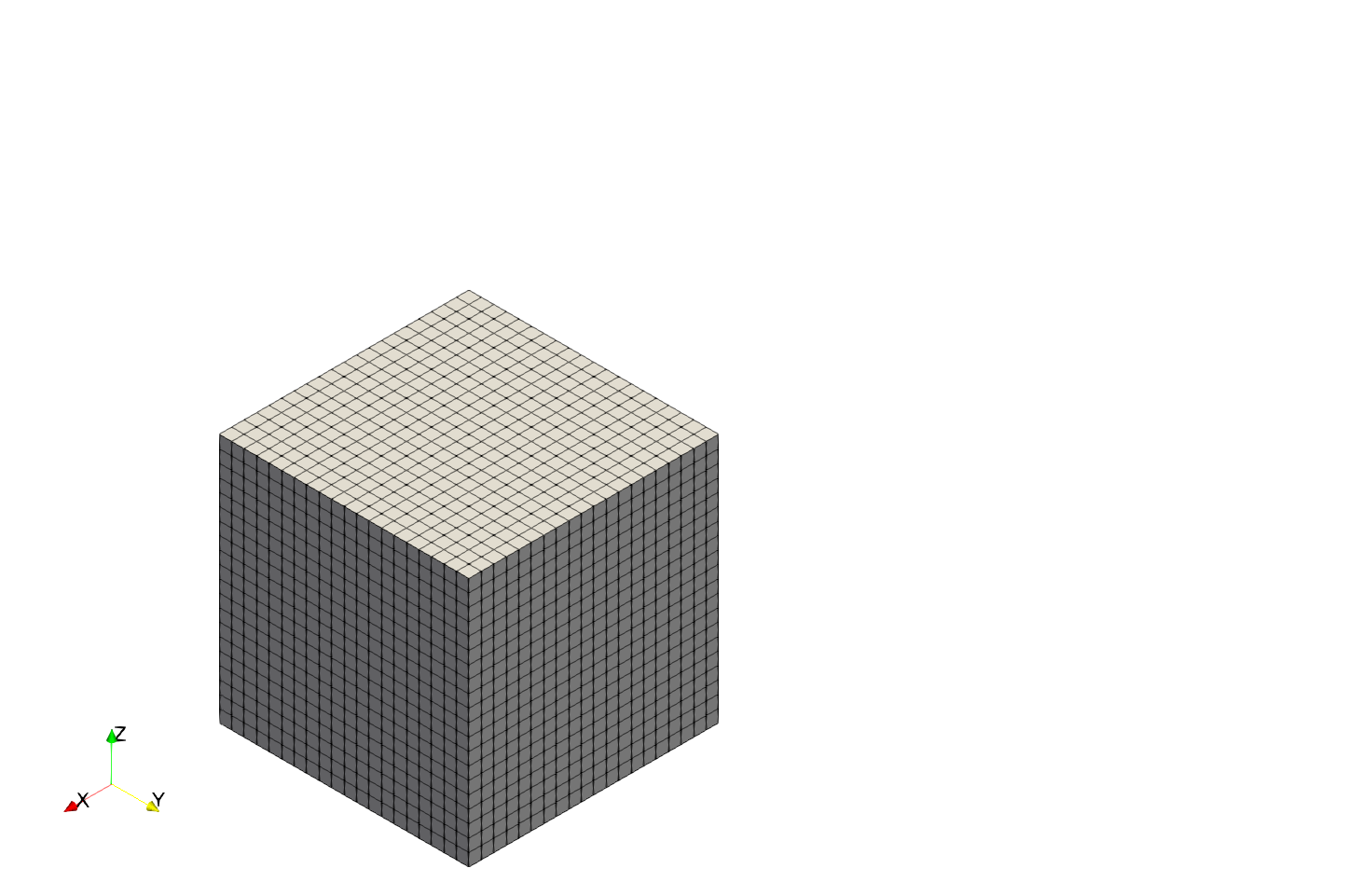}}
		\label{fig:space-time-whole-mesh}
	\end{minipage}
	\hspace{2em}
	\begin{minipage}[t]{0.3\textwidth}
		{\includegraphics[trim=200 20 650 300, clip, width=\textwidth] {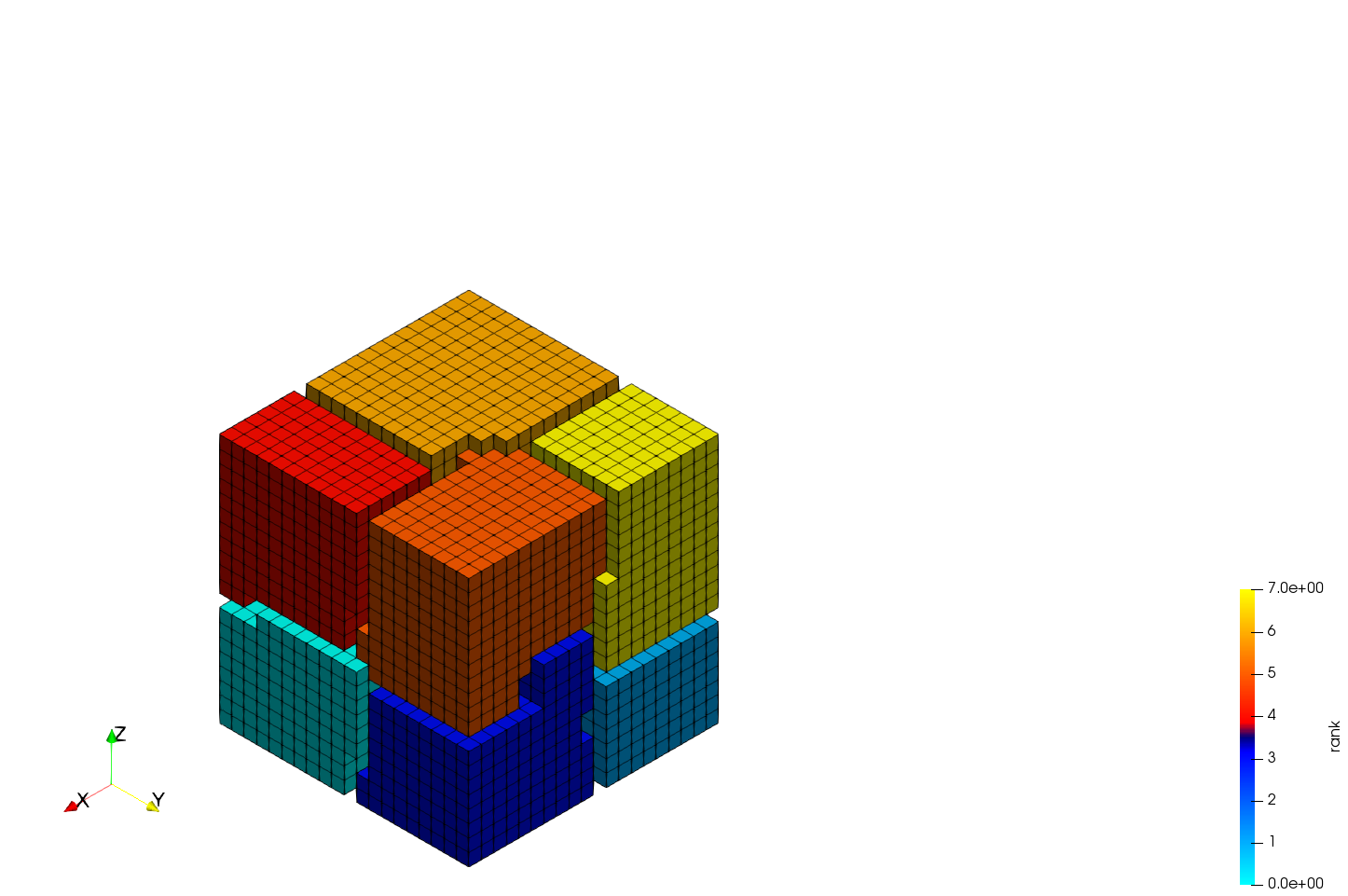}}
	\end{minipage}
	\caption{(\textit{Left}) A cubic space-time mesh in $ (2D+t) $. Spatial coordinates are $ (x,y) $, and $ z $ denotes time. (\textit{Right}) domain decomposition in space-time using 8 sub-domains.}
	\label{fig:space-time-dd}
\end{figure}
In this work, we perform domain decomposition over the full space-time mesh. Such full domain decomposition methods have been previously studied in  \cite{hughes1988space, hulbert1990space, LionsEtAl2001, pontaza2004space, tezduyar2006space, tezduyar2008arterial, scovazzi2007stabilized, song2015nitsche, gander2016analysis, dyja2018parallel} in relation to either finite-difference or discontinuous Galerkin methods. The benefit of developing the stabilized space-time continuous Galerkin method is that the space-time mesh can be decomposed without regard to the sequential time levels. Different sub-domains of the mesh contain different time points, but they are all computed in parallel. We have implemented the equations \eqref{eq:fem-problem-vms} in our in-house MPI-based parallel finite-element analysis software written in C++. The domain decomposition is performed using \parmetis{} \cite{karypis2003parmetis}, and the linear algebra and nonlinear algorithms are implemented using \petsc{} \cite{balay2001petsc}. An example visualization of the space-time domain decomposition is shown in \figref{fig:space-time-dd}. More details on the parallelization and scalability of this approach, albeit applied to linear parabolic problems, can be found in our previous work~\cite{ishii2019solving, dyja2018parallel}. We defer discussing the parallelization of the space-time Navier Stokes equations to a subsequent paper.

\section{Numerical results}
\label{sec:num-examples-results}
In this section, we demonstrate the method developed herein on benchmark problems. 
\subsection{Convergence studies with manufactured solution}
\label{sec:mms-convergence}
To test the accuracy of the method outlined above, we use the method of manufactured solutions for both $ \nsd = 2,3 $. Once again, the spatiotemporal domain (and thus the computational domain) is $ \stDom = \spDom \times \temporalInterval $. For the convergence studies, we take $ \temporalInterval = [0,1] $. For both $ \nsd = 2,3 $, we choose divergence-free velocities.
\subsubsection{2D example (3D in space-time) [$\spDom \subset \mathbb{R}^2,\ \stDom \subset \mathbb{R}^3$]}
\label{sec:mms-2d}
\input{sections/900_convergence_mms.tex}
In two spatial dimensions (i.e., $\spDom \subset \mathbb{R}^2,\ \stDom \subset \mathbb{R}^3$), we take the exact solution $ (u_x, u_y, p) $ to be
\begin{subequations}
    \begin{align}
    	u_x &= \sin(\pi x) \cos(\pi y) \sin(\pi t), \\
    	u_y &= -\cos(\pi x) \sin(\pi y) \sin(\pi t), \\
    	p &= \sin(\pi x) \sin(\pi y) \cos(\pi t).
    \end{align}
\end{subequations}
The velocity vector $\mvec{u} = (u_x, u_y) $ readily satisfies $ \grad\cdot\mvec{u} = 0 $. We obtain the right hand side forcing $ \mvec{f} $ by substituting these exact solutions $ (u_x, u_y, p) $ in  \eqref{eq:ns-original}. With the obtained forcing $ \mvec{f} $, we now solve  \eqref{eq:fem-problem-vms} using both linear ($ Q_1 $) and quadratic ($ Q_2 $) finite elements in 3D. The computational domain is based on $\stDom$, therefore the FEM problem is solved on a 3D mesh. The resulting errors in the $ L^2(\stDom) $-norm are presented for three different Reynolds numbers in  \figref{fig:conv-study-uniref-linear-basis} (for $ Q_1 $ elements) and in  \figref{fig:conv-study-uniref-quadratic-basis} (for $ Q_2 $ elements). On these $ \log-\log $ plots, we notice a slope of 2 for linear $ Q_1 $ elements and a slope of 3 for $ Q_2 $ elements. We also notice that the rate of convergence in this problem is independent of the Reynolds number. This result can be compared with the convergence estimates provided in \thmref{thm:convergence}. It can be seen that the error estimates provided therein under-predict the order of convergence because of the choice of a different norm.

\subsubsection{3D example (4D in space-time) [$\spDom \subset \mathbb{R}^3,\ \stDom \subset \mathbb{R}^4$]}
\label{sec:mms-3d}
In \emph{three} spatial dimensions (i.e., $\spDom \subset \mathbb{R}^3,\ \stDom \subset \mathbb{R}^4$), we take the exact solution $ (u_x, u_y, u_z, p) $ to be
\begin{subequations}
    \begin{align}
    	u_x &= \sin(\pi x) \cos(\pi y) \cos(\pi z) \sin(\pi t), \\
    	u_y &= -\cos(\pi x) \sin(\pi y) \cos(\pi z) \sin(\pi t), \\
        u_z &= \cos(\pi x) \sin(\pi y) \cos(\pi t), \\
    	p &= \sin(\pi x) \sin(\pi y) \sin(\pi z) \cos(\pi t).
    \end{align}
\end{subequations}
Once again, the velocity vector $\mvec{u} = (u_x, u_y, u_z) $ readily satisfies $ \grad\cdot\mvec{u} = 0 $. The forcing $ \fvec $ is obtained in the same way as for $\nsd=2$ (previous subsection). But now we use tesseract-like elements (see \figref{fig:q1-4d}, both $ Q_1 $ and $ Q_2 $) to solve this problem. The $L^2(\stDom)$-norm of the errors in the velocities and the pressure are plotted against the mesh size in \figref{fig:conv-study-uniref-linear-basis-4d}. Once again, $ Q_1 $ elements yield an order of convergence of 2 for all variables. However, $ Q_2 $ elements yield an order of convergence of 3 for the velocities, but 2 for pressure.

%\subsection{Benchmark problems}
\subsection{Lid Driven Cavity Flow}
Next, we test the present method on the well-known ``lid-driven cavity'' problem, for both $\nsd = 2, 3.$

\begin{figure}[!t]
	\centering
	\begin{subfigure}[b]{0.23\textwidth}
		\raisebox{-5cm}{\includegraphics[width=\textwidth]{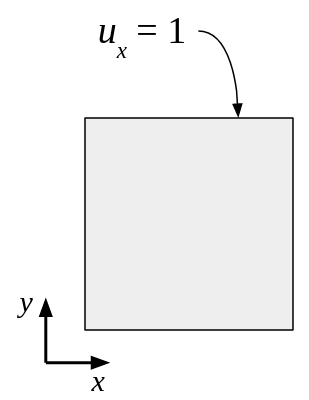}}
		\caption{$\spDom = [0,1]^2$}
		\label{fig:ldc-schematic-2d-spatial}
	\end{subfigure}
	\hspace{5em}
	\begin{subfigure}[b]{0.5\textwidth}
		\begin{tikzpicture}
		\draw (0, 0) node[inner sep=0] {\includegraphics[trim=0 0 80 0, clip, width=\textwidth] {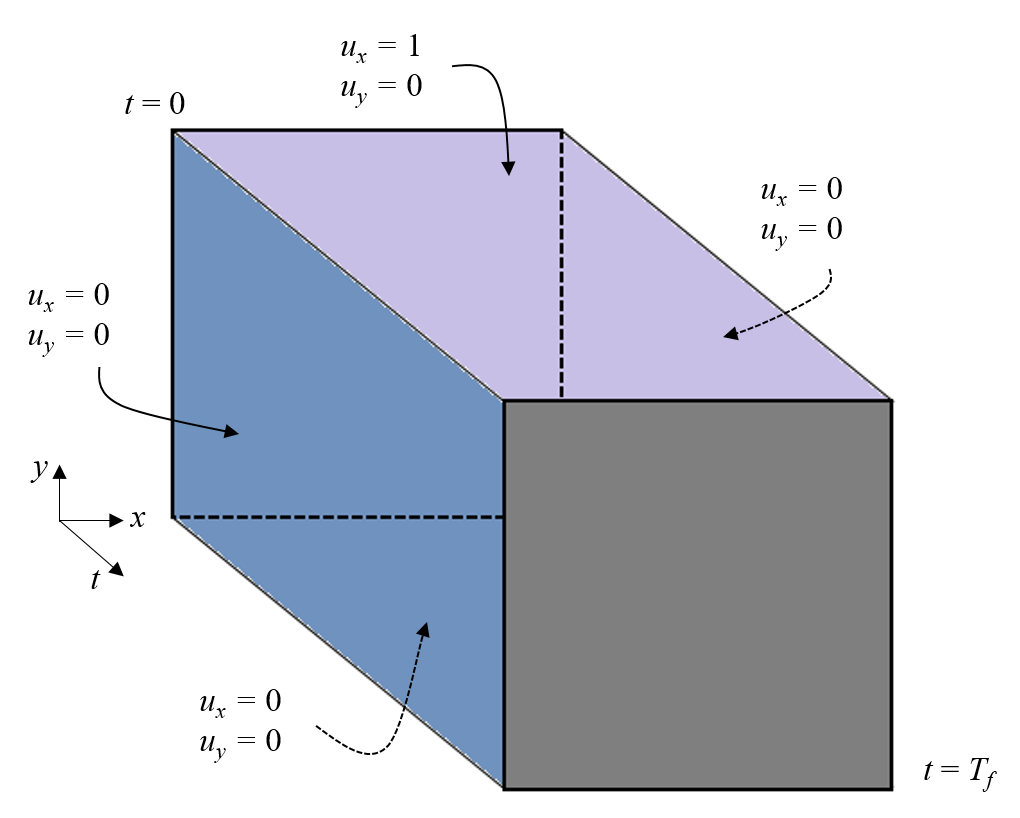}};
		\draw (4.7, -3.5) node {$t=T$};		
		\end{tikzpicture}
		\caption{$\stDom = \spDom\times [0,T]$}
		\label{fig:ldc-schematic-2d-space-time}
	\end{subfigure}
	\caption{Two dimensional lid-driven cavity problem. The spatial domain is shown in $\spDom$. All sides are no-slip, except the top surface where $u_x = 1$. The ``space-time'' computational domain is shown in (b), along with all the boundary conditions.}
	\label{fig:ldc-schematic-2d}
\end{figure}

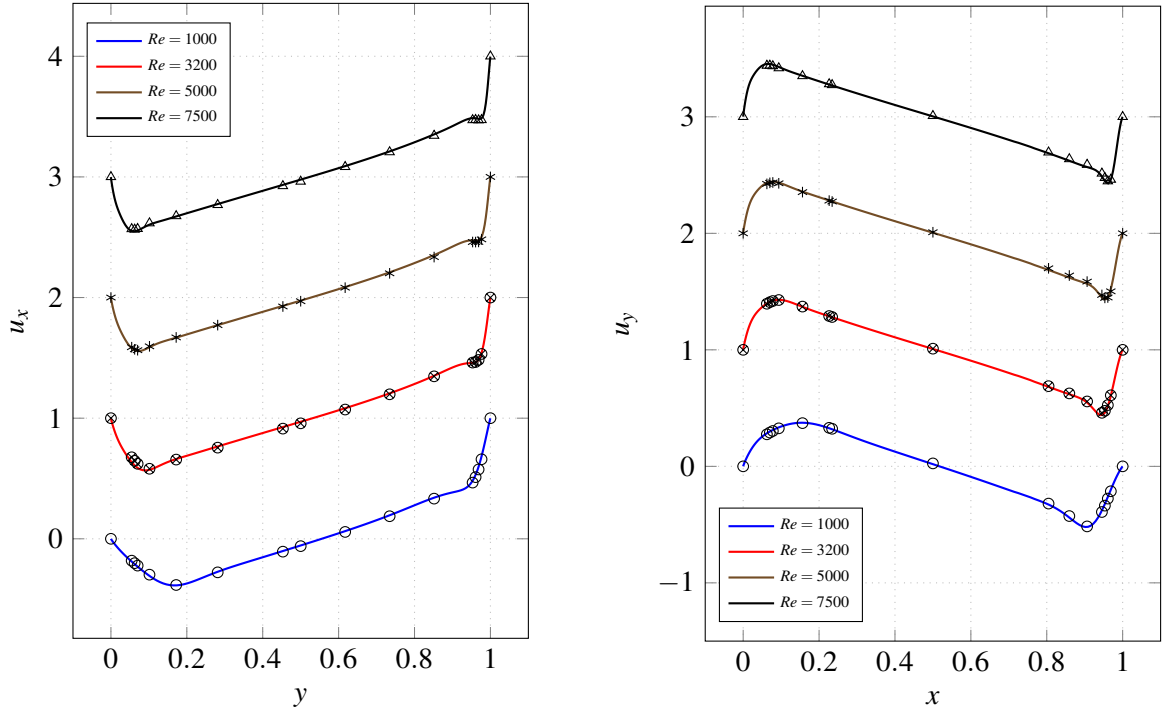
\begin{figure}[]
	\centering
	\begin{minipage}{.45\textwidth}
		\centering
		\begin{tikzpicture}
		\begin{axis}[
		%title={Velocity in X direction along vertical line (Re=1000)},
		%scale only axis,
		%width=0.75\textwidth,
		width=0.99\linewidth,
		height=10cm,
		xlabel={$y$},
		ylabel={$u_x$},
		legend pos=north west,
		legend cell align=left,
		major grid style=dotted,
		grid=major,
		]

		\addplot+ [no markers, thick] 						table[x expr={\thisrow{y}},y expr={0.00+\thisrow{vel_x}},col sep=space]{data_files/LDC_study/Re-1000-128x128-cos-vel_x.curve};
		\addplot+ [no markers, thick] 						table[x expr={\thisrow{y}},y expr={1.00+\thisrow{vel_x}},col sep=space]{data_files/LDC_study/Re-3200-128x128-cos-vel_x.curve};
		\addplot+ [no markers, thick] 						table[x expr={\thisrow{y}},y expr={2.00+\thisrow{vel_x}},col sep=space]{data_files/LDC_study/Re-5000-128x128-cos-vel_x.curve};
		\addplot+ [no markers, thick] 						table[x expr={\thisrow{y}},y expr={3.00+\thisrow{vel_x}},col sep=space]{data_files/LDC_study/Re-7500-128x128-cos-vel_x.curve};
		\addplot+ [only marks,color = black,mark=o,] 		table[x expr={\thisrow{y}},y expr={0.00+\thisrow{1000}},col sep=space]{data_files/LDC_study/ghia_vel_x.curve};
		\addplot+ [only marks,color = black,mark=otimes,]	table[x expr={\thisrow{y}},y expr={1.00+\thisrow{3200}},col sep=space]{data_files/LDC_study/ghia_vel_x.curve};
		\addplot+ [only marks,color = black,mark=asterisk,] table[x expr={\thisrow{y}},y expr={2.00+\thisrow{5000}},col sep=space]{data_files/LDC_study/ghia_vel_x.curve};
		\addplot+ [only marks,color = black,mark=triangle,] table[x expr={\thisrow{y}},y expr={3.00+\thisrow{7500}},col sep=space]{data_files/LDC_study/ghia_vel_x.curve};
		\legend{\tiny $Re = 1000$, \tiny $Re = 3200$, \tiny $Re = 5000$, \tiny $Re = 7500$}
		\end{axis}
		\end{tikzpicture}
		% \subcaption{Pure diffusion with analytical solution $u = 
		%e^{-2\pi^2\kappa t}\sin{(\pi x)}\sin{(\pi y}$)}
		\label{fig:ldc-re-1000-ux}
		%\subcaption{$Re=1000$, $u_x$}
	\end{minipage}
	\hspace{3mm}
	\begin{minipage}{.45\textwidth}
		\centering
		\begin{tikzpicture}
		\begin{axis}[
		%title={Velocity in Y direction along horizontal line (Re=1000)},
		%scale only axis,
		%width=0.75\textwidth,
		width=0.99\linewidth,
		height=10cm,
		ymin=-1.5,
		xlabel={$x$},
		ylabel={$u_y$},
		legend pos=south west,
		legend cell align=left,
		major grid style=dotted,
		grid=major,
		]

		\addplot+ [no markers, thick]						table[x expr={\thisrow{x}},y expr={0.00+\thisrow{vel_y}},col sep=space]{data_files/LDC_study/Re-1000-128x128-cos-vel_y.curve};
		\addplot+ [no markers, thick]						table[x expr={\thisrow{x}},y expr={1.00+\thisrow{vel_y}},col sep=space]{data_files/LDC_study/Re-3200-128x128-cos-vel_y.curve};
		\addplot+ [no markers, thick]						table[x expr={\thisrow{x}},y expr={2.00+\thisrow{vel_y}},col sep=space]{data_files/LDC_study/Re-5000-128x128-cos-vel_y.curve};
		\addplot+ [no markers, thick]						table[x expr={\thisrow{x}},y expr={3.00+\thisrow{vel_y}},col sep=space]{data_files/LDC_study/Re-7500-128x128-cos-vel_y.curve};
		\addplot+ [only marks,color = black,mark=o,] 		table[x expr={\thisrow{x}},y expr={0.00+\thisrow{1000}},col sep=space]{data_files/LDC_study/ghia_vel_y.curve};
		\addplot+ [only marks,color = black,mark=otimes,]	table[x expr={\thisrow{x}},y expr={1.00+\thisrow{3200}},col sep=space]{data_files/LDC_study/ghia_vel_y.curve};
		\addplot+ [only marks,color = black,mark=asterisk,] table[x expr={\thisrow{x}},y expr={2.00+\thisrow{5000}},col sep=space]{data_files/LDC_study/ghia_vel_y.curve};
		\addplot+ [only marks,color = black,mark=triangle,] table[x expr={\thisrow{x}},y expr={3.00+\thisrow{7500}},col sep=space]{data_files/LDC_study/ghia_vel_y.curve};
		\legend{\tiny $Re = 1000$, \tiny $Re = 3200$, \tiny $Re = 5000$, \tiny $Re = 7500$}
		\end{axis}
		\end{tikzpicture}
		% \subcaption{Pure diffusion with analytical solution $u = 
		%e^{-2\pi^2\kappa 
		%t}\sin{(\pi x)}\sin{(\pi y}$)}
		\label{fig:ldc-re-1000-uy}
		%\subcaption{$Re=1000$, $u_y$}
	\end{minipage}	
	\caption{Midline velocity profiles for the lid-driven cavity problem (see \secref{sec:ldc2d} and \figref{fig:ldc-schematic-2d}) at $ t=T $ for different Reynolds numbers. Left column shows $ u_x(x=0.5,\ y,\ t=T) $ against $ y $, whereas right column shows $ u_y(x,\ y=0.5,\ t=T) $ against $ x $. Curves for different $ Re $ are shifted successively by one unit in the vertical direction for clarity. Results from Ghia et al. \cite{ghia1982high} included for comparison (shown in markers).}
	\label{fig:ldc-re-all-panel-1}
\end{figure}
\subsubsection{2D example (3D in space-time) [$\spDom \subset \mathbb{R}^2,\ \stDom \subset \mathbb{R}^3$]}
\label{sec:ldc2d}

\figref{fig:ldc-schematic-2d-space-time} shows a schematic of the ``2D 
version'' of this problem, where $\spDom = [0,1]^2$, and $\stDom = [0,1]^2 \times [0,T]$. The benchmark results for this problem generally refer to the steady-state solution. Thus, the final time $ T $ is chosen such that a steady state is achieved, i.e., $ \partial_t u_i = 0,\ i=1,\ 2 $ for some $ t = t_s $; with $ T $ being strictly larger than  $ t_s $. The ``steady-state'' time value ($ t_s $) increases with increase 
in the Reynolds number, thus $ T $ also grows with $ Re $. In the present work, the time-horizon values 
chosen for the space-time simulations of this problem are $ T =  180, 180, 360, 720$ for $ Re = 1000,\ 3200,\ 5000,\ 7500 $ respectively.

%As the value of the Reynolds number grows, a central vortex develops in the domain, along with boundary layers near all spatial boundaries. Thus a higher mesh refinement is needed near the spatial boundaries. This is shown in  \figref{fig:ldc_mesh_c2}. This spatial-boundary refinement can be done independent of any refinement in the temporal direction. It is important to note that a higher temporal refinement is helpful in resolving the initial transient behavior as well as an accurate solution at the final time. A typical spatio-temporal mesh with all boundary refinement is shown in  \figref{fig:ldc_mesh_c3}. All the following results on this problem are obtained on boundary-refined clustered meshes.

\figref{fig:ldc-re-all-panel-1} shows the results obtained for different Reynolds numbers, $ Re = 1000,\ 3200,\ 5000,\ 7500 $. The left column plots the horizontal velocity $ u_x $ on a vertical line at $ x = 0.5 $ and $ t = T $, in other words, $ u_x(x=0.5, y, t=T) $ vs. $ y $. And similarly the right column plots $ u_y(x, y=0.5, t=T) $ vs. $ x $. The results in \figref{fig:ldc-re-all-panel-1} are compared with reference solutions found in \cite{ghia1982high}.

\subsubsection{3D example (4D in space-time), [$\spDom \subset \mathbb{R}^3,\ \stDom \subset \mathbb{R}^4$]}
\label{sec:ldc3d}

A schematic diagram of the driven cavity problem in three dimensions is shown in \figref{fig:ldc-schematic-3d}, where $\spDom = [0,1]^3$ and $\stDom =\spDom \times [0,T]$. The boundary conditions are given by:
\begin{subequations}
\label{eq:ldc3d-bc}
\begin{align}
u_x &= \begin{cases}
1 \ \text{at}\ y=1, \\
0 \ \text{elsewhere}, \\
\end{cases} \\
u_y=u_z &= 0 \ \text{on all surfaces}.
\end{align}
\end{subequations}
We solve \eqref{eq:fem-problem-vms} along with \eqref{eq:ldc3d-bc} for $ \visco = 0.002 $ (or $ Re = 500 $) using a mesh of $ 32^3 \times 100 $ $ Q_1 $ tesseract elements (see \figref{fig:q1-4d}). Temporal slices of the streamlines are shown in \figref{fig:ldc-3d-streamlines}.
\begin{figure}[!b]
	\centering
	\begin{subfigure}[b]{0.3\textwidth}
		\raisebox{-5cm}{\includegraphics[width=\textwidth]{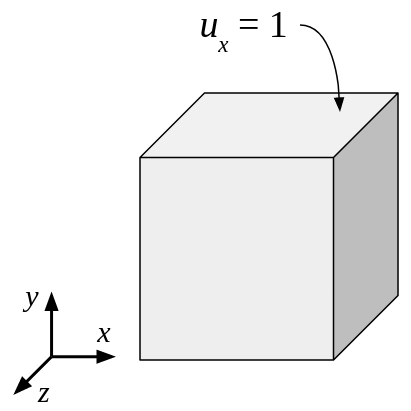}}
		% \caption{$\spDom = [0,1]^3$}
		% \label{fig:ldc-schematic-3d-spatial}
	\end{subfigure}
	\caption{Three dimensional lid-driven cavity problem. The spatial domain is $\spDom = [0,1]^3$. All sides are no-slip, except the $y=1$ surface where $u_x = 1$. The computational domain $\stDom = \spDom \times [0,T]$ is four-dimensional.}
	\label{fig:ldc-schematic-3d}
\end{figure}

%\begin{figure}[!b]
%	\centering
%	\begin{subfigure}[b]{0.9\textwidth}
%		\includegraphics[trim=280 50 175 80,clip, width=\textwidth]{data_files/ldc3d/re500/re500-tspan2p5-vmag.png}
%		% \caption{$\spDom = [0,1]^3$}
%		% \label{fig:ldc-schematic-3d-spatial}
%	\end{subfigure}
%	\caption{Contours of $|\uvec |$ after solving the lid-driven cavity problem in 3D. The computational domain is 4D, therefore, only a $z$-constant slice ($z=0.5)$ is shown here.}
%	\label{fig:ldc-3d}
%\end{figure}

\begin{figure}[!t]
	\centering
	\begin{subfigure}[b]{0.99\textwidth}
		\includegraphics[trim=0 50 0 0,clip, width=\textwidth]{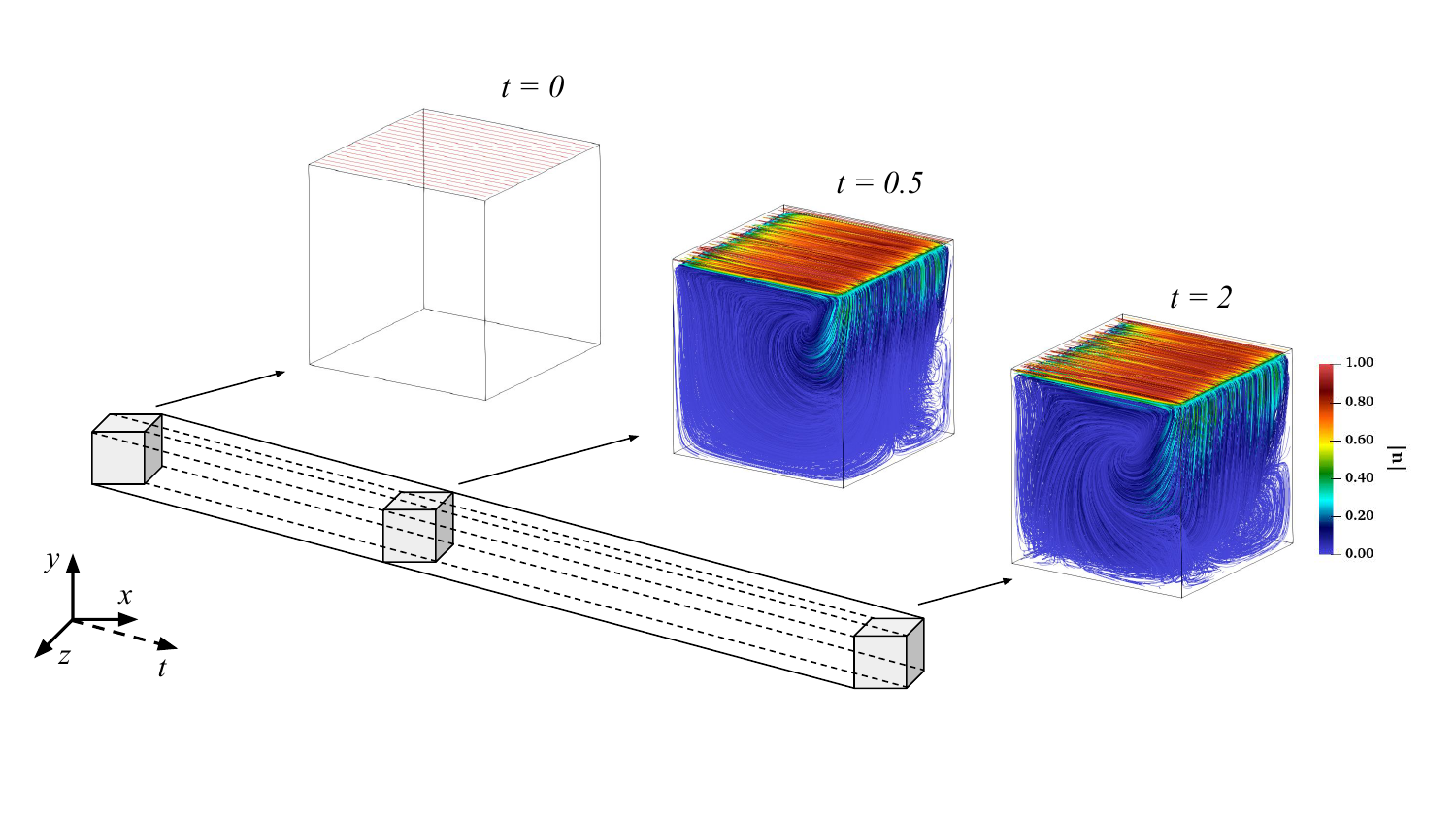}
		% \caption{$\spDom = [0,1]^3$}
		% \label{fig:ldc-schematic-3d-spatial}
	\end{subfigure}
	\caption{Temporal slices of streamlines (colored by the velocity magnitured $|\uvec |$) after solving the lid-driven cavity problem in 3D. The computational domain is 4D.}
	\label{fig:ldc-3d-streamlines}
\end{figure}

\subsection{Flow past a circular cylinder}
\label{sec:fpc}
\subsubsection{Problem definition}
%\begin{figure}[!tb]
%	\centering
%	\begin{minipage}[b]{0.6\textwidth}
%		\includegraphics[width=\textwidth]{images/fpc-space.PNG}
%		\caption{Flow past a circular cylinder (2D): problem definition in 
%			space}
%		\label{fig:fpc-schematic-space}
%	\end{minipage}
%\end{figure}
%\begin{figure}[!tb]
%	\centering
%	\begin{minipage}[b]{0.7\textwidth}
%		\includegraphics[width=\textwidth]{images/fpc-schematic-1.png}
%		\caption{Flow past a circular cylinder: problem definition in 
%			space-time}
%		\label{fig:fpc-schematic-space-time}
%	\end{minipage}
%\end{figure}
\begin{figure}[!tb]
	\centering
	\begin{minipage}[b]{0.39\textwidth}
		\includegraphics[width=\textwidth]{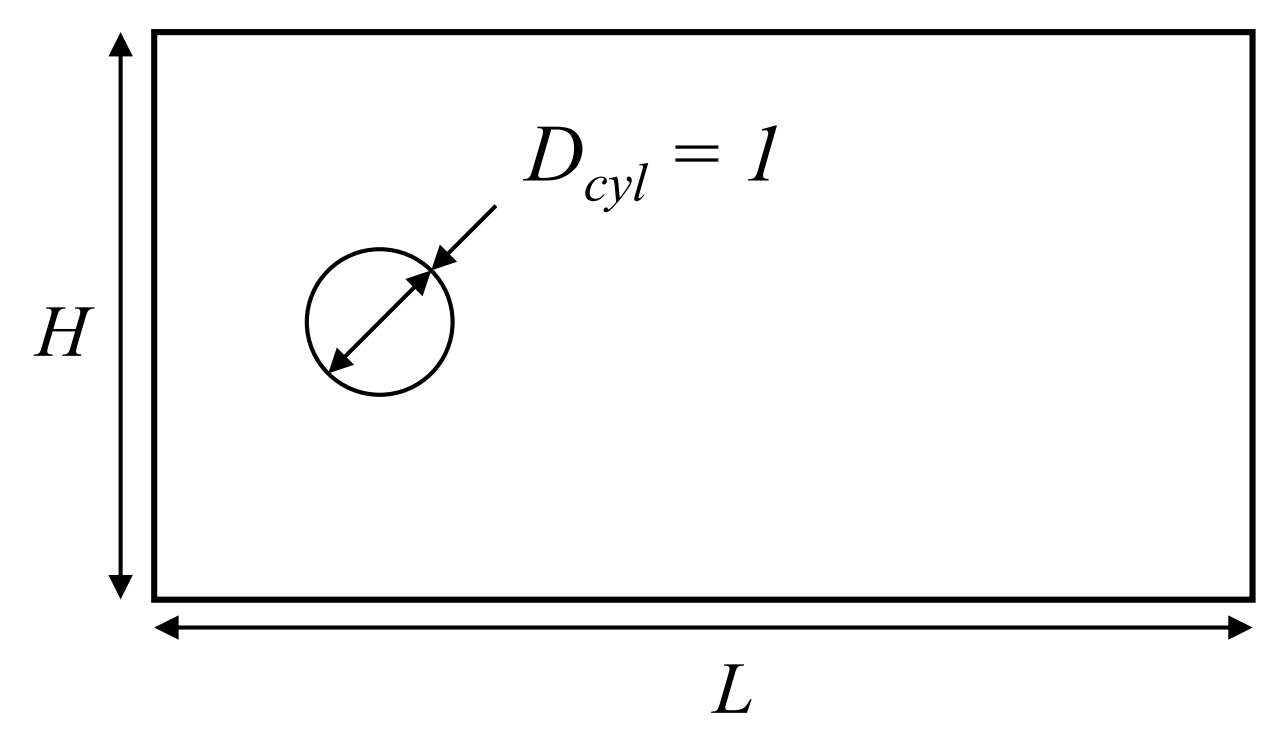}
%		\subcaption{\small $ 2 $D spatial domain for flow past a cylinder.}
%		\label{fig:fpc-spatial-domain}
	\end{minipage}
	\begin{minipage}[b]{0.49\textwidth}
		\includegraphics[trim=120 50 75 140, clip, width=\textwidth]{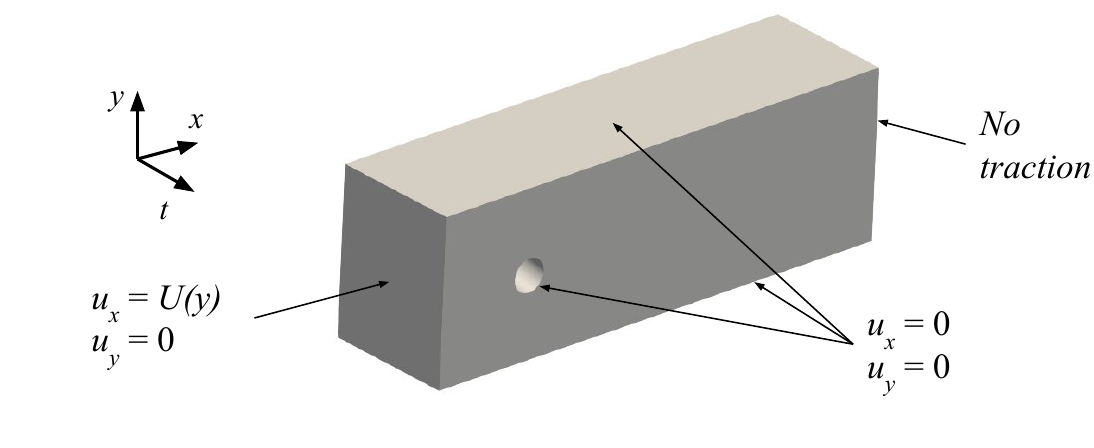}
%		\subcaption{\small $ (2 + 1) $D Spatiotemporal domain}
%		\label{fig:fpc-space-time-domain}
	\end{minipage}
	\caption{Schematic diagram of the flow past a circular cylinder problem: 
	(left) spatial domain, (right) spatiotemporal domain.}
	\label{fig:fpc-schematic-2d}
\end{figure}

\begin{figure}[!tb]
	\centering
	\begin{minipage}[b]{0.6\textwidth}
		\includegraphics[trim=200 20 200 20, clip, width=\textwidth] {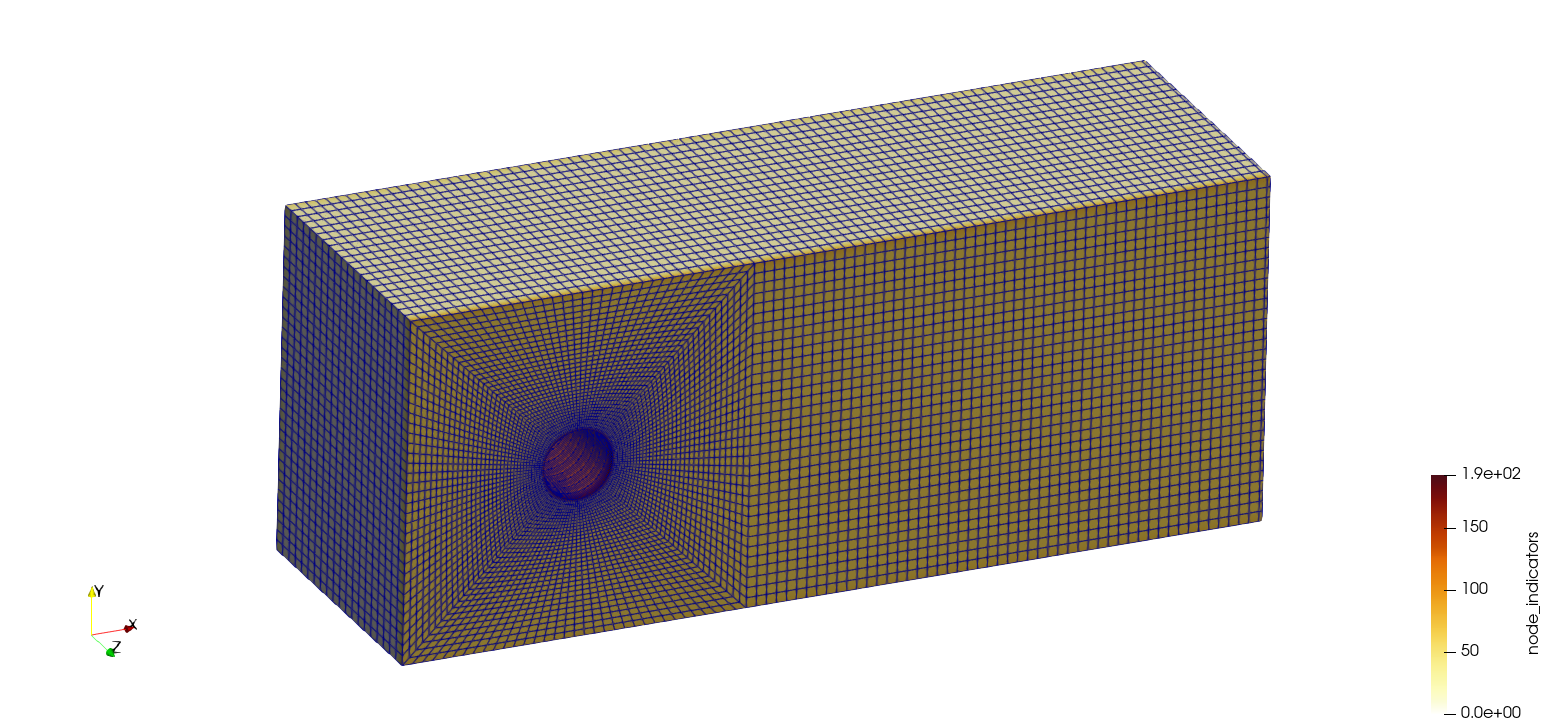}
		\caption{Space-time mesh with hexahedral elements}
		\label{fig:fpc_mesh_const}
	\end{minipage}
	% \hfill
	% \begin{minipage}[b]{0.45\textwidth}
	% 	\includegraphics[width=\textwidth]{data_files/FPC_study/fpc_mesh_setup_prog-ver2.png}
	% 	\caption{Space-time mesh with clustering in time dimension}
	% 	\label{fig:fpc_mesh_prog}
	% \end{minipage}
\end{figure}

% Please add the following required packages to your document preamble:
% \usepackage{booktabs}
% \usepackage{multirow}
\begin{table}[]
	\centering
	\caption{Benchmark configurations}
	\label{tab:expt-config}
	\begin{tabular}{@{}c|ccc|ccc|c|cc|cc@{}}
		\toprule[2pt]
		\multirow{2}{*}{\begin{tabular}[c]{@{}c@{}}Reference\\ text\end{tabular}} & \multirow{2}{*}{$D$} & \multirow{2}{*}{$H/D$} & \multirow{2}{*}{$L/D$} & \multirow{2}{*}{$u_c$} & \multirow{2}{*}{$u_{ref}$} & \multirow{2}{*}{$\nu$} & \multirow{2}{*}{$Re = \frac{u_{ref}D}{\visco}$} & \multicolumn{2}{c|}{\begin{tabular}[c]{@{}c@{}}Reference\\ values\end{tabular}} & \multicolumn{2}{c}{\begin{tabular}[c]{@{}c@{}}Computed\\ values\end{tabular}} \\ \cmidrule(l){9-12} 
		&  &  &  &  &  &  &  & $C_d$ & $St$ & $C_d$ & $St$ \\  \midrule[2pt]
		\multirow{2}{*}{\cite{schafer1996benchmark}} & \multirow{2}{*}{0.1} & \multirow{2}{*}{4.1} & \multirow{2}{*}{22} & 0.3 & 0.2 & $10^{-3}$ & 20 & 5.57 & - & 5.64 & - \\ \cmidrule(l){5-12} 
		&  &  &  & 1.5 & 1.0 & $10^{-3}$ & 100 & 3.22 & 0.295 & 3.12 & 0.29 \\ \midrule[1pt]
		\multirow{2}{*}{\cite{kanaris2011three}} & \multirow{2}{*}{1} & \multirow{2}{*}{5} & \multirow{2}{*}{40} & 1 & 1 & $ \nicefrac{1}{15} $ & 15 & 3.11 & - & 3.24 & - \\ \cmidrule(l){5-12} 
		&  &  &  & 1 & 1 & $ \nicefrac{1}{100} $ & 100 & 1.30 & 0.17 & 1.36 & 0.17 \\ \bottomrule[2pt]
	\end{tabular}
\end{table}

Next, we test our method on the problem of two-dimensional flow past a circular 
cylinder. The setup is taken from \cite{schafer1996benchmark,kanaris2011three,zovatto2001flow}. The schematic of the problem is shown in  
\figref{fig:fpc-schematic-2d}. The spatial domain is $ \spDom = ((I_x\times I_y) \backslash C) $, where $ I_x = [0,L],\ I_y = [0,H] $ and $ C $ denotes a circle of diameter $ {\spDom}_{cyl} $ with center at $ (x_c,y_c) $. The space-time domain $ \stDom = \spDom \times I_t $.
The boundary conditions are specified as:
\begin{subequations}
	\label{eq:fpc-boundary-conditions}
	\begin{align}
	&\text{at } x=0:\enspace u_x(y,t) = \sigma(y),  \quad u_y(y,t) = 0, \\
	&\text{at } x=L:\enspace \normalvec\cdot(p \mmat{I} + \visco \grad \uvec) = 0, \\
	&\text{at } y = 0,\ H:\enspace u_x(y,t) = u_y(y,t) = 0, \\
	&\text{on } C:\enspace u_x(y,t) = u_y(y,t) = 0,
	\end{align}
\end{subequations}
where the inlet flow profile $ \sigma(y) $ is given by
\begin{align}\label{eq:fpc-inlet-vel}
	\sigma(y) = 4 u_c H (y-H) / H^2,
\end{align}
which is a parabolic profile with $ \sigma(0) = \sigma(H)=0 $ and $ \sigma_m := \sigma(\nicefrac{H}{2}) = u_c $. The average inlet velocity is given by $ \bar{\sigma} = \nicefrac{2}{3} u_c $. The force on the cylinder in the direction $ \wvec $ is expressed as
\begin{align}
F_w = \int_{\Gamma_C} \left[ \left( \visco [\grad \uvec + \grad \uvec^T] -p \mathbb{I}\right) \normalvec \right]\cdot \wvec \ dS.
\end{align}
We set $ \wvec = (1,0) $ to extract the drag force, and $ \wvec = (0,1) $ to extract the lift force.

We compare our results against two different setups. These are summarized in \tabref{tab:expt-config}. The first setup is taken from \cite{schafer1996benchmark}, where the viscosity is kept constant, and the Reynolds number is changed by changing the inlet velocity. The second setup is taken from \cite{kanaris2011three}. The Reynolds number for this problem is defined as $ Re = \frac{U_{ref} D}{\nu}$. $ U_{ref} = \bar{\sigma} = \nicefrac{2}{3} u_c $ in the first setup, and $ U_{ref} = \sigma_m = u_c $ in the second setup.
%For this problem, there exists a critical Reynolds number $ Re_{cr} \approx 50 $ below which the flow is laminar and reaches a steady state after a certain time $ T $ \cite{zdravkovich1997flow}. But for higher Reynolds numbers, the wake becomes unstable with periodic oscillations.

We solve \eqref{eq:fem-problem-vms} along with \eqref{eq:fpc-boundary-conditions} using hexahedral $ Q_1 $ elements. \figref{fig:fpc_mesh_const} shows a mesh obtained by discretizing the domain $ \stDom $. For solving the nonlinear system of equations using the Newton-Raphson method, an initial guess is required. Note that the initial guess here refers to a guess for the entire space-time mesh. Coming up with a meaningful guess is non-trivial. In the present case, the initial guess was set to zero for all the nodes in the mesh, except those on the Dirichlet boundaries (this includes the initial-time boundary). Even with this simple initial guess, the Newton-Raphson method is successful in reaching a spatiotemporal solution for this problem.

\subsubsection{Results}
%\begin{figure}[!tb]
%	\centering
%	\includegraphics[trim=0 0 0 0,clip, width=0.75\textwidth]{data_files/FPC_study/fpc-Re20-streamlines.png}
%	\caption{Streamlines for FPC with Re=20}
%	\label{fig:fpc_Re20_streamlines}
%\end{figure}
We present results for $ 20 \leq Re \leq 400 $. Benchmark comparisons are presented in \tabref{tab:expt-config} and in \figref{fig:fpc-2d-validation}. \tabref{tab:expt-config} summarizes the mean coefficient of drag $ C_d $ and the Strouhal number $ St $ for one steady and one unsteady case each taken from \cite{kanaris2011three} and \cite{schafer1996benchmark}. A more extensive comparison (based on the setup of \cite{kanaris2011three}) is presented in \figref{fig:fpc-2d-validation}, including additional references. For $ Re > 50 $, the solution $ (\uvech,\ph) $ reach a time-periodic state after an initial transient phase. To obtain this result in a space-time simulation as discussed in \secref{sec:formulations}, we need to choose the domain as a tensor product of the spatial domain $ \spDom = (I_x\times I_y)\backslash C $ and the time interval $ I_T = [0,T] $ such that $ T > t_c $ or $ [0,t_c] \subset I_T $ (see  \figref{fig:fpc-schematic-2d}).
\begin{figure}[!tb]
	\centering
	\begin{minipage}{0.49\textwidth}
		\centering
		\includegraphics[trim=0 70 0 0,clip, width=0.9\linewidth]{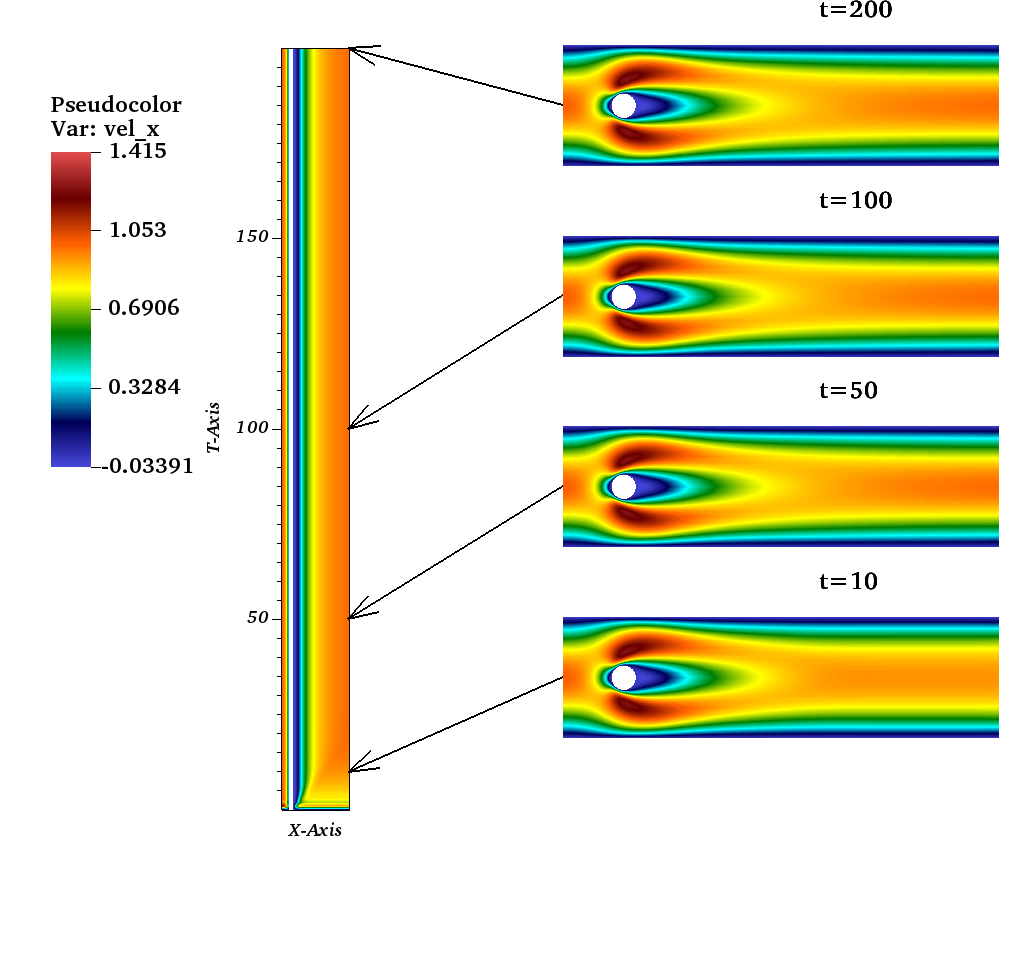}
		\subcaption{$ u_x $}
		\label{fig:fpc_Re20_velx}
	\end{minipage}
	\begin{minipage}{0.49\textwidth}
		\centering
		\includegraphics[trim=0 70 0 0,clip, width=0.9\linewidth]{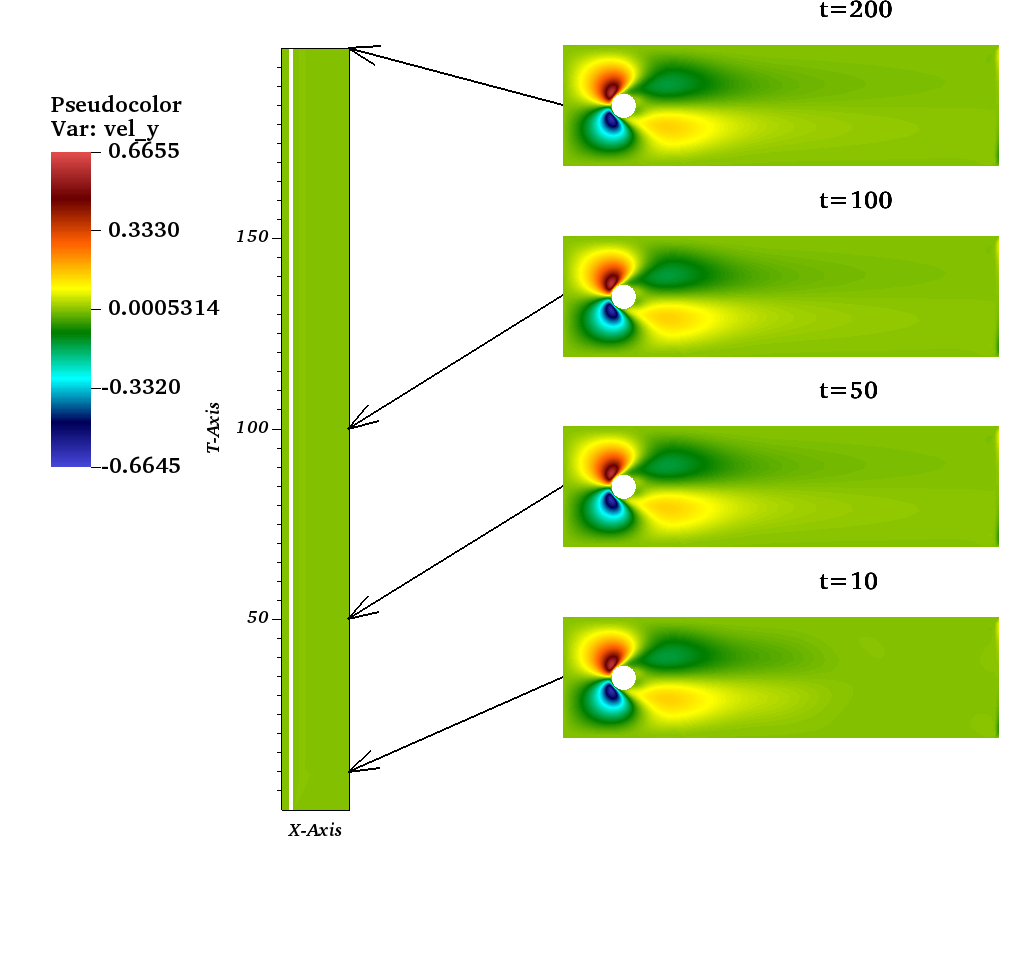}
		\subcaption{$ u_y $}
		\label{fig:fpc_Re20_vely}
	\end{minipage}
	\caption{Velocity contours for flow past a circular cylinder, $ Re = 20 $.}
	\label{fig:fpc_Re20_velocities}
\end{figure}
\begin{figure}[!tb]
	\centering
	\begin{minipage}{0.49\textwidth}
		\centering
		\includegraphics[trim=0 70 0 0,clip, width=0.99\linewidth]{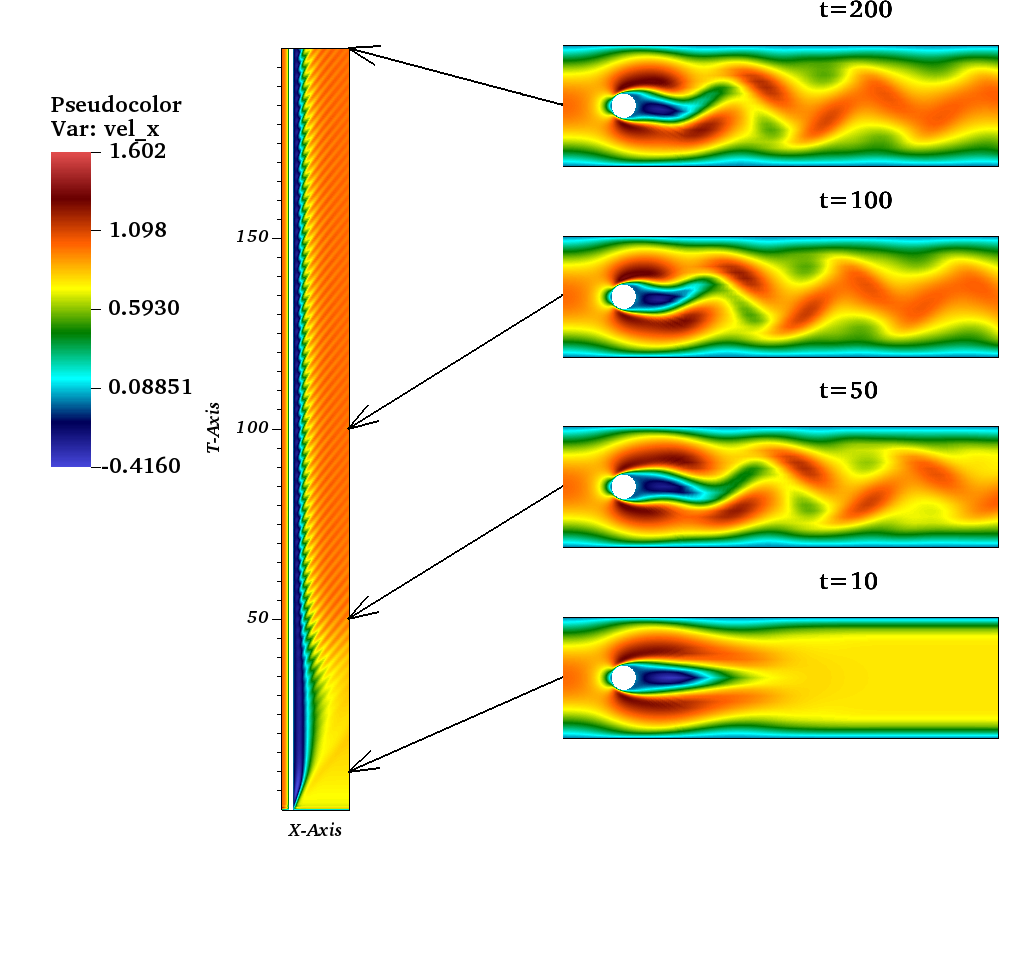}
		\subcaption{$ u_x $}
		\label{fig:fpc_Re100_velx}
	\end{minipage}
	\begin{minipage}{0.49\textwidth}
		\centering
		\includegraphics[trim=0 70 0 0,clip, width=0.99\linewidth]{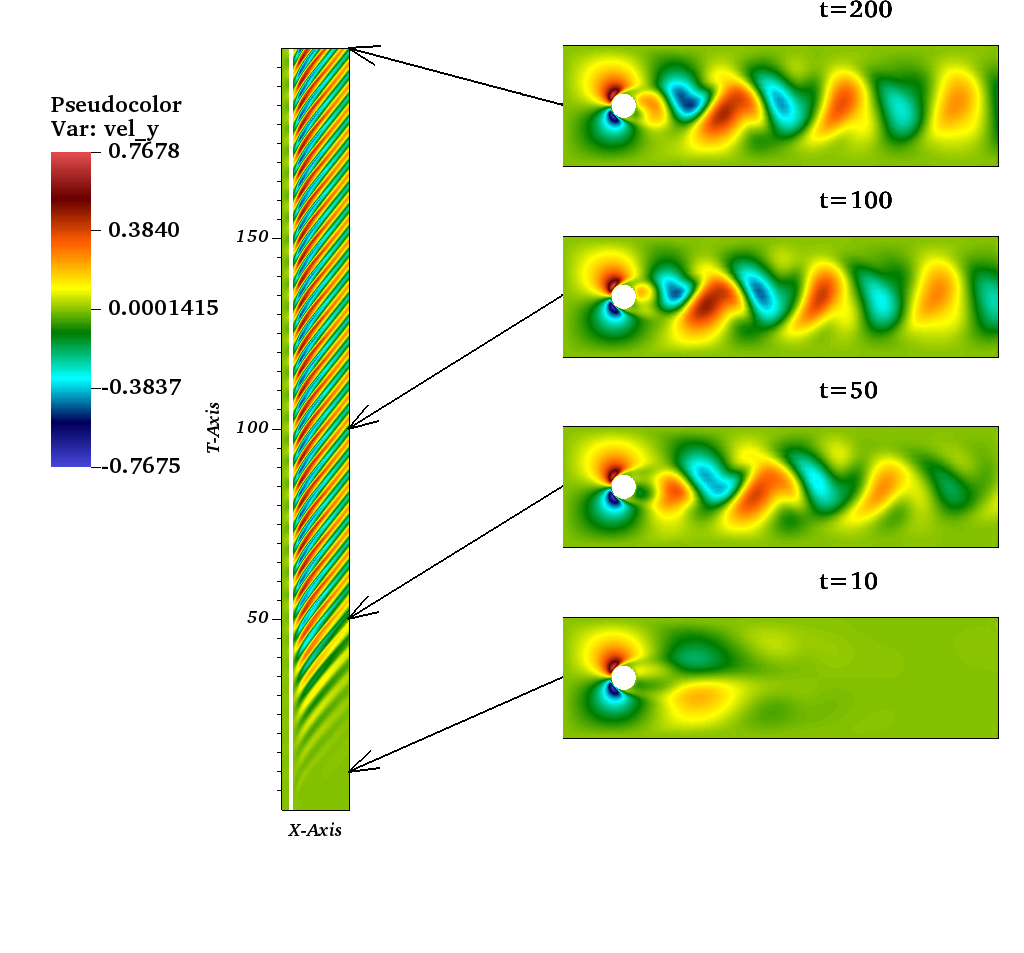}
		\subcaption{$ u_y $}
		\label{fig:fpc_Re100_vely}
	\end{minipage}
	\caption{Velocity contours for flow past a circular cylinder, $ Re = 100 $.}
	\label{fig:fpc_Re100_velocities}
\end{figure}

\figref{fig:fpc_Re20_velocities} and \figref{fig:fpc_Re100_velocities} show temporal slices of the velocity solutions for $ Re=20 $ and $ Re = 100 $ respectively.
Both show an $ x-t $ slice (at $ y=H/2 $) and several $ x-y $ slices at multiple $ t $-values of the 3D domain. The $ x-t $ slice at $ y=H/2 $ is seen on the left, with the time-axis extending upward. This slice shows the solution $ u_x(x,t)|_{y=H/2} $ evolving in time. Also, the slices in \figref{fig:fpc_Re20_velocities} show the emergence of a steady state, whereas those in \figref{fig:fpc_Re100_velocities} show that the solution exhibits some signs of periodicity around $ t=50 $ and really falls into a periodic pattern roughly around $ t>70 $.

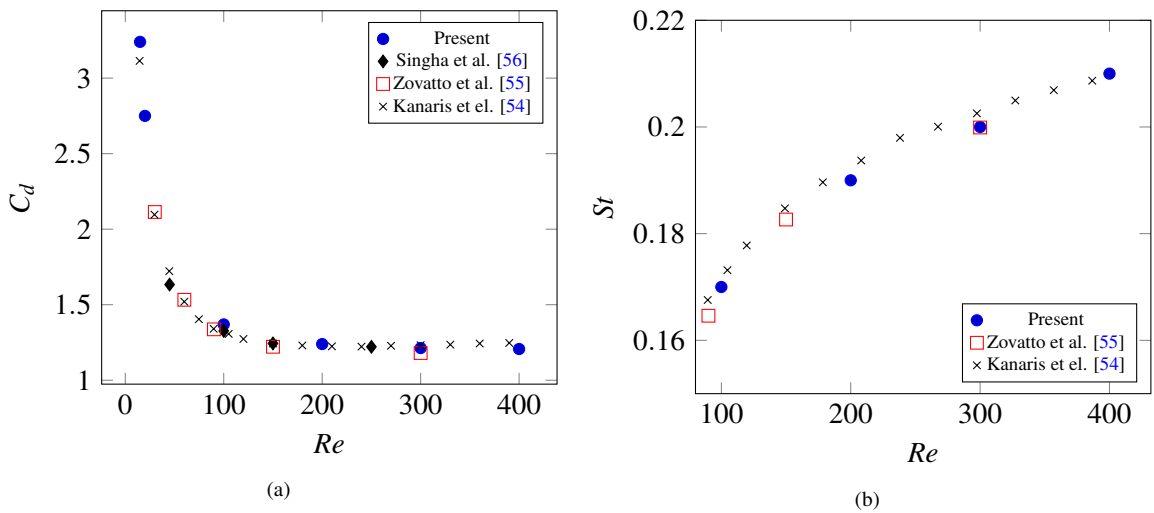
\begin{figure}[!htb]
		\centering
		\begin{minipage}{0.45\textwidth}
			\centering
		\begin{tikzpicture}
		\begin{axis}[width=0.99\linewidth,
		scaled y ticks=true,
		xlabel={$Re$},
		ylabel={$C_d$},
		legend style={at={(0.5,-0.25)},anchor=north, nodes={scale=0.65, transform shape}}, 
		legend pos=north east, 
		%legend columns=-1,
		% title={(a)},
		%ymin=1e-4,ymax=3e-0,
		%xtick={1e-1,7.5e-2,5e-2,2.5e-2,1e-2,7.5e-3,5e-3}
		]
		\addplot+[only marks, mark size=2.1pt, color=blue] table [x expr=\thisrowno{0},y expr=\thisrowno{1},col sep=comma] {data_files/FPC_study/strouhal_data.txt}; \addlegendentry{Present}
		\addplot+[only marks, mark=diamond*, mark options={scale=1.2}, color=black] table [x expr=\thisrowno{0},y expr=\thisrowno{1},col sep=comma] {data_files/FPC_study/singha-drag-data.txt}; \addlegendentry{Singha et al. \cite{singha2010flow}}
		\addplot+[only marks, mark=square, mark options={scale=1.2}, color=red] table [x expr=\thisrowno{1},y expr=\thisrowno{2},col sep=comma] {data_files/FPC_study/zovatto-drag-data.txt}; \addlegendentry{Zovatto et al. \cite{zovatto2001flow}}
		\addplot+[only marks, mark=x, color=black] table [x expr=\thisrowno{0},y expr=\thisrowno{1},col sep=comma] {data_files/FPC_study/kanaris-drag-data.txt}; \addlegendentry{Kanaris et el. \cite{kanaris2011three}}
		\end{axis}
		\end{tikzpicture}
		\subcaption{}
		\label{fig:drag}
		\end{minipage}
		\begin{minipage}{0.45\textwidth}
			\centering
		\begin{tikzpicture}
		\begin{axis}[width=0.99\linewidth,
		scaled y ticks=true,
		xlabel={$Re$},
		ylabel={$St$},
		legend style={at={(0.5,-0.25)},anchor=north, nodes={scale=0.65, transform shape}}, 
		legend pos=south east, 
		xmin = 80,
		ymin=0.15,ymax=0.22
		]
		\addplot+ [only marks, mark size=2.1pt, color=blue] table [x expr=\thisrowno{0},y expr=\thisrowno{2},col sep=comma] {data_files/FPC_study/strouhal_data.txt}; \addlegendentry{Present}
		\addplot+ [only marks, mark=square, mark options={scale=1.2, fill=white}] table [x expr=\thisrowno{1},y expr=\thisrowno{3},col sep=comma] {data_files/FPC_study/zovatto-drag-data.txt}; \addlegendentry{Zovatto et al. \cite{zovatto2001flow}}
		\addplot+ [only marks, mark=x, color=black] table [x expr=\thisrowno{0},y expr=\thisrowno{1},col sep=comma] {data_files/FPC_study/kanaris-st-data.txt}; \addlegendentry{Kanaris et el. \cite{kanaris2011three}}
		\end{axis}
		\end{tikzpicture}
		\subcaption{}
		\label{fig:st}
		\end{minipage}
\caption{{Results for flow past a circular cylinder in two spatial dimensions:} (a) Drag coefficient $ C_d $, (b) Strouhal number $ St $}
\label{fig:fpc-2d-validation}
\end{figure}

%\subsection{Solving with large space-time blocks}

%\subsection{Clustered meshes}
%$ <\text{pending}> $
%One of the way to leverage the FEM approximation in the time dimension is to refine the mesh in time 

%\input{sections/1008_LDC_benchmark_comparison_re_all_2.tex}
%\end{document}

\section{Conclusions}
\label{sec:conclusions}

In this paper, we present a method for solving flow problems in space-time using continuous Galerkin finite element method. At the discrete level, such a problem can become unstable when the equation is dominated by the convection term. Also, applying equal-order finite elements for both velocity and pressure can render a non-trivial null space for the pressure solution, thus running into non-uniqueness of pressure. Both these problems can be resolved by formulating the FEM problem using a variational multiscale approach. The application of VMS provides stability against both spatial advection as well as temporal advection (i.e., the time derivative term). It also transforms the saddle point operator into a coercive operator, thus restoring the uniqueness of solutions.

As shown in the numerical results in this paper, the proposed method is able to provide convergence rates as expected. The convergence is both in space and in time. The convergence rate is found to be independent of the Reynolds number. We also tested the method on two benchmark problems, namely, the lid driven cavity problem and the flow past a cylinder problem. Both problems displayed satisfactory results that match previously reported results in the literature.

It can be argued that, everything else being equal, a space-time solution is more expensive than a marching solution for the same problem. But this increased complexity of computation also gives us an opportunity to leverage the scaling capabilities of modern codebases designed for supercomputing facilities. Since the entire time-horizon is included in the mesh, a space-time mesh is extremely well-suited for domain-decomposed analyses. Therefore, even though currently a singular marching simulation is cheaper than a single space-time simulation, it can be shown \cite{ishii2019solving} that with a proper design of scalable codes and proper selection of supercomputing resources, the time-cost of solving PDEs in space-time can be significantly reduced in comparison to time-marching approaches.

\section{Acknowledgments}
This work was partly supported by the National Science Foundation under the grants NSF LEAP-HI 2053760, NSF 1935255.

% \clearpage

% The other stuff
%\nocite{*}
% \section{References}
\bibliographystyle{unsrt}
\bibliography{reflist}

\begin{thebibliography}{10}

\bibitem{evans2022partial}
Lawrence~C Evans.
\newblock {\em Partial differential equations}, volume~19.
\newblock American Mathematical Society, 2022.

\bibitem{leveque2007finite}
Randall~J LeVeque.
\newblock {\em Finite difference methods for ordinary and partial differential equations: steady-state and time-dependent problems}.
\newblock SIAM, 2007.

\bibitem{nievergelt1964parallel}
J{\"u}rg Nievergelt.
\newblock Parallel methods for integrating ordinary differential equations.
\newblock {\em Communications of the ACM}, 7(12):731--733, 1964.

\bibitem{Hackbusch:1985:PMM:4673.4714}
Wolfgang Hackbusch.
\newblock Parabolic multi-grid methods.
\newblock In {\em Proc. Of the Sixth Int'L. Symposium on Computing Methods in Applied Sciences and Engineering, VI}, pages 189--197, Amsterdam, The Netherlands, The Netherlands, 1985. North-Holland Publishing Co.

\bibitem{chartier1993parallel}
Philippe Chartier and Bernard Philippe.
\newblock A parallel shooting technique for solving dissipative ode's.
\newblock {\em Computing}, 51(3):209--236, 1993.

\bibitem{horton1995space}
Graham Horton and Stefan Vandewalle.
\newblock A space-time multigrid method for parabolic partial differential equations.
\newblock {\em SIAM Journal on Scientific Computing}, 16(4):848--864, 1995.

\bibitem{gander1996overlapping}
Martin~Jakob Gander.
\newblock Overlapping schwarz for linear and nonlinear parabolic problems.
\newblock In {\em 9th International Conference on Domain Decomposition Methods}, pages 97--104, 1996.

\bibitem{gander1999optimal}
Martin~Jakob Gander, Laurence Halpern, and Fr{\'e}d{\'e}ric Nataf.
\newblock Optimal convergence for overlapping and non-overlapping schwarz waveform relaxation.
\newblock In {\em 11th international conference on domain decomposition methods}, pages 27--36, 1999.

\bibitem{LionsEtAl2001}
J.-L. Lions, Yvon Maday, and Gabriel Turinici.
\newblock {A "parareal" in time discretization of {PDE}'s}.
\newblock {\em Comptes Rendus de l'Académie des Sciences - Series I - Mathematics}, 332:661--668, 2001.

\bibitem{gander201550}
Martin~J Gander.
\newblock 50 years of time parallel time integration.
\newblock In {\em Multiple Shooting and Time Domain Decomposition Methods}, pages 69--113. Springer, 2015.

\bibitem{saha1996parallel}
Prasenjit Saha, Joachim Stadel, and Scott Tremaine.
\newblock A parallel integration method for solar system dynamics.
\newblock {\em arXiv preprint astro-ph/9605016}, 1996.

\bibitem{lubich1987multi}
Ch~Lubich and A~Ostermann.
\newblock Multi-grid dynamic iteration for parabolic equations.
\newblock {\em BIT Numerical Mathematics}, 27(2):216--234, 1987.

\bibitem{dyja2018parallel}
Robert Dyja, Baskar Ganapathysubramanian, and Kristoffer~G van~der Zee.
\newblock Parallel-in-space-time, adaptive finite element framework for nonlinear parabolic equations.
\newblock {\em SIAM Journal on Scientific Computing}, 40(3):C283--C304, 2018.

\bibitem{hughes1988space}
Thomas~JR Hughes and Gregory~M Hulbert.
\newblock Space-time finite element methods for elastodynamics: formulations and error estimates.
\newblock {\em Computer methods in applied mechanics and engineering}, 66(3):339--363, 1988.

\bibitem{hulbert1990space}
Gregory~M Hulbert and Thomas~JR Hughes.
\newblock Space-time finite element methods for second-order hyperbolic equations.
\newblock {\em Computer methods in applied mechanics and engineering}, 84(3):327--348, 1990.

\bibitem{pontaza2004space}
JP~Pontaza and JN~Reddy.
\newblock Space--time coupled spectral/hp least-squares finite element formulation for the incompressible navier--stokes equations.
\newblock {\em Journal of Computational Physics}, 197(2):418--459, 2004.

\bibitem{tezduyar2006space}
Tayfun~E Tezduyar, Sunil Sathe, Ryan Keedy, and Keith Stein.
\newblock Space--time finite element techniques for computation of fluid--structure interactions.
\newblock {\em Computer methods in applied mechanics and engineering}, 195(17-18):2002--2027, 2006.

\bibitem{tezduyar2008arterial}
Tayfun~E Tezduyar, Sunil Sathe, Matthew Schwaab, and Brian~S Conklin.
\newblock Arterial fluid mechanics modeling with the stabilized space--time fluid--structure interaction technique.
\newblock {\em International Journal for Numerical Methods in Fluids}, 57(5):601--629, 2008.

\bibitem{scovazzi2007stabilized}
G~Scovazzi.
\newblock Stabilized shock hydrodynamics: I. a lagrangian method.
\newblock {\em Computer Methods in Applied Mechanics and Engineering}, 196(4-6):923--966, 2007.

\bibitem{song2015nitsche}
T~Song and G~Scovazzi.
\newblock A nitsche method for wave propagation problems in time domain.
\newblock {\em Computer Methods in Applied Mechanics and Engineering}, 293:481--521, 2015.

\bibitem{gander2016analysis}
Martin~J Gander and Martin Neumuller.
\newblock Analysis of a new space-time parallel multigrid algorithm for parabolic problems.
\newblock {\em SIAM Journal on Scientific Computing}, 38(4):A2173--A2208, 2016.

\bibitem{munts2006space}
Edwin~Albert Munts.
\newblock {\em Space-time multiscale methods for large eddy simulation}.
\newblock PhD thesis, Citeseer, 2006.

\bibitem{behr2008simplex}
Marek Behr.
\newblock Simplex space--time meshes in finite element simulations.
\newblock {\em International journal for numerical methods in fluids}, 57(9):1421--1434, 2008.

\bibitem{steinbach2015space}
Olaf Steinbach.
\newblock Space-time finite element methods for parabolic problems.
\newblock {\em Computational methods in applied mathematics}, 15(4):551--566, 2015.

\bibitem{langer2016space}
Ulrich Langer, Stephen~E Moore, and Martin Neum{\"u}ller.
\newblock Space--time isogeometric analysis of parabolic evolution problems.
\newblock {\em Computer methods in applied mechanics and engineering}, 306:342--363, 2016.

\bibitem{fuhrer2021space}
Thomas F{\"u}hrer and Michael Karkulik.
\newblock Space--time least-squares finite elements for parabolic equations.
\newblock {\em Computers \& Mathematics with Applications}, 92:27--36, 2021.

\bibitem{ishii2019solving}
Masado Ishii, Milinda Fernando, Kumar Saurabh, Biswajit Khara, Baskar Ganapathysubramanian, and Hari Sundar.
\newblock Solving pdes in space-time: 4d tree-based adaptivity, mesh-free and matrix-free approaches.
\newblock In {\em Proceedings of the International Conference for High Performance Computing, Networking, Storage and Analysis}, pages 1--61, 2019.

\bibitem{andreev2012stability}
Roman Andreev.
\newblock {\em Stability of space-time Petrov-Galerkin discretizations for parabolic evolution equations}.
\newblock PhD thesis, ETH Zurich, 2012.

\bibitem{brooks1982streamline}
Alexander~N Brooks and Thomas~JR Hughes.
\newblock Streamline upwind/petrov-galerkin formulations for convection dominated flows with particular emphasis on the incompressible navier-stokes equations.
\newblock {\em Computer methods in applied mechanics and engineering}, 32(1-3):199--259, 1982.

\bibitem{hughes1989new}
Thomas~JR Hughes, Leopoldo~P Franca, and Gregory~M Hulbert.
\newblock A new finite element formulation for computational fluid dynamics: Viii. the galerkin/least-squares method for advective-diffusive equations.
\newblock {\em Computer methods in applied mechanics and engineering}, 73(2):173--189, 1989.

\bibitem{douglas1989absolutely}
Jim Douglas and Jun~Ping Wang.
\newblock An absolutely stabilized finite element method for the stokes problem.
\newblock {\em Mathematics of computation}, 52(186):495--508, 1989.

\bibitem{hughes1995multiscale}
Thomas~JR Hughes.
\newblock Multiscale phenomena: Green's functions, the dirichlet-to-neumann formulation, subgrid scale models, bubbles and the origins of stabilized methods.
\newblock {\em Computer methods in applied mechanics and engineering}, 127(1-4):387--401, 1995.

\bibitem{hughes1986new}
Thomas~JR Hughes, Leopoldo~P Franca, and Marc Balestra.
\newblock A new finite element formulation for computational fluid dynamics: V. circumventing the babu{\v{s}}ka-brezzi condition: A stable petrov-galerkin formulation of the stokes problem accommodating equal-order interpolations.
\newblock {\em Computer Methods in Applied Mechanics and Engineering}, 59(1):85--99, 1986.

\bibitem{bazilevs2007variational}
Y~Bazilevs, VM~Calo, JA~Cottrell, TJR Hughes, A~Reali, and G~Scovazzi.
\newblock Variational multiscale residual-based turbulence modeling for large eddy simulation of incompressible flows.
\newblock {\em Computer methods in applied mechanics and engineering}, 197(1-4):173--201, 2007.

\bibitem{temam1968methode}
Roger Temam.
\newblock Une m{\'e}thode d'approximation de la solution des {\'e}quations de navier-stokes.
\newblock {\em Bulletin de la Soci{\'e}t{\'e} Math{\'e}matique de France}, 96:115--152, 1968.

\bibitem{shen1996error}
Jie Shen.
\newblock On error estimates of the projection methods for the navier-stokes equations: second-order schemes.
\newblock {\em Mathematics of computation}, 65(215):1039--1065, 1996.

\bibitem{shen1997pseudo}
Jie Shen.
\newblock Pseudo-compressibility methods for the unsteady incompressible navier-stokes equations.

\bibitem{codina2001stabilized}
Ramon Codina.
\newblock A stabilized finite element method for generalized stationary incompressible flows.
\newblock {\em Computer methods in applied mechanics and engineering}, 190(20-21):2681--2706, 2001.

\bibitem{babuvska1973finite}
Ivo Babu{\v{s}}ka.
\newblock The finite element method with lagrangian multipliers.
\newblock {\em Numerische Mathematik}, 20(3):179--192, 1973.

\bibitem{brezzi1974existence}
Franco Brezzi.
\newblock On the existence, uniqueness and approximation of saddle-point problems arising from lagrangian multipliers.
\newblock {\em Publications math{\'e}matiques et informatique de Rennes}, (S4):1--26, 1974.

\bibitem{john2016finite}
Volker John.
\newblock {\em Finite element methods for incompressible flow problems}.
\newblock Springer, 2016.

\bibitem{codina1998comparison}
Ramon Codina.
\newblock Comparison of some finite element methods for solving the diffusion-convection-reaction equation.
\newblock {\em Computer methods in applied mechanics and engineering}, 156(1-4):185--210, 1998.

\bibitem{tezduyar1992incompressible}
Tayfun~E Tezduyar, Sanjay Mittal, SE~Ray, and R~Shih.
\newblock Incompressible flow computations with stabilized bilinear and linear equal-order-interpolation velocity-pressure elements.
\newblock {\em Computer Methods in Applied Mechanics and Engineering}, 95(2):221--242, 1992.

\bibitem{hughes1986newBeyondSUPG}
Thomas~JR Hughes, Michel Mallet, and Mizukami Akira.
\newblock A new finite element formulation for computational fluid dynamics: Ii. beyond supg.
\newblock {\em Computer methods in applied mechanics and engineering}, 54(3):341--355, 1986.

\bibitem{shakib1991new}
Farzin Shakib, Thomas~JR Hughes, and Zden{\v{e}}k Johan.
\newblock A new finite element formulation for computational fluid dynamics: X. the compressible euler and navier-stokes equations.
\newblock {\em Computer Methods in Applied Mechanics and Engineering}, 89(1-3):141--219, 1991.

\bibitem{ciarlet2002finite}
Philippe~G Ciarlet.
\newblock {\em The finite element method for elliptic problems}.
\newblock SIAM, 2002.

\bibitem{brenner2007mathematical}
Susanne Brenner and Ridgway Scott.
\newblock {\em The mathematical theory of finite element methods}, volume~15.
\newblock Springer Science \& Business Media, 2007.

\bibitem{zhou1993least}
Tian~Xiao Zhou and Min~Fu Feng.
\newblock A least squares petrov-galerkin finite element method for the stationary navier-stokes equations.
\newblock {\em mathematics of computation}, 60(202):531--543, 1993.

\bibitem{jt1976introduction}
JT~(John~Tinsley) Oden and Junuthula~Narasimha Reddy.
\newblock {\em An introduction to the mathematical theory of finite elements}.
\newblock John Wiley \& Sons, Limited, 1976.

\bibitem{karypis2003parmetis}
George Karypis, Kirk Schloegel, and Vipin Kumar.
\newblock Parmetis.
\newblock {\em Parallel graph partitioning and sparse matrix ordering library. Version}, 2, 2003.

\bibitem{balay2001petsc}
Satish Balay, Kris Buschelman, William~D Gropp, Dinesh Kaushik, Matthew~G Knepley, L~Curfman McInnes, Barry~F Smith, and Hong Zhang.
\newblock Petsc.
\newblock {\em See http://www. mcs. anl. gov/petsc}, 2001.

\bibitem{ghia1982high}
UKNG Ghia, Kirti~N Ghia, and CT~Shin.
\newblock High-re solutions for incompressible flow using the navier-stokes equations and a multigrid method.
\newblock {\em Journal of computational physics}, 48(3):387--411, 1982.

\bibitem{schafer1996benchmark}
Michael Sch{\"a}fer, Stefan Turek, Franz Durst, Egon Krause, and Rolf Rannacher.
\newblock {\em Benchmark computations of laminar flow around a cylinder}.
\newblock Springer, 1996.

\bibitem{kanaris2011three}
Nicolas Kanaris, Dimokratis Grigoriadis, and Stavros Kassinos.
\newblock Three dimensional flow around a circular cylinder confined in a plane channel.
\newblock {\em Physics of fluids}, 23(6), 2011.

\bibitem{zovatto2001flow}
Luigino Zovatto and Gianni Pedrizzetti.
\newblock Flow about a circular cylinder between parallel walls.
\newblock {\em Journal of Fluid Mechanics}, 440:1--25, 2001.

\bibitem{singha2010flow}
Sintu Singha and KP~Sinhamahapatra.
\newblock Flow past a circular cylinder between parallel walls at low reynolds numbers.
\newblock {\em Ocean Engineering}, 37(8-9):757--769, 2010.

\bibitem{coffey2003pseudotransient}
Todd~S Coffey, Carl~Tim Kelley, and David~E Keyes.
\newblock Pseudotransient continuation and differential-algebraic equations.
\newblock {\em SIAM Journal on Scientific Computing}, 25(2):553--569, 2003.

\bibitem{kelley1998convergence}
Carl~Timothy Kelley and David~E Keyes.
\newblock Convergence analysis of pseudo-transient continuation.
\newblock {\em SIAM Journal on Numerical Analysis}, 35(2):508--523, 1998.

\end{thebibliography}

\appendix
\section{Solution of the nonlinear system of equations}
\subsection{Method of Continuation}\label{sec:implementation-moc}
The FEM problem formulated in \eqref{eq:fem-problem-vms} is a coupled nonlinear equation that result in a nonlinear algebraic system of equations as
\begin{align}\label{eq:nla-nonlinear-system}
\mvec{b^g}(\mvec{x^g})=\mvec{l^g},
\end{align}
where $ \mvec{b^g} $ and $ \mvec{l^g} $ are the numerical versions of $ \bformh $ and $ \lformh $ respectively; and $ \mvec{x^g} $ is the vector containing all the unknown discrete degrees of freedoms, i.e.,
\begin{align}
\mvec{x^g} = \left( u_{x1},u_{y1}, u_{z1}, p_1, u_{x2},u_{y2}, u_{z2}, p_2,...,u_{xN},u_{yN}, u_{zN}, p_N  \right)
\end{align}
where $ N $ is the total number of nodes in the discretization $ \mesh $. The algebraic problem in  \eqref{eq:nla-nonlinear-system} contain unknowns that span the full spatio-temporal domain $ \Omega $. When solving such equations using the Newton-Raphson scheme, an initial guess for $ \xvec^{\mvec{g}} $ is required. But if the initial guess is far from the exact solution, it is well known that the Newton-Raphson method might fail to converge to the exact solution.

Thus, to overcome this issue, we use a variant of the method of continuation, known as the pseudo-transient continuation (PTC)  \cite{coffey2003pseudotransient, kelley1998convergence}. We do this by embedding the fully coupled space-time system of equations \eqref{eq:nla-nonlinear-system} in an auxiliary evolution space $ \tau $ as follows:
\begin{align}\label{eq:nla-nonlinear-system-ptc}
\partialder{\mvec{x^g}}{\tau} + \mvec{b^g}(\mvec{x^g})=\mvec{l^g}.
\end{align}
If we compare \eqref{eq:nla-nonlinear-system-ptc} with  \eqref{eq:nla-nonlinear-system}, we can infer the role of the variable $ \tau $. In the ideal case, we want  \eqref{eq:nla-nonlinear-system-ptc} and \eqref{eq:nla-nonlinear-system} to yield the same solution for $ \mvec{x^g} $, thereby making $ \partialder{\mvec{x^g}}{\tau} = 0 $. This gives us an opportunity to design an iteration scheme similar to a time-marching algorithm that ``evolves'' the space-time solution $ \mvec{x^g} $  with respect to $ \tau $. Assume a discretization of the auxiliary variable $ \tau $ given by $ \tau^h = (\tau^0, \tau^1,...) $, where the elements in $ \tau^h $ are not necessarily equidistant. With this, we can approximate  \eqref{eq:nla-nonlinear-system-ptc} as
\begin{align}\label{eq:nla-nonlinear-system-ptc-beuler}
\frac{\mvec{x}^{\mvec{g}^{k+1}} - \mvec{x}^{\mvec{g}^{k}} }{\Delta \tau^{k+1}} + \mvec{b^g}(\mvec{x}^{\mvec{g}^{k+1}}) &=\mvec{l^g} \\
\label{eq:nla-nonlinear-system-ptc-beuler-iterations}
\implies \quad\quad\frac{\mvec{x}^{\mvec{g}^{k+1}}}{\Delta \tau^{k+1}} + \mvec{b^g}(\mvec{x}^{\mvec{g}^{k+1}}) &=\frac{\mvec{x}^{\mvec{g}^{k}} }{\Delta \tau^{k+1}} + \mvec{l^g},\quad k = 0,1,2,...,
\end{align}
where $ \Delta \tau^{k+1} = \tau^{k+1}-\tau^k $. With a sufficiently small choice of $ \Delta \tau^{k+1} $,  \eqref{eq:nla-nonlinear-system-ptc-beuler-iterations} is a contraction for $ \mvec{x^g} $, i.e., there exists a $ K^{\infty}\in\mathbb{N} $ such that $\|\mvec{x}^{\mvec{g}^k}- \mvec{x}^{\mvec{g}}_* \| \leq \epsilon$ for all $ k \geq K^{\infty} $. At $ k=0 $, we do not have a solution, rather, we just have an initial guess denoted by $ \mvec{x}^{\mvec{g}^{0}} $. But as we solve  \eqref{eq:nla-nonlinear-system-ptc-beuler-iterations} successively, each $ \mvec{x}^{\mvec{g}^k} $ gets closer and closer to the true solution $ \mvec{x}^{\mvec{g}}_* $.

\eqref{eq:nla-nonlinear-system-ptc-beuler-iterations} is still nonlinear in $ \mvec{x^g} $. Thus applying Newton-Raphson linearization to  \eqref{eq:nla-nonlinear-system-ptc-beuler-iterations}, we have:
\begin{align}\label{eq:nla-linearized-system-ptc-beuler}
\left[ \frac{\mvec{M}}{\Delta \tau^{k+1}} + \mvec{J}(\mvec{x}^{\mvec{g}^{k+1}}_r) \right]\cdot\delta \mvec{x}^{\mvec{g}^{k+1}}_r &= \mvec{z}^{\mvec{g}^{k+1}},\\
\label{eq:nla-newton-update-ptc-beuler}
\mvec{x}^{\mvec{g}^{k+1}}_{r+1} &= \mvec{x}^{\mvec{g}^{k+1}}_{r} + \delta \mvec{x}^{\mvec{g}^{k+1}}_{r},\quad r = 0,1,2,...
\end{align}
where $ \mvec{z}^{\mvec{g}^{k+1}}  = \left(\mvec{x}^{\mvec{g}^k}/\Delta \tau^{k+1} + \mvec{l^g} - \mvec{b^g}(\mvec{x}^{\mvec{g}^k}_r)\right)$. 

Once again, iteration \eqref{eq:nla-newton-update-ptc-beuler} is stopped when $ \norm{\delta\xvec^{\mvec{g}^{k+1}}_r} < \gamma_1$, and the backward Euler iterations \eqref{eq:nla-nonlinear-system-ptc-beuler-iterations} are stopped when $ \norm{\xvec^{\mvec{g}^{k+1}} - \xvec^{\mvec{g}^{k}}} < \gamma_2$, where $ \gamma_1>0 $ and $ \gamma_2>0 $ are preset tolerances of a small magnitude.

\section{Additional details on simulations}
All the simulations in this paper were run on the compute machines provided by the Nova computing cluster at the Iowa State University. These are Intel Xeon Platinum 8358 2.60GHz CPU (``Icelake'') machines with 64 cores on two sockets (32 cores per socket). Also, each node is equipped with 512 GB of memory (RAM).

A few of the 4D simulations were run on the Frontera cluster hosted by TACC. These are Intel Xeon Platinum 8280 2.7GHz CPU (``Cascade Lake'') machines with 56 cores on two sockets (28 cores per socket). These nodes have 192 GB of memory (RAM).

\subsection{Convergence analysis}
%The convergence studies in \secref{sec:mms-2d} were run on 1-2 compute nodes of Nova, and the ones in \secref{sec:mms-3d} were run on 5-10 compute nodes of Frontera.
The \petsc{} options for the linear basis functions are as follows:
\begin{lstlisting}
-ksp_type bcgs
-pc_type asm
-snes_rtol 1e-10
-snes_ksp_ew
\end{lstlisting}
For quadratic basis functions, the nonlinear relative tolerance was set to 
\begin{lstlisting}
-snes_rtol 1e-12
\end{lstlisting}
\subsection{Lid-driven Cavity}
%\subsubsection{$ Re=1000 $}
All simulations in this paper were done using the linear algebra solvers 
provided by $ \petsc $. Below are the command line options for the $ \petsc $ 
linear solvers. $ T $ refers to the physical time horizon of the space-time 
domain. 

The common options used in \emph{all} cases:
\begin{lstlisting}
	-ksp_rtol 1e-6 #default
	-snes_rtol 1e-8
	-snes_atol 1e-8
\end{lstlisting}
Case-by-case options are as follows:

%\begin{minipage}{0.999\textwidth}
%	\begin{minipage}{0.49\textwidth}
%		\begin{itemize}		
%		\end{itemize}
%	\end{minipage}
%	\begin{minipage}{0.49\textwidth}
%		\begin{itemize}		
%		\end{itemize}
%	\end{minipage}
%\end{minipage}

\begin{itemize}
	\item $ Re=1000 $, $ T = 180 $ 
	\\
	\begin{minipage}{0.999\textwidth}
		\begin{minipage}{0.49\textwidth}
			\begin{itemize}		
				\item \textsc{mesh}: $ 32^3 $				
				\begin{lstlisting}
				-ksp_type ibcgs
				-pc_type asm
				-snes_ksp_ew
				\end{lstlisting}
			\end{itemize}
		\end{minipage}
		\begin{minipage}{0.49\textwidth}
			\begin{itemize}		
				\item \textsc{mesh} : $ 128^3 $
				\begin{lstlisting}
				-ksp_type bcgsl
				-pc_type bjacobi #default
				\end{lstlisting}
			\end{itemize}
		\end{minipage}
	\end{minipage}
	\item $ Re=3200 $, $ T = 180 $ 
	\\
	\begin{minipage}{0.999\textwidth}
		\begin{minipage}{0.49\textwidth}
			\begin{itemize}	
				\item \textsc{mesh}: $ 32^3 $
				\begin{lstlisting}
				-ksp_type ibcgs
				-pc_type asm
    				-snes_ksp_ew
				\end{lstlisting}	
			\end{itemize}
		\end{minipage}
		\begin{minipage}{0.49\textwidth}
			\begin{itemize}	
				\item \textsc{mesh} : $ 128^3 $
				\begin{lstlisting}
				-ksp_type bcgsl
				-pc_type bjacobi #default
				\end{lstlisting}	
			\end{itemize}
		\end{minipage}
	\end{minipage}
	\item $ Re=5000 $, $ T = 360 $ 
	\\
	\begin{minipage}{0.999\textwidth}
		\begin{minipage}{0.49\textwidth}
			\begin{itemize}	
				\item \textsc{mesh}: $ 32^3 $
				\begin{lstlisting}
				-ksp_type ibcgs
				-pc_type asm		
				\end{lstlisting}	
			\end{itemize}
		\end{minipage}
		\begin{minipage}{0.49\textwidth}
			\begin{itemize}	
				\item \textsc{mesh} : $ 128^3 $
				\begin{lstlisting}
				-ksp_type ibcgs
				-pc_type asm
				-snes_ksp_ew
				\end{lstlisting}	
			\end{itemize}
		\end{minipage}
	\end{minipage}
	\item $ Re=7500 $, $ T = 720 $ 
	\\
	\begin{minipage}{0.999\textwidth}
		\begin{minipage}{0.49\textwidth}
			\begin{itemize}	
				\item \textsc{mesh}: $ 32^3 $
				\begin{lstlisting}
				-ksp_type bcgsl
				-pc_type bjacobi #default
				\end{lstlisting}	
			\end{itemize}
		\end{minipage}
		\begin{minipage}{0.49\textwidth}
			\begin{itemize}	
				\item \textsc{mesh} : $ 128^3 $
				\begin{lstlisting}
				-ksp_type ibcgs
				-pc_type asm
				-snes_ksp_ew
				\end{lstlisting}	
			\end{itemize}
		\end{minipage}
	\end{minipage}
%	\item $ Re=10,000 $, $ T = 10^3 $ 
%	\\
%	\begin{minipage}{0.999\textwidth}
%		\begin{minipage}{0.49\textwidth}
%			\begin{itemize}	
%				\item \textsc{mesh}: $ 32^3 $
%				\begin{lstlisting}
%				-ksp_type bcgsl
%				-pc_type bjacobi #default	
%				\end{lstlisting}	
%			\end{itemize}
%		\end{minipage}
%		\begin{minipage}{0.49\textwidth}
%			\begin{itemize}
%				\item \textsc{mesh} : $ 128^3 $
%				\begin{lstlisting}
%				-- (not run)
%				\end{lstlisting}		
%			\end{itemize}
%		\end{minipage}
%	\end{minipage}
\end{itemize}

%\begin{table*}[t!]
%	\setlength\extrarowheight{2.5pt}
%	\caption{Baka}
%	\vspace{-0.15in}
%	\begin{center}
%		\label{Table:azure-costs}
%		\resizebox{0.7\linewidth}{!}{
%			\begin{tabular}{l  l }
%				\textsc{mesh}: $ 32^3 $     &   \textsc{mesh}: $ 128^3 $ \\
%				\multirow{3}{*}{MG}    &  32 \\
%				& 64 \\
%				& 128 \\				
%			\end{tabular}
%		}
%		\vspace{-0.1in}
%	\end{center}
%\end{table*}

%\begin{lstlisting}
%-ns_ksp_type bcgs
%-ns_pc_type asm
%-ch_ksp_type bcgs
%-ch_pc_type asm
%-ns_snes_monitor 
%-ns_snes_converged_reason 
%-ns_ksp_converged_reason
%-ch_snes_monitor
%-ch_snes_converged_reason
%-ch_ksp_converged_reason
%\end{lstlisting}

\subsection{Flow past a cylinder}
In \secref{sec:fpc}, the mesh for $ T = 100 $ with $ h_t = 0.1 $ has 4164160 nodes with 4160 nodes per time slice ($ 4160 \times (\nicefrac{100}{0.1}+1) = 4164160 $). \petsc{} options for these simulations:
\begin{lstlisting}
-snes_max_it 1000
-snes_rtol 1e-6
-ksp_rtol 1e-8
-ksp_type fgmres
-pc_type gamg
-mg_levels_ksp_type bcgs
-mg_levels_pc_type asm
-mg_levels_ksp_max_it 20
\end{lstlisting}
%# mg_levels_sub_pc_type lu

\end{document}